%% file: FinalClassification.tex
\documentclass[12pt,oneside]{amsart}
\usepackage{amssymb, amsmath, amsthm}

\usepackage{epsfig}
\usepackage{epstopdf}
\usepackage{graphicx}
\usepackage{color}
\usepackage{mathptmx}
\usepackage{enumerate}

\theoremstyle{plain}
\newtheorem{theorem}{Theorem}[section]

\newtheorem*{theorem*}{Theorem}
\newtheorem{proposition}[theorem]{Proposition}
\newtheorem{corollary}[theorem]{Corollary}
\newtheorem{lemma}[theorem]{Lemma}

\theoremstyle{definition}

\newtheorem*{remark}{Remark}

\newcommand{\nil}{\varnothing}
\newcommand{\tild}{\widetilde}
\newcommand{\wihat}{\widehat}
\newcommand{\defn}[1]{\textbf{#1}}
\newcommand{\boundary}{\partial}

\newcommand{\mc}[1]{\mathcal{#1}}
\newcommand{\mb}[1]{\mathbf{#1}}
\newcommand{\ob}[1]{\overline{#1}}
\newcommand{\cls}{\operatorname{cl}} 
\newcommand{\genus}{\operatorname{genus}} 
\newcommand{\inter}[1]{\mathring{#1}}

      \makeatletter
      \def\@setcopyright{}
      \def\serieslogo@{}
      \makeatother
      
\begin{document}

   \title[]{Heegaard surfaces for certain graphs in compressionbodies}
   \author{Scott A Taylor and Maggy Tomova}
   \thanks{The second author was supported by a grant from the National Science Foundation during this research.}


   \date{\today}


\begin{abstract}
Let $M$ be a compressionbody containing a properly embedded graph $T$ (with at least one edge) such that $\boundary_+ M - T$ is parallel to $T \cup \boundary_- M$. We extend methods of Hayashi and Shimokawa to show that if $H$ is a bridge surface for $T$ then one of the following occurs:
\begin{itemize}
\item $H$ is stabilized, boundary stabilized, or perturbed
\item $T$ contains a removable path
\item $M$ is a trivial compressionbody and $H - T$ is properly isotopic in $M - T$ to $\boundary_+ M - T$.
\end{itemize}

The results of this paper are used in later work to show that if a bridge surface for a graph in a 3--manifold is c-weakly reducible then either a degenerate situation occurs or the exterior of the graph contains an essential meridional surface.
\end{abstract}

\maketitle

\section{Introduction}
Hayashi and Shimokawa \cite{HS3} created a version of thin position which combines the notion of thin position for a knot or link in $S^3$ \cite{G} with thin position for a 3--manifold \cite{ST3}. They prove an analogue of a famous theorem of Casson and Gordon \cite{CG} for their version of thin position. Informally: if a bridge surface for a link in a 3--manifold can be untelescoped then either the bridge position was not ``minimal'' or there is an essential closed or meridional surface in the exterior of the link. The arguments in \cite{HS3} rely heavily on other work \cite{HS1, HS2} which classifies Heegaard splittings of certain 1--manifolds in certain 3--manifolds. These classification theorems are used to understand what happens if every component of a thin surface is boundary parallel in the exterior of the 1--manifold.

In \cite{ST2}, Scharlemann and Thompson defined thin position for a graph in the 3--sphere (see also \cite{GST,S}). Applications have included a new proof of the classification of Heegaard splittings of $S^3$ and a theorem about levelling unknotting tunnels of tunnel number one knots and links. Li \cite{L} used thin position for graphs to show that if for a graph in the 3--sphere thin position is not equal to bridge position then there is an essential meridional or almost meridional planar surface in its exterior.

Although the definition for thin position of a graph in $S^3$ can be easily adapted to define thin position for a graph in any 3--manifold, this definition has not been much used. In \cite{TT}, we extend Hayashi and Shimokawa's definition of thin position for a link in a 3--manifold to graphs in a 3--manifold. We use this definition to prove a Casson-Gordon type theorem which says (informally) that either thin position for a graph in a 3--manifold is bridge position, or there exists an essential meridional or closed surface in its complement, or one or more various degenerate situations occur. This theorem generalizes both \cite{HS3} and \cite{T}.

Just as Hayashi and Shimokawa's work in \cite{HS3} rests on the classification results in \cite{HS1, HS2}, all of which are quite technical, so our result in \cite{TT} rests on the classification results of this paper. One odd feature of this paper is that (in some situations) to classify splittings for a certain type of graph in a certain type of 3--manifold it is helpful to make the graph more complicated by adding certain types of edges. This is, however, in the spirit of \cite{HS1} where Heegaard splittings for a (3--manifold, graph) pair are first defined. 

\section{Definitions}

\subsection{Trivially embedded graphs in compressionbodies}
Suppose that $T$ is a finite graph. We will usually assume that there are no vertices of valence 2; exceptions will be explicitly mentioned. We allow $T$ to contain components homeomorphic to $S^1$. Let $\boundary T$ denote the vertices of valence 1. The vertices of $T$ which are not in $\boundary T$ are called the \defn{interior vertices} of $T$. We say that $T$ is \defn{properly embedded} in a 3--manifold $M$ if $T \cap \boundary M = \boundary T$. A \defn{pod} is a finite tree having at least 3 edges, 0 or 1 of which is a distinguished edge called a \defn{handle}. The edges of the pod which are not the handle are called the \defn{legs} of the pod. If $T' \subset T$ is a subgraph of a graph $T$, then we say that $\cls(T - T')$ is obtained by \defn{removing} $T'$ from $T$.

Let $I = [0,1]$. Let $F$ be a closed, possibly disconnected, surface and let $T$ be the disjoint union of finitely many edges and pods properly embedded in $F \times I$. An edge $e$ of $T$ is \defn{vertical} if it is isotopic to $\{\text{point}\} \times I$. An edge of $T$ is a \defn{bridge edge} if there exists an embedded disc $D$ so that $\boundary D$ is the endpoint union of $e$ and an arc in $F \times \{1\}$. The disc $D$ is called a \defn{bridge disc} for the arc $e$. A pod $p$ of $T$ is a \defn{bridge pod} if $p$ has no handle and is contained in a disc $D$ such that $\boundary D \subset F \times \{1\}$. The disc $D$ is called a \defn{pod disc}. The closures of the components of $D - p$ are \defn{bridge discs} for $p$. More generally a disc $E$ such that $\boundary E$ is the endpoint union of an arc traversing exactly two edges of $p$ and an arc in $F \times \{1\}$ is a \defn{bridge disc} for $p$. It is not difficult to see that if every pair of edges of $p$ has a bridge disc then there exists a pod disc for $p$. Finally, suppose that $p$ is a pod with handle $h$. Assume that there is a pod disc for $p - h$. Notice that compressing $F \times I$ along the boundary of a regular neighborhood of a pod disc creates a 3--manifold with one component homeomorphic to $F \times I$. The pod $p$ is a \defn{vertical pod} if $h$ is a vertical edge in that component. 

Suppose that $T$ is the disjoint union of vertical edges, bridge edges, bridge pods, and vertical pods such that the bridge edges, bridge pods, and vertical pods have pairwise disjoint bridge discs and pod discs. Suppose also that these bridge discs and pod discs are disjoint, except at the endpoints of the handles, from the vertical edges and handles. Assume also that the vertical edges and handles can all be simultaneously isotoped to be $\{\text{points}\} \times I$ in $F \times I$. Then $T$ is \defn{trivially embedded} in $F \times I$. If $T$ is properly embedded in a 3--ball $B^3$, $T$ is \defn{trivially embedded} if it is the disjoint union of bridge edges and bridge pods which have pairwise disjoint bridge discs and pod discs.

A \defn{compressionbod}y $C$ is formed from $F \times I$ by attaching a finite number of 1--handles to $F \times \{1\}$. Let $\boundary_- C = F \times \{0\}$ and let $\boundary_+ C = \boundary C - \boundary_- C$. A collection of compressing discs $\Delta$ for $\boundary_+ C$ is a \defn{complete collection} if boundary-reducing $C$ using $\Delta$ produces a manifold homeomorphic to $\boundary_- C \times I$. We consider a handlebody to be a compressionbody with $\boundary_- C = \nil$. In this case, a complete collection of discs is a collection of compressing discs for $\boundary_+ C = \boundary C$ which cut $C$ into a 3--ball. We require compressionbodies to be connected and nonempty. 

Let $T$ be a finite graph properly embedded in a compressionbody $C$. We say that $T$ is \defn{trivially embedded} in $C$ if there exists a complete collection of discs $\Delta$ for $C$ disjoint from $T$ such that if $C'$ is a component of $C$ reduced by $\Delta$, then $T \cap C'$ is trivially embedded in $C'$. Figure \ref{Fig: TrivialGraph} is a schematic depiction of a trivially embedded graph in a compressionbody. Let $\tild{C}$ denote the compressionbody obtained from $C$ by capping off with 3--balls any 2--sphere components of $\boundary_- C$ which are disjoint from $T$.

\begin{center}
\begin{figure}
\scalebox{0.5}{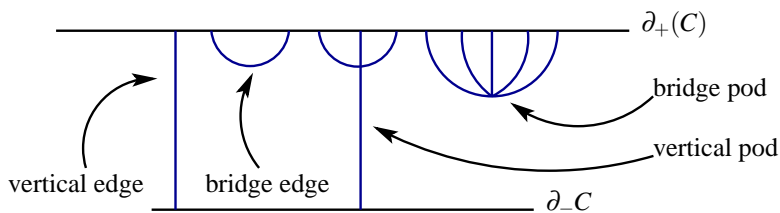}
\caption{A trivially embedded graph in a compressionbody}
\label{Fig: TrivialGraph}
\end{figure}
\end{center}

\begin{remark}
Because we allow a trivially embedded graph in a compressionbody to contain pods with handles, our definition of ``trivially embedded'' is more general than that of Hayashi and Shimokawa \cite{HS1}. (Allowing pod handles is advantageous if a Heegaard surface for a graph in a 3--manifold is perturbed and you wish to unperturb it. See below for the definition of ``perturbed''. See \cite{TT} for details on unperturbing Heegaard splittings.) 
\end{remark}

A \defn{spine} for a compressionbody $C$ with trivially embedded graph $T$ is a 2 or less dimensional complex $Q$ embedded in $C$ so that:
\begin{enumerate}
\item $\boundary_+ C \cap Q = \nil$.
\item $\boundary_- C \cap Q$ is contained in the valence 1 vertices of $Q$ not contained in a 2--cell of $Q$.
\item Every bridge arc of $T$ intersects $Q$ at precisely one vertex of $Q$.
\item If $\tau$ is a bridge pod, then $\tau \cap Q$ is a vertex of both $Q$ and $\tau$.
\item If $\tau$ is a vertical pod, then $\tau \cap Q$ is the handle of $\tau$.
\item All vertical arcs of $T$ are disjoint from $Q$.
\item All valence one vertices of $Q$ lie in $T \cup \boundary_- C$.
\item $C$ collapses to $\boundary_- C \cup Q \cup T$.
\end{enumerate}

We let $\boundary_1 Q$ denote those vertices of $Q$ which lie on $T$. Let $Q^2$ be the union of the 2--simplices of $Q$ and let $Q^1$ be the union of 1--simplices of $Q$ not contained in $Q^2$. If $Q = Q^1$, then we say that the spine $Q$ is \defn{elementary}. In the arguments which involve a spine $Q$, we will always begin by assuming that $Q$ is elementary, but deformations in the course of the argument may convert $Q$ into a non-elementary spine. If $C$ is a handlebody, we will never convert $Q$ into a non-elementary spine.

If $C = B^3$ and $T = \nil$, then a single point in the interior of $C$ is a spine for $C$. If $C = \boundary_- M \times I$ and every arc of $T$ is a vertical arc, then $\nil$ is a spine for $M$.

In general, for a graph $T$ trivially embedded in $C$, it is straightforward to construct a spine for $(C,T)$. Figure \ref{Fig: Spine} depicts a spine for a genus 2 handlebody containing a bridge edge and a bridge pod.

\begin{center}
\begin{figure}[ht]
\scalebox{0.5}{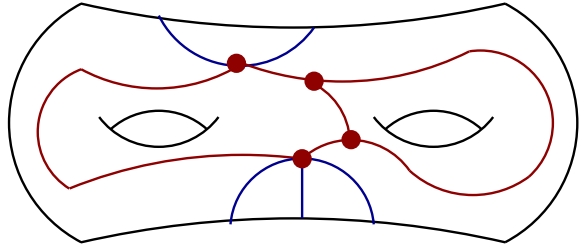}
\caption{A spine for a genus 2 handlebody with a bridge edge and a bridge pod}
\label{Fig: Spine}
\end{figure}
\end{center}

\subsection{Heegaard surfaces and compressions}
Let $M$ be a compact connected orientable 3--manifold. A \defn{Heegaard surface} $H \subset M$ is an orientable surface such that $\cls(M - H)$ consists of two distinct compressionbodies $C_1$ and $C_2$ with $H = \boundary_+ C_1 = \boundary_+ C_2$. We will be studying Heegaard surfaces in a compressionbody $M$.  We always assume that $\boundary_+ M \subset C_2$. Note the possibly confusing notation: $\boundary_+ M$ is a component of $\boundary_- C_2$.

Suppose that $T \subset M$ is a properly embedded finite graph. We say that $H$ is a \defn{Heegaard surface} for $(M,T)$ if $T_i = T \cap C_i$ is trivially embedded in $C_i$ for $i \in \{1,2\}$. We also will say that $T$ is in \defn{bridge position} with respect to $H$ and that $H$ is a \defn{splitting} of $(M,T)$. Notice that this definition of bridge position generalizes that in \cite{HS1} since we allow vertical pods in a compressionbody.

Suppose that $F \subset M$ is a surface such that $\boundary F \subset (\boundary M \cup T)$. Then $F$ is \defn{$T$--compressible}, if there exists a compressing disc for $F - T$ in $M - T$. If $F$ is not $T$--compressible, it is \defn{$T$--incompressible}. $F$ is \defn{$T$--$\boundary$--compressible} if there exists a disc $D \subset M - T$ with interior disjoint from $F$ such that $\boundary D$ is the endpoint union of an arc $\gamma$ in $F$ and an arc $\delta$ in $\boundary M$. We require that $\gamma$ not be parallel in $F - T$ to an arc of $\boundary F - T$. If $F$ is not $T$--$\boundary$--compressible, it is \defn{$T$--$\boundary$--incompressible}.

\subsection{Stabilization, Cancellation, Perturbation, and Removable Edges}
A Heegaard surface $H$ for $(M,T)$ is \defn{stabilized} if there exist discs $D_1 \subset C_1$ and $D_2 \subset C_2$ which are compressing discs for $H$,  have boundaries intersecting transversally in a single point, and are disjoint from $T$. If a Heegaard surface $H$ is stabilized, there is another Heegaard surface $H'$ with $\genus(H') = \genus(H) - 1$ and with $|T \cap H| = |T \cap H'|$. 

Suppose that $D_1 \subset C_1$ and $D_2 \subset C_2$ are bridge discs for $T_1$ and $T_2$ respectively such that the arcs $\boundary D_1 \cap H$ and $\boundary D_2 \cap H$ have disjoint interiors but share at least one endpoint. If such discs exist, $H$ is \defn{cancellable}. If, in addition, $\boundary D_1 \cap H$ and $\boundary D_2 \cap H$ share only a single endpoint then $H$ is \defn{perturbed}. Hayashi and Shimokawa's notion of ``strongly cancellable'' is the same as perturbed. The discs $\{D_1, D_2\}$ will be known as either a \defn{cancelling pair} or \defn{perturbing pair} of discs. If $\boundary D_1 \cup \boundary D_2$ does not contain a vertex of $T$ then $(\boundary D_1 \cup \boundary D_2) \cap T$ lies in a single edge $e$ of $T$. We say that the edge $e$ is \defn{cancellable} or \defn{perturbed} (corresponding to whether $\{D_1,D_2\}$ is a cancelling or perturbing pair of discs). See \cite{TT}, for situations where a splitting can be unperturbed.

Suppose that $F$ is a closed connected surface and that $V = F \times I$. A \defn{type I} Heegaard surface for $V$ is a Heegaard surface which separates the components of $\boundary V$. A \defn{type II} Heegaard surface is a Heegaard surface for $V$ which does not separate the components of $\boundary V$. Type I and Type II Heegaard splittings were classified in \cite{ST}.

Let $F$ be a component of $\boundary_- C_1 \subset \boundary M$ and let $T'$ be a collection of vertical arcs in $F \times [-1,0]$. Let $H'$ be a minimal genus type II Heegaard surface for $F \times [-1,0]$ which intersects each arc in $T'$ exactly twice. $H'$ can be formed by tubing two parallel copies of $F$ along a vertical arc not in $T'$. Assume that $T' \cap (F \times \{0\}) = T \cap F$. (Recall that $F \times[0,1] \subset M$.) We can form a Heegaard surface $H''$ for $M \cup (F \times [-1,0])$ by \defn{amalgamating} $H$ and $H'$. This is simply the usual notion of amalgamation of Heegaard splittings (see \cite{Sc}). In fact, $H''$ is a Heegaard surface for $(M \cup (F \times [-1,0]), T \cup T')$. Since $(M \cup (F \times [-1,0]), T \cup T')$ is homeomorphic to $(M,T)$, we may consider $H''$ to be a Heegaard surface for $(M,T)$. $H''$ is called a \defn{boundary stabilization} of $H$. See Figure \ref{Fig: Amalgamation}.

\begin{center}
\begin{figure}[ht]
\scalebox{0.35}{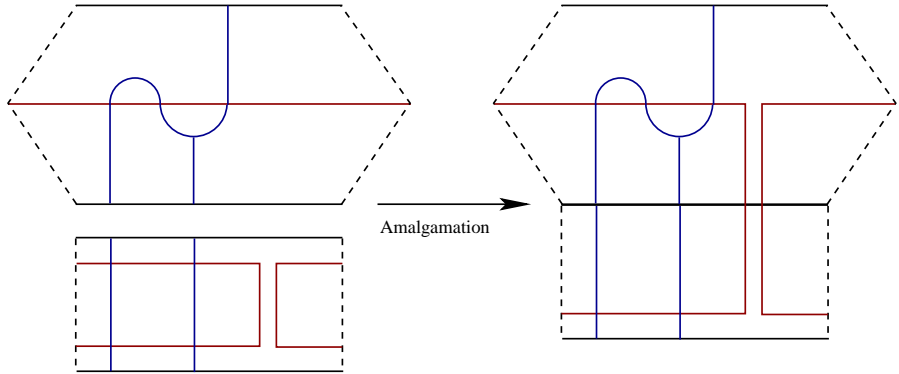}
\caption{Boundary stabilizing a Heegaard surface}
\label{Fig: Amalgamation}
\end{figure}
\end{center}

A \defn{monotonic interior edge} is an edge $e$ of $T$ with no endpoint on $\boundary M$ and which intersects $H$ in a single point. 

Suppose that $\zeta \subset T$ is a 1--manifold which is the the union of edges in $T$ (possibly a closed loop containing zero or one vertices of $T$). We say that $\zeta$ is a \defn{removable path} if the following hold:
\begin{enumerate}
\item[(RP 1)] Either the endpoints of $\zeta$ lie in $\boundary M$ or $\zeta$ is a cycle in $T$.

\item[(RP 2)] $\zeta$ intersects $H$ exactly twice 

\item[(RP 3)] If $\zeta$ is a cycle, there exist a cancelling pair of discs $\{D_1,D_2\}$ for $\zeta$ with $D_j \subset C_j$. Furthermore there exists a compressing disc $E$ for $H$ such that $|E \cap T| = 1$ and if $E \subset C_j$ then $|\boundary E \cap \boundary D_{j+1}| = 1$ (indices run mod 2) and $E$ is otherwise disjoint from a complete collection of bridge discs for $T - H$ containing $D_1 \cup D_2$. 

\item[(RP4)] If the endpoints of $\zeta$ lie on $\boundary M$, there exists a bridge disc $D$ for the bridge arc component of $\zeta - H$ such that $D - T$ is disjoint from a complete collection of bridge discs $\Delta$ for $T - H$. Furthermore, there exists a compressing disc $E$ for $H$ on the opposite side of $H$ from $T$ such that $|E \cap D| = 1$ and $E$ is disjoint from $\Delta$.
\end{enumerate}

If $\zeta$ is a removable path, then a slight isotopy of $\rho$ which does not move the rest of $T$, moves $\zeta$ to be a subset of a spine of one of the compressionbodies $M - H$. (See \cite[Lemma 3.3]{STo}.) Also, note that by (RP2), $\zeta$ can contain at most 3 vertices of $T$ (and that only if it contains two pod handles).

\section{The Main Result}

Let $T$ be properly embedded graph in a compact connected orientable 3--manifold $M$. For the remainder, we assume that there is a component $\ob{F}$ of $\boundary M$ so that $\ob{F}_T = \ob{F} - \inter{\eta}(T)$ is isotopic (rel $\boundary$) in $M_T = M - \inter{\eta}(T)$ to $\boundary M_T - F_T$. Notice that this implies that $M$ is a compressionbody with $\ob{F} = \boundary_+ M$. Henceforth, we write $\boundary_+ M$ instead of $\ob{F}$. 

\begin{theorem}[Main Theorem]\label{Main Theorem}
Suppose that $\boundary_+ M - \inter{\eta}(T)$ is isotopic to the frontier of a regular neighborhood of $\boundary_- M \cup T$. Let $H$ be a Heegaard surface for $(M,T)$. Assume that $T$ contains an edge. Then one of the following occurs:
\begin{enumerate}
\item $H$ is stabilized
\item $H$ is boundary stabilized
\item $H$ is perturbed
\item $T$ contains a removable path disjoint from $\boundary_+ M$. 
\item $M$ is a 3--ball. $T$ is a tree with a single interior vertex. $H - \inter{\eta}(T)$ is parallel to $\boundary M - \inter{\eta}(T)$ in $M - \inter{\eta}(T)$. 
\item $M = \boundary_- M \times I$ and $H- \inter{\eta}(T)$ is isotopic in $M - \inter{\eta}(T)$ to $\boundary_+ M - \inter{\eta}(T)$.
\end{enumerate}
\end{theorem}

The remainder of the paper is devoted to proving the Main Theorem.

We divide $T$ into the union of 3 subgraphs: $T_0$, $T_v$, and $T_s$ as follows. 

Since $\boundary_+ M - \inter{\eta}(T)$ is parallel to $\boundary_- M \cup \boundary \eta(T)$, $T$ contains an elementary spine $T_0$ for $M$. Let $T_v$ denote the components of $T - T_0$ which join $\boundary_- M$ to $\boundary_+ M$. Each component of $T_v$ is a tree with one valence one vertex on $\boundary_- M$. Let $T_s$ denote $T - (T_v \cup T_0)$. Each component of $T_s$ is a tree which joins $T_0$ to $\boundary_+ M$. A component of $T_s$ is called a \defn{spoke}. 

Notice that if $M = B^3$, then $T = T_s$ is a tree. In this case, we let $T_0$ be a vertex of $T$, or if $T$ is a single edge, a vertex of valence 2 in the interior of the edge.

Figure \ref{Fig: GraphT} shows an example with $T$ a $\theta$-graph in a genus two handlebody and every spoke a single edge. The graph $T_0$ is drawn with a thicker line.

\begin{center}
\begin{figure}[ht]
\scalebox{0.5}{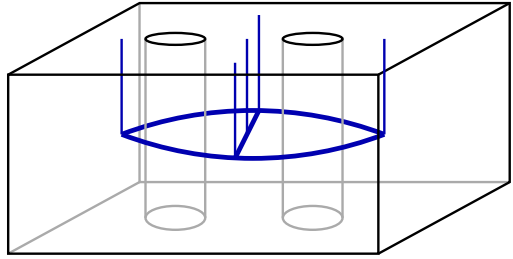}
\caption{An example of the graph $T$ in a genus two handlebody}
\label{Fig: GraphT}
\end{figure}
\end{center}

Figure \ref{Fig: Cbdy With Graph} schematically depicts a compressionbody $M$ with $\boundary_- M$ a torus and $\boundary_+ M$ a genus 2 surface. In the picture, opposite sides of the cube should be glued together. $T_0$ consists of a single edge with both endpoints on $\boundary_- M$. $T_s$ is a single edge joining $T_0$ to $\boundary_+ M$ and $T_v$ consists of a tree with four interior vertices. It shows up twice in the diagram, since $T^2 \times I$ has been cut open into a $\text{square} \times I$.

\begin{center}
\begin{figure}[ht]
\scalebox{0.5}{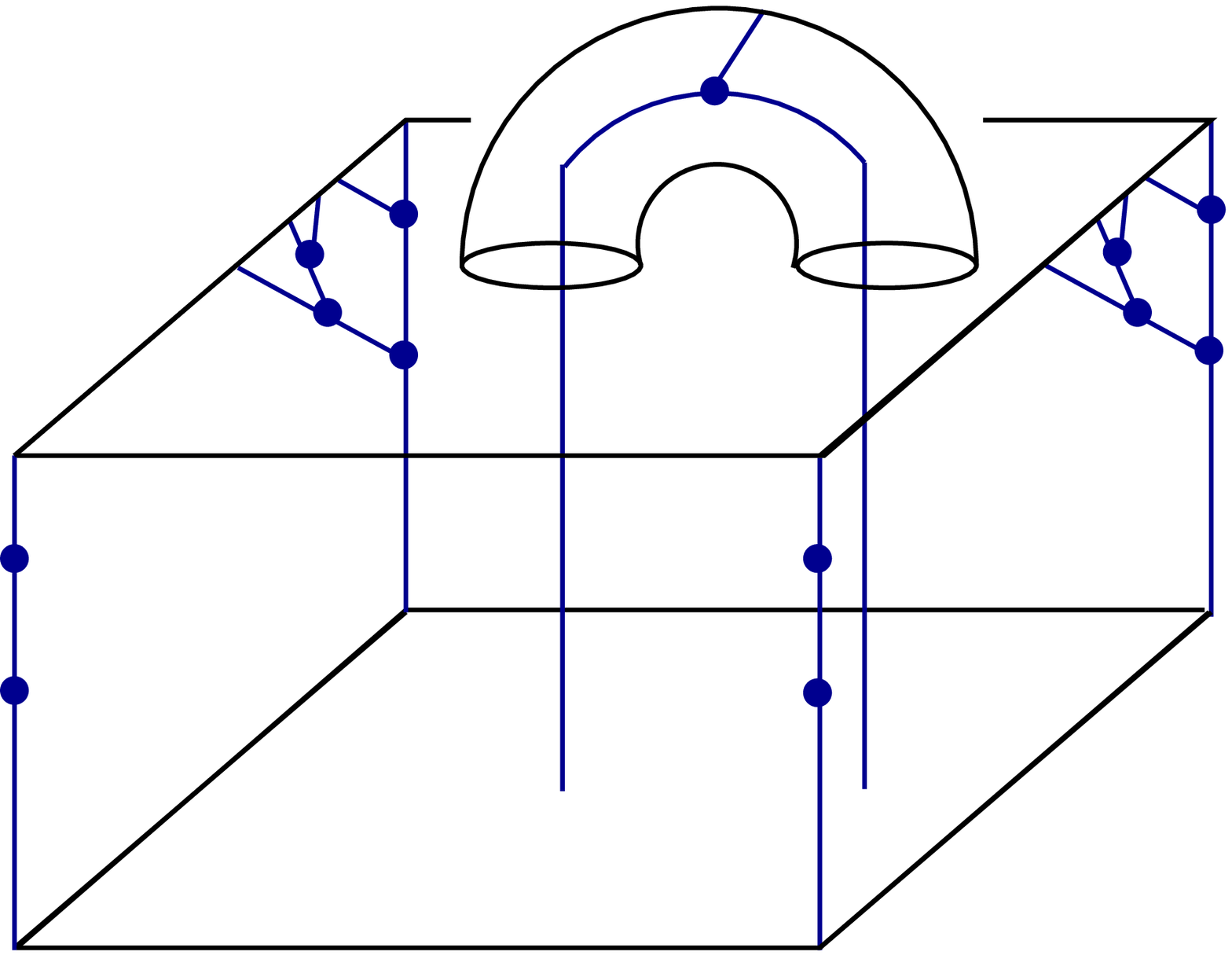}
\caption{An example of the graph $T$ in a compressionbody $M$ with $\boundary_- M$ a torus and $\boundary_+ M$ a genus two surface. The vertical sides of the cube should have opposite faces identified.}
\label{Fig: Cbdy With Graph}
\end{figure}
\end{center}

Figure \ref{Fig: Cbdy With Graph 2} shows the same compressionbody $M$, but with a different graph $T$. In this case, $T_0$ consists of two edges, one of which has a single endpoint on $\boundary_- M$. $T_s$ consists of two edges. One edge joins the interior vertex of $T_0$ to $\boundary_+ M$ and the other joins the interior of the edge which forms the loop in $T_0$ to $\boundary_+ M$. $T_v$ is the same as in the previous example. 

\begin{center}
\begin{figure}[ht]
\scalebox{0.5}{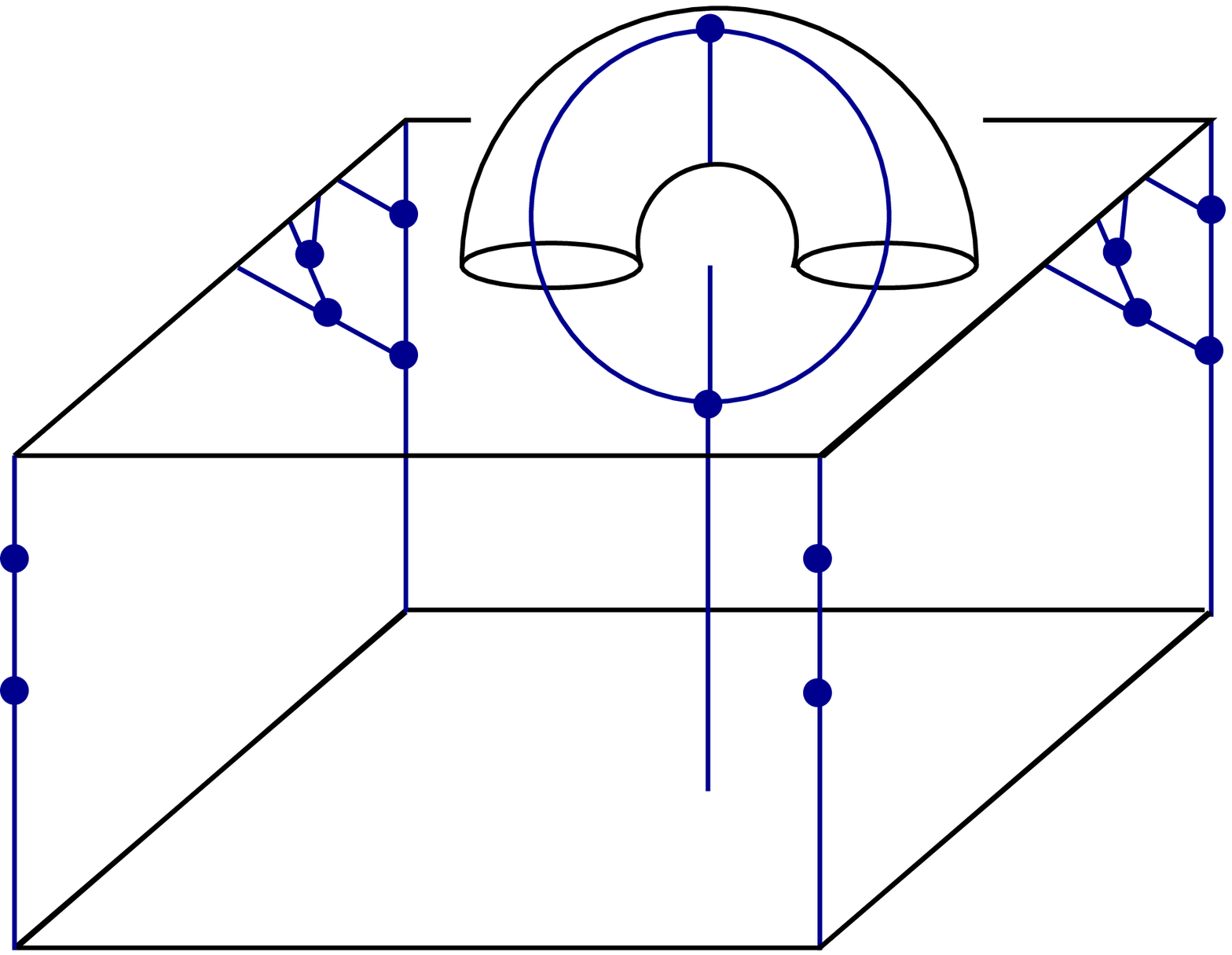}
\caption{An example of the graph $T$ in a compressionbody $M$ with $\boundary_- M$ a torus and $\boundary_+ M$ a genus two surface. The vertical sides of the cube should have opposite faces identified. }
\label{Fig: Cbdy With Graph 2}
\end{figure}
\end{center}

\section{Preliminary Lemmas}

\subsection{Resolving pod handles}

Suppose that $\tau$ is a trivially embedded graph in a compressionbody $C$. Although we allow $\tau$ to contain pod handles, we begin by showing that, for the purposes of proving the theorem, certain pod handles can be eliminated.

Suppose that $h$ is the handle of a pod $p$. Let $\rho$ be a regular neighborhood of $h$ in $C$ and notice that $\boundary \rho$ intersects each pod leg of $p$ once. For each pod leg, $\lambda$, choose a path $h_\lambda \subset \boundary \rho$ from $\lambda \cap \boundary\rho$ to $\boundary_- C$ so that $h_\lambda$ is parallel to $h$ and so that the collection $\{h_\lambda\}$ for legs $\lambda$ of $p$ is pairwise disjoint. Let $\tau'$ be the result of removing $\tau \cap \rho$ from $\tau$ and adding the union of the $h_\lambda$. Notice that $\tau'$ is trivially embedded in $C$. We say that $(C,\tau')$ is obtained by \defn{resolving} the pod handle $h$. See Figure \ref{Fig: Resolving Pod Handles}.

\begin{center}
\begin{figure}[ht]
\scalebox{0.5}{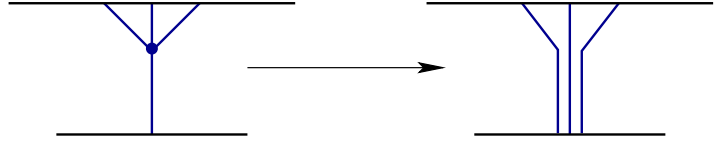}
\caption{Resolving a pod handle in a compressionbody.}
\label{Fig: Resolving Pod Handles}
\end{figure}
\end{center}

If $T$ is a graph properly embedded in $M$ and if $H$ is a Heegaard surface for $(M,T)$, then we can obtain a new graph $T' \subset M$ in $M$ by resolving one or more of the pod handles in either or both compressionbodies of $\cls(M - H)$. $H$ is still a Heegaard surface for $(M,T')$. 

\begin{lemma}\label{Lem: Resolving Pod Handles 1}
Suppose that $(M,T')$ is obtained from $(M,T)$ by resolving a pod handle of $T$ in $M - H$ which is adjacent to $\boundary_- M$. If $(M,T)$ satisfies the hypotheses of Theorem \ref{Main Theorem}, then so does $(M,T')$.
\end{lemma}

\begin{proof}
Let $h$ be a handle for a pod $\tau$ in $M - H$ which is adjacent to $\boundary_- M$. We can think about resolving it to a collection of vertical arcs $\tau'$ by shrinking $h$ to $\boundary_- M$. Combining this shrinking with the isotopy of $\boundary_+ M - \inter{\eta}(T)$ to $\boundary_- M \cup T$ gives an isotopy of $\boundary_+ M - \inter{\eta}(T')$ to $\boundary_- M \cup T'$.
\end{proof}

\begin{lemma}\label{Lem: Resolving Pod Handles 2}
Suppose that $(M,T)$ satisfies the hypotheses of Theorem \ref{Main Theorem} and that $(M,T')$ is obtained from $(M,T)$ by resolving a pod handle of $T$ in $M - H$ which is adjacent to $\boundary_- M$. If $(M,T')$ satisfies the conclusion of Theorem \ref{Main Theorem}, then so does $(M,T)$.
\end{lemma}
\begin{proof}
Let $h$ be the handle of the pod $\tau \subset T$ which is resolved. Suppose that $(M,T)$ satisfies the hypotheses of Theorem \ref{Main Theorem}, and that $(M,T')$ satisfies the conclusion of Theorem \ref{Main Theorem}. Let $\rho = \eta(h)$ and let $h_\lambda$ be the arcs from the definition of handle resolution.

\begin{enumerate}
\setlength{\parskip}{6.6pt}
 
\item Suppose that $H$ is stabilized as a splitting of $(M,T')$. Let $S$ be a sphere in $M$ disjoint from $T'$ intersecting $H$ in a single simple closed curve. Out of all such spheres, we may assume that $S$ has been chosen to minimize $|S \cap \boundary \rho|$. Since $\boundary \rho - \cup h_\lambda$ is a disc, an innermost disc argument shows that $S$ is disjoint from $\rho$. Hence, $H$ is stabilized as a splitting of $(M,T)$.

\item Suppose that $H$ is boundary-stabilized as a splitting of $(M,T')$. If the boundary stabilization is along a component of $\boundary_- M$ not adjacent to a resolved pod handle, then the splitting of $(M,T)$ is boundary stabilized.

Suppose, therefore,  that $h$ is adjacent to the component $F$ of $\boundary_- M$ along which $H$ is boundary stabilized as a splitting of $(M,T')$. Let $\psi$ be the vertical arc (not in $T_v$) along which the boundary stabilization was performed. An innermost disc/outermost arc argument shows that the compressing disc for $H$ which is a meridian of $\eta(\psi)$ is disjoint from $\rho = \eta(h)$. Thus, $\psi$ is disjoint from $\rho$.

Consider a square $V$ in $M$, with one edge of its boundary on $F$, one edge on $H$, one edge on $\psi$, and one edge on an edge $e$ of $T'$ containing an $h_\lambda$. We may arrange for the interior of $V$ to intersect $H$ in a single arc and for $V$ to contain an edge $e' \neq e$ which contains an $h_{\lambda'}$. To reconstruct $\tau$ from $e$, $e'$, and possibly other vertical arcs, $h_{\lambda}$ and $h_{\lambda'}$ (and possibly other arcs) are merged into a single arc. We may perform this merger within $V$. Suppose that $C_i$ contains $F$ and that $C_j$ is the other compressionbody of $M - H$. The intersection of the component of $V - T'$ with $C_j$ is a bridge disc for a component of $e' \cap C_j$. This bridge disc intersects a bridge disc for the pod $\tau$ in $V$ in a single endpoint, showing that $H$ is perturbed as a splitting of $(M,T)$. See Figure \ref{Fig: PodHandlesBdyStab}. 

\begin{center}
\begin{figure}[ht]
\scalebox{0.5}{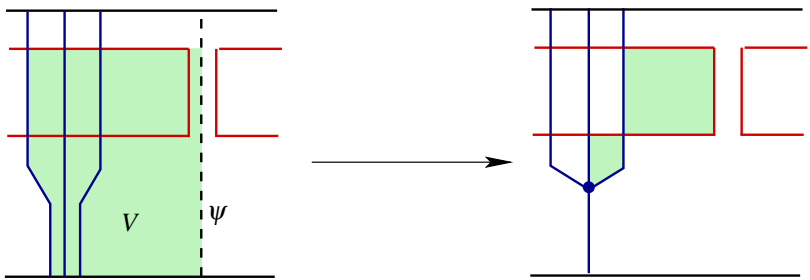}
\caption{Undoing a handle resolution may result in a perturbation.}
\label{Fig: PodHandlesBdyStab}
\end{figure}
\end{center}

\item Suppose that $H$ is perturbed as a splitting of $(M,T')$. An argument similar to the case when $H$ is stabilized as a splitting of $(M,T')$ shows that $H$ is also perturbed as a splitting of $(M,T)$.

\item Suppose that $T'$ contains a removable path $\zeta$ disjoint from $\boundary_+ M$. If the endpoints of $\zeta$ are identified in $T'$, $\zeta$ is also a removable path in $T$. Suppose, therefore, that the endpoints of $\zeta$ lie on $\boundary_- M$.

If $\zeta$ is disjoint from $\rho$, then $\zeta$ continues to be a removable path in $T$. Assume that $\zeta$ is not disjoint from $\rho$. Resolving $h$ results in a collection of vertical edges. At most two of these vertical edges lie in $\zeta$. 

If exactly one edge $h_\lambda$ lies in $\rho$, then $T$ is perturbed as in Case (2). Suppose, therefore, that $h_\lambda \neq h_{\lambda'}$ both lie in $\zeta \cap \rho$. Then the component of $T$ containing $h$ consists of a loop $\gamma$ and the edge $h$. There are cancelling discs $\{D_1,D_2\}$ for $\gamma$ such that $D_1$ is a bridge disc for the bridge arc component of $\zeta - H$. The disc $D_2$ is disjoint from the disc $E$ in (RP 4).  See Figure \ref{Fig: PodHandlesRemovablePath}.

\begin{center}
\begin{figure}[ht]
\scalebox{0.5}{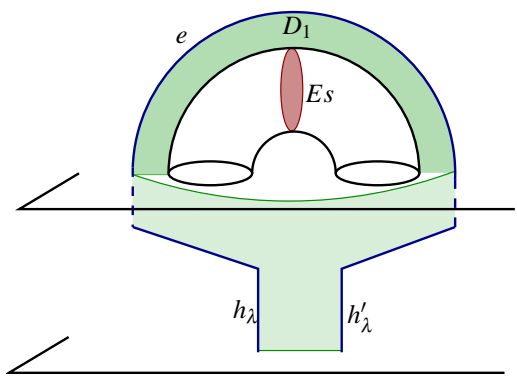}
\caption{Undoing a handle resolution causes removable paths to persist. A portion of the lower green disc becomes a bridge disc after undoing the pod resolution.}
\label{Fig: PodHandlesRemovablePath}
\end{figure}
\end{center}

\item If $M$ is a 3--ball, then $\boundary_- M = \nil$ and $T = T'$.

\item Suppose that $M = \boundary_- M \times I$, that $H$ is isotopic in $M - \inter{\eta}(T')$ to $\boundary_+ M - \inter{\eta}(T')$ and that each edge of $T'$ intersects $H$ no more than one time. Since $H$ is parallel to $\boundary_+ M$, $H$ separates $\boundary_+ M$ from $\boundary_- M$. Thus, $(C_2,T'_2) = (\boundary_+ M \times I,\text{points} \times I)$. Since $h \subset C_1$, we have $(C_2,T_2) = (C_2,T'_2)$. Thus, conclusion (6) holds for $(M,T)$.
\end{enumerate}
\end{proof}

\subsection{Puncturing Interior Monotonic Edges}
Suppose that $(M,T)$ satisfies the hypotheses of Theorem \ref{Main Theorem} and that $e$ is an interior monotonic edge of $T \subset M$. Then one endpoint $v$ of $e$ is in $C_1$. Let $M' = M - \inter{\eta}(v)$ and $T' = T \cap M'$. If there are no pod handles adjacent to $\boundary_- M$, then since $v \in C_1$, $H'$ is a Heegaard surface for $(M',T')$. It is clear that $(M',T')$ satisfies the hypotheses of Theorem \ref{Main Theorem}.

\begin{lemma}\label{Lem: puncturing mie}
Suppose that no pod handles of $T$ are adjacent to $\boundary_- M$. If $H$ as a splitting of $(M',T')$ satisfies the conclusion of Theorem \ref{Main Theorem}, then it also does so as a splitting of $(M,T)$.
\end{lemma}

\begin{proof}
Let $C'_i$ for $i = 1,2$ denote the compressionbodies into which $H$ splits $M'$. If $H$ is stabilized or perturbed as a splitting of $(M',T')$ then it is stabilized or perturbed as a splitting of $(M,T)$ since any disc which is essential in $C'_i$ is also essential in $C_i$.

Suppose that $H$ is boundary-stabilized as a splitting of $(M',T')$ and let $F$ be the component of $\boundary_- M'$ along which the boundary-stabilization was performed. If $F \neq \boundary \eta(v)$, then $H$ is boundary-stabilized as a splitting of $(M,T)$. If $F = \boundary \eta(v)$, then an argument similar to that of Lemma \ref{Lem: Resolving Pod Handles 2} (Resolving Pod Handles) shows that $H$ is perturbed as a splitting of $(M,T)$. 

Suppose that $\zeta$ is a removable path in $T'$. If $\zeta$ is disjoint from $\boundary \eta(v)$, then $\zeta$ remains a removable path in $T$. Suppose, therefore, that $\zeta$ has both endpoints on $\boundary \eta(v)$. Let $D$, $\Delta$, and $E$ be the discs from (RP 4). Choose a vertical disc $V$ in $C_1$ which is disjoint from $\Delta \cup E$ and which contains the edges $\zeta \cap C_1$. Let $V' \subset \eta(v)$ be a bridge disc for $T \cap \eta(v)$ such that $V \cap \boundary \eta(v) = V' \cap \boundary \eta(v)$. Let $D_1 = D$ and $D_2 = V \cup V'$. Let $\zeta_T$ be the loop in $T$ containing $\zeta$. The pair $\{D_1,D_2\}$ is a cancelling pair for $\zeta_T$ which is disjoint from $\Delta$ and which intersects $E$ exactly once. Thus, $\zeta_T$ is a removable loop in $T$.

It is impossible for $M'$ to equal $B^3$. If $M' = \boundary_- M' \times I$, then $M' = S^2 \times I$ since $\boundary \eta(v) \subset \boundary_- M'$. If $H - \inter{\eta}(T')$ is isotopic in $M - \inter{\eta}(T')$ to $\boundary_+ M' - \inter{\eta}(T')$, then both endpoints of $e \cap M'$ lie on $\boundary M'$. This contradicts the fact that $e$ is an interior monotonic edge.
\end{proof}

We say that $(M',T')$ is obtained by \defn{puncturing} an interior monotonic edge of $(M,T)$.

\subsection{Essential Surfaces in (Compressionbody, Trivial Graph)}
\begin{proposition}[{cf. \cite[Lemma 2.4]{HS1}}]\label{Prop: surface classification}
Let $T$ be a graph trivially embedded in a compressionbody $C$. Suppose that $F$ is a compact embedded surface in $C$ such that $F - T$ is connected. Suppose also that $\boundary F \subset (\boundary C \cup T)$, $(F \cap \boundary C) \subset \boundary F$, and $F \cap T$ is the union of edges of $T$. If $F$ is $T$--incompressible and $T$--$\boundary$--incompressible, then $F$ is one of the following:
\begin{enumerate}
\item a sphere disjoint from $T$
\item a properly embedded disc $D \subset C$ which is disjoint from $T$
\item a properly embedded disc $D$ such that $\boundary D \subset \boundary_+ C$ and $D \cap T$ is a single pod leg
\item a properly embedded disc $D$ such that $\boundary D \subset \boundary_- C$ such that $D \cap T$ is a single pod handle
\item a bridge disc
\item a vertical disc $D$ such that $\boundary D \cap T$ has two components. Each component is either a vertical edge or a pod leg and an adjacent pod handle
\item a vertical annulus $A$ such that $A \cap T$ is either empty, consists of a vertical edge, or consists of a pod handle and one or two pod legs.
\item one of types (6) or (7) that also contains some number of pod handles but none of their adjacent pod legs
\item a closed component parallel to $\boundary_- \tild{C}$.
\end{enumerate}
\end{proposition}

Figure \ref{Fig: EssentialSurfCBdy} shows a schematic representation of surfaces of type (7) and (8).
\begin{center}
\begin{figure}[ht]
\scalebox{0.5}{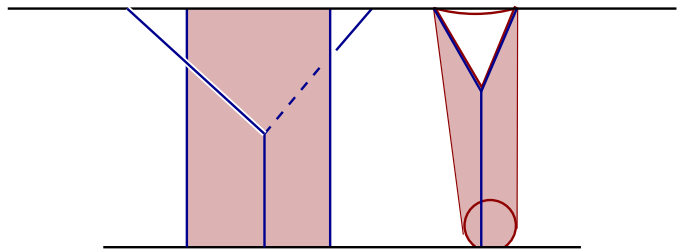}
\caption{A disc of type (8) and an annulus of type (7) which contains a handle and two adjacent pod legs.}
\label{Fig: EssentialSurfCBdy}
\end{figure}
\end{center}
\begin{proof}
Let $F$ be a surface satisfying the hypotheses of the proposition and suppose that $F$ is not of type (3), (4), or (5).

\textbf{Claim 1:} $F \cap T$ contains a pod leg in its interior only if it also contains an adjacent pod handle.

Let $e$ be a pod leg of a pod $\tau \subset T$ such that $e \subset F$ and $F$ does not contain a handle for $\tau$. Let $D$ be a pod disc for $\tau$. The boundary $\delta$ of a small regular neighborhood of $D$ is a disc which intersects a handle for $\tau$ once. Either $F$ divides $\delta$ into two discs or $F \cap \delta = \nil$. In the latter case, $F \subset D$ and so $F$ is a bridge disc. In the former case, let $\delta'$ be the disc which is the closure of a component of $\delta - F$ not intersecting the handle for $\tau$. The disc $\delta'$ is a $T$--$\boundary$-compressing for $F$, unless $F$ is of type (3). \qed (Claim 1)

\textbf{Claim 2:} If $F \cap T$ contains a valence 1 vertex in its interior then there is a graph $T'$ in $C$ such that $C$, $T'$, and each component of $F - T'$ satisfy the hypotheses of the theorem. The graph $T'$ is obtained by introducing additional pod legs to $T$ which lie in $F$. $F \cap T'$ has no valence 1 vertices in the interior of $F$.

Any valence one vertex of $T \cap F$, by Claim 1, must be an endpoint of a pod handle in $F$. Let $h$ be the pod handle and $\tau$ be the pod containing $h$.  Let $\rho$ be a regular neighborhood of $h$ in $C$ and consider the punctured disc $Q = \boundary \rho - \tau$. The boundary of the disc lies in $\boundary_- C$ and has one puncture for each leg of the pod $\tau$. The intersection $Q \cap F$ consists of a single essential arc on $Q$. 

Let $p_1$ and $p_2$ be punctures in $Q$ on opposite sides of the arc $Q \cap F$. Let $e_1$ and $e_2$ be the pod legs of $\tau$ associated to $p_1$ and $p_2$. Let $\delta$ be the bridge disc containing them. An innermost disc/outermost arc argument shows that we may assume that $F \cap \delta$ consists of a single arc which has one endpoint at the pod vertex and the other at $\boundary_+ C$. 

Add the edge $F \cap \delta$ to $\tau$ creating a pod $\tau'$ with an additional pod leg. The discs $\delta - F$ are bridge discs for $\tau'$, showing that $\tau'$ is trivially embedded. Do this for each such $\tau$ to create the graph $T'$.  It is easy to see that $F$ is $T'$-incompressible and $T'$--$\boundary$-incompressible. Hence, $C$, $T'$, and each component of $\cls(F - T')$ satisfy the hypotheses of the proposition. \qed (Claim 2)

Let $F'$ be a component of the closure of $F - T'$, where $T'$ is the graph provided by Claim 2. Notice that, by Claims 1 and 2, each vertex of $F' \cap T'$ in the interior of $F'$ has valence at least 2. Let $T''$ be the graph created from $T'$ by resolving all the pod handles of $T$. Since $F' - T$ is connected, after a small isotopy of $F'$ to intersect a new vertical arc, rather than a pod handle, resolving the pod handles does not have any effect on $F'$ and the intersection betweeen $F'$ and $T'$ is the same as between $F'$ and $T''$.

Suppose that $\tau$ is a pod of $T$ having a handle $h$. Let $e_1, \hdots, e_n$ be the vertical arcs in $C$ created by resolving $h$. We think of the $e_i$ as lying in $\boundary \eta(\tau)$. Let $D$ be a pod disc for $\tau$. Let $\delta_1, \hdots, \delta_n$ be the vertical discs contained in $\boundary \eta(\tau \cup D)$ so that $\delta_i$ contains $e_i$ and $e_{i+1}$ (with indices mod $n$) in its boundary. By construction, $F''$ is disjoint from the interior of each $\delta_i$. The boundary of $F''$ may intersect $\boundary \delta_i$ along $e_i$ and $e_{i+1}$.

Suppose that $D$ is a $T''$-compressing or $\boundary$--compressing disc for $F'$. A standard innermost disc/outermost arc argument shows that we may assume that $D \cap \delta_i = \nil$ for all $i$. Then $D$ is a $T'$-compressing disc or $\boundary$--compressing disc for $F'$, a contradiction. Hence, $F'$ is $T''$ incompressible and $T''$--$\boundary$-incompressible.

Since $F' - T''$ is connected and not of type (3) or (5), by \cite[Lemma 2.4]{HS1}, $\cls(F' - T'')$ is one of the following:
\begin{enumerate}
\item[(A.)] a sphere disjoint from $T''$
\item[(B.)] a properly embedded disc disjoint from $T''$
\item[(C.)] a bridge disc for $T''$
\item[(D.)] a vertical disc such that $\boundary F' \cap T''$ contains two vertical arc components
\item[(E.)] a vertical annulus $A$ such that $A \cap T''$ is either empty or contains a vertical arc
\item[(F.)] a closed component parallel in $\tild{C}$ to $\boundary_- C$.
\end{enumerate}

If $F' - T''$ is of type (A.) or (B.), $F' - T'$ is a sphere or disc disjoint from $T'$ and, thus, $F$ is a sphere or disc disjoint from $T$.

If $F' - T''$ is of type (C.), then $F' - T'$ is a bridge disc for $T'$, since $F'$ does not contain a pod handle of $T'$. For the same reason, $F - T$ is a bridge disc for $T$.

If $F' - T''$ is of type (E.), then $T' \cap F'$ consists of at most one vertical arc and so either $F$ is of type (7) or (8). In the latter case, it contains exactly one pod handle.

If $F' - T''$ is of type (D.), then $F' - T'$, if it not the same as $F' - T''$, is obtained from $F' - T''$ by gluing together two copies of a pod handle in $\boundary F'$. Thus, $F' - T'$ is either a vertical disc or is a vertical annulus of type (7) which contains two pod legs. $T'$ was created from $T$ by attaching pod legs to pod handles in the interior of $F$ which were not adjacent to pod handles in $F$. If $F'$ is a vertical annulus containing a pod handle, the pod handle is adjacent to two pod legs in $F'$ and so if $F'$ is a vertical annulus, then $F$ is of type (7). 

By the previous observations, we may assume that each component of $F - T'$ is a vertical disc with interior disjoint from $T'$. The closure of $(F - T)$ is obtained by removing a pod leg of $T' \cap F$ which is the only pod leg in $F$ attached to a particular pod handle in $F$. Thus, $F$ is of type (6), (7), or (8).
\end{proof}

\subsection{Frohman's Trick}

We will frequently use a technique due to Frohman \cite{F} for determining if a Heegaard splitting is stabilized. Informally, if a spine for the compressionbody on one side of a Heegaard surface for $(M,T)$ contains a cycle which is contained in a 3--ball in $M - T$, then the Heegaard surface is stabilized. The version we will most often use is \cite[Lemma 4.1]{HS1}. We refer the reader to that Lemma for a precise description of how the trick works in our context.

\section{The complex $R$}\label{sec:R}

Suppose $M$ is a compressionbody and $T$ is a properly embedded graph so that $\boundary_+ M - \inter{\eta}(T)$ is isotopic to the frontier of a regular neighborhood of $\boundary_- M \cup T$. Furthermore, suppose $T$ satisfies the following conditions:

\begin{enumerate}
\item[(A)] Each edge of $T_0$ with zero or two endpoints on $\boundary_- M$ has at least one spoke attached to its interior.

\item[(B)] Each vertex of $T_0 - \boundary T_0$ has a spoke attached to it.

\item[(C)] Each component of $\boundary_- M$ is adjacent to at least one component of $T_v$.
\end{enumerate}

In this case we can construct a complex $R$ which has many useful properties, in particular $T \subset R$ and cutting $M$ along $R$ produces a collection of $3$-balls. Our complex $R$ generalizes the surface $R$ in \cite{HS1} and the surface $\mc{R}$ in \cite{HS2}. Our $R$ is considerably more complicated than either of those surfaces, so it is recommended that the reader understand the constructions in \cite{HS1} and \cite{HS2} before tackling this construction. The case when $M$ is a handlebody is also easier than when $M$ is a compressionbody. The first time reader may want to concentrate on that situation. The figures below often (but not always) depict the situation when each spoke has a single edge. The situation when spokes have multiple edges is not significantly more complicated, so we avoid cluttering the figures with unnecessary detail. In describing the construction we will often isotope $T$. The isotopies described are always ambient isotopies which never move $T$ through $H$.

Recall that if $M = B^3$, then $T = T_s$. In this case, let $D$ be a properly embedded disc in $M$, containing $T$ and let $R = D$. 

If $M \neq B^3$, for each edge of $T_0$ with zero or two endpoints on $\boundary_- M$ and for each spoke adjacent to the interior of the edge choose a disc which intersects the edge in exactly one point, contains a spoke adjacent to the interior of the edge, and is otherwise disjoint from $T$. See Figure \ref{Fig: Discs D}. Let $D$ be the union of these discs. For each disc $D'$ in $D$, choose a minimal path in $T_s \cap D'$ from $D' \cap T_0$ to $\boundary D'$. Call this path the \defn{distinguished path}.

\begin{center}
\begin{figure}[ht]
\scalebox{0.5}{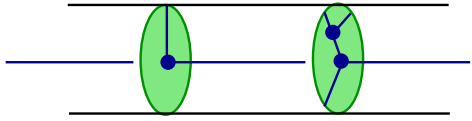}
\caption{Two discs in $D$ adjacent to a single edge of $T_0$. Each contains a spoke. One spoke is a single edge and the other is a tree with four edges.}
\label{Fig: Discs D}
\end{figure}
\end{center}

If $M \neq B^3$, boundary reducing $M$ using $D$ creates a 3--manifold which is the union of $M_0 = \boundary_- M \times I$ and a collection of 3--balls. Let $\boundary_+ M_0 = \boundary_- M \times \{1\}$ and $\boundary_- M_0 = \boundary_- M \times \{0\}$. For a component $M'_0$ of $M_0$, let $\boundary_+ M'_0 = \boundary M'_0 \cap \boundary_+ M_0$ and let $\boundary_- M'_0 = \boundary M'_0 \cap \boundary_- M_0$. Let $D_+$ denote the discs in $D$ which are contained in $\boundary_+ M_0$.  

Each component of $T_0 \cap M_0$ is a tree with a single edge having an endpoint on $\boundary_- M$. Let $\tau$ be a component of $T_0 \cap M_0$ and let $h$ be the edge of $\tau$ with an endpoint on $\boundary_- M$. There is a properly embedded disc $E(\tau)$ with boundary in $\boundary_+ M$ which is inessential in $M$ such that $E(\tau) \cap T = E(\tau) \cap h$ is a single point that separates the point $\boundary h - \boundary_- M$ from all other vertices of $T$ on $h$. See Figure \ref{Fig: Disc E}.

\begin{center}
\begin{figure}[ht]
\scalebox{0.5}{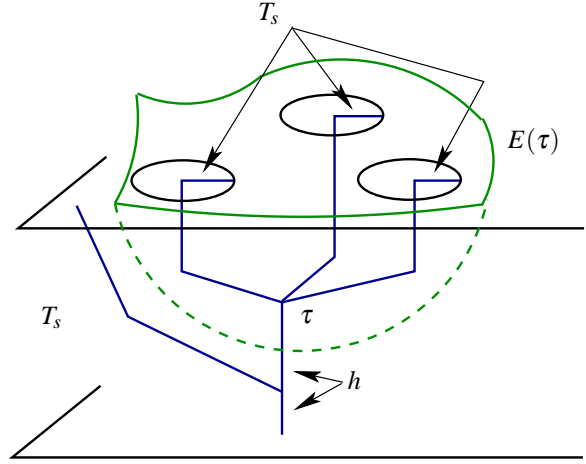}
\caption{An example of the disc $E(\tau)$.}
\label{Fig: Disc E}
\end{figure}
\end{center}

If $M$ is a handlebody, let $\mb{E} = \nil$. Otherwise, let $\mb{E} = \bigcup_\tau E(\tau)$. Boundary reducing $M$ using $\mb{E}$ cuts off handlebodies $U$ which are disjoint from $\boundary_- M$. (If $M$ is a handlebody, then $U = M$.) $T_0 \cap U$ is a graph which is isotopic into $\boundary U$. Let $A'$ be an embedded 2--complex which is the union of squares in $U$. A square of $A' - T$ has two edges in $T$ and in most cases is a square of the form \[\Big(\text{edge of } (T_0 - T_s)\Big) \times I.\] The one exception is a square adjacent to some $h$. Such a square consists only of a $(\text{portion of } h) \times I$. We may isotope $A'$ so that for each disc $D'$ in $D$, $A' \cap D'$ is the distinguished path of $T_s \cap D'$. Each vertex of $T_0$ is adjacent to a spoke. Isotope each of these spokes to lie in $A'$ so that there is a minimal path in the spoke which is $(\text{vertex } \times I)$. Then $A' - T$ is the union of pairwise disjoint discs which meet in unions of edges of $T$.  If $\boundary_- M = \nil$ (i.e. if $M$ is a handlebody) we let $A = A'$ and $R = A \cup D$. Notice that in this case, there is no edge $h$.

Recall that each square of $A' \subset U$ is of the form \[\Big(\text{edge of } (T_0 - T_s)\Big) \times I.\] Suppose that $\tild{A}$ is such a square that is not adjacent to an edge $h$ of $T_0$ having one endpoint on $\boundary_- M$. Then one edge of $\boundary \tild{A}$ lies on an edge of $T_0$, two edges of $\boundary \tild{A}$ lie on distinct components of $T_s$, and one edge of $\boundary \tild{A}$ lies on $\boundary_+ M$. $\tild{A}$ may contain edges of $T_s$ in its interior. No edge of $T_s$ interior to $\tild{A}$ has an endpoint interior to $T_0 \cap \boundary \tild{A}$.

Figure \ref{Fig: Rhandlebody} depicts the surface $R$ when $M$ is a genus 2 handlebody and $T_0$ is a $\theta$-graph.

\begin{center}
\begin{figure}[ht]
\scalebox{0.5}{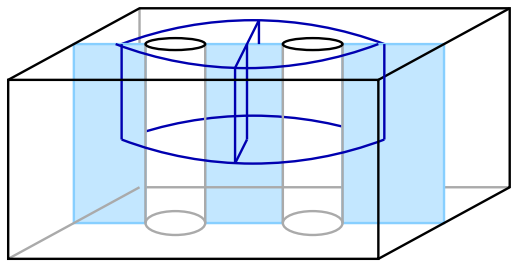}
\caption{An example of the surface $R$ in a genus 2 handlebody. The discs $D$ are shaded. The surface $A$ is outlined.}
\label{Fig: Rhandlebody}
\end{figure}
\end{center}

If $\boundary_- M \neq \nil$ we still need to extend the squares of $A'$ adjacent to an edge $h$ of $T_0$ having an endpoint on $\boundary_- M$ to a surface with boundary on $\boundary M \cup T$. To that end, suppose that $\boundary_- M \neq \nil$. Let $M'_0$ be a component of $M_0$ and let $g = \genus(\boundary_+ M'_0)$. Since $M$ is a compressionbody, by construction, the complex $(A' \cup D) \cap \boundary_+ M'_0$ is contained in a disc in $\boundary_+ M'_0$. In $M'_0$ there exists, by assumption (C), a component of $T_v$. Choose a distinguished path $p$ in this component from $\boundary_- M'_0$ to $\boundary_+ M'_0$ and perform an isotopy so that $p$ is equal to $\{\text{point}\} \times I$ in $M'_0$. Choose loops $\gamma$ in $\boundary_+ M'_0$ based at $p \cap \boundary_+ M'_0$. If $\boundary_+ M'_0 = S^2$, then $\gamma$ should be a single loop cutting $\boundary_+ M_0$ into two discs. If $\boundary_+ M'_0 \neq S^2$, $\gamma$ should cut $\boundary_+ M'_0$ into a $4g$--gon. There is a collection of vertical annuli $P'_0$ in $M'_0$ so that $P'_0 \cap \boundary_+ M'_0 = \gamma$. The annuli intersect on $p$. Choose the curves $\gamma$ and the annuli $P'_0$ to be disjoint from the remnants of $D_+$ and $\mb{E}$. Also isotope $P'_0$ so that $P'_0 \cap T = T_v$ with the intersection of the annuli still $p$. Cutting $M'_0$ along $P'_0$ creates $G$ which is a $(4g-\text{gon})\times I$ if $\boundary_+ M'_0 \neq S^2$ and otherwise is the disjoint union of two copies of $D^2 \times I$. Let $Z_\pm = G \cap \boundary_\pm M'_0$. Thus, $Z_-$ and $Z_+$ are each a $4g$--gon. The polygon $Z_+$ contains $(A' \cup D) \cap \boundary_+ M'_0$. See Figure \ref{Fig: G1} for an example with $g = 1$. Let $P$ be the union of the $P'_0$ over all components $M'_0$.

\begin{center}
\begin{figure}[ht]
\scalebox{0.35}{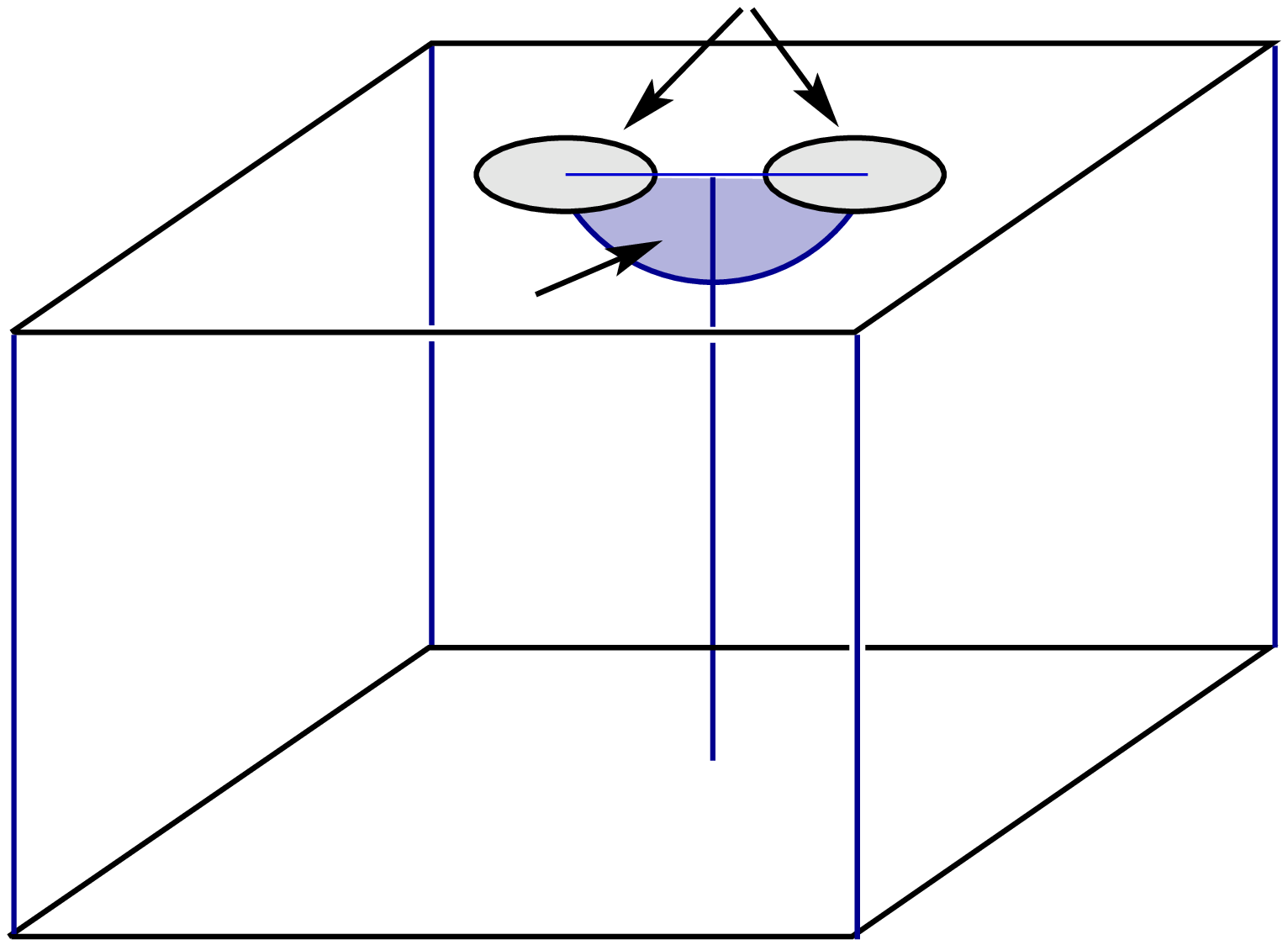}
\scalebox{0.35}{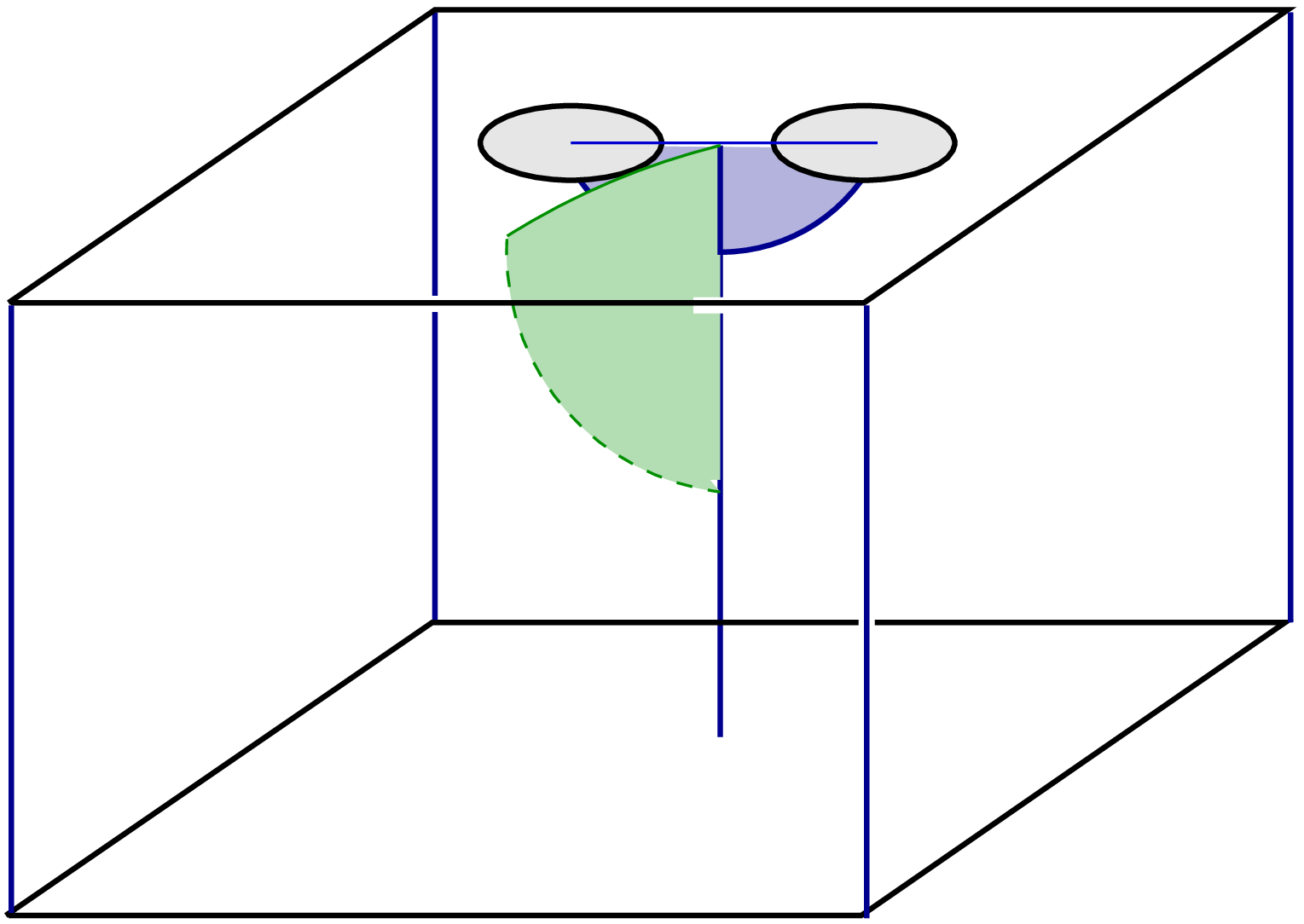}
\caption{An example with $g=1$. The first picture does not show the portion of $A' \cap G$ adjacent to $h$. The disc $E$ is not pictured.}
\label{Fig: G1}
\end{figure}
\end{center}

The complex $A'$ contains portions of the edges of $T_0$ which have an endpoint on $\boundary_- M$. We now extend $A'$ to a complex which we write as $A \cup B$, so that $T$ is contained in $A \cup B \cup D \cup P$. Let $p_+$ be a vertex of the polygon $Z_+ \subset M'_0$. Of course, $p_+$ is identified with $p \cap \boundary_+ M'_0$ after gluing the $4g$--gon together to obtain $\boundary_+ M'_0$. Let $p_-$ be the vertex of $p \cap Z_-$ so that in $G$ there is a copy of $p$ joining $p_-$ to $p_+$. 

Consider an edge $h$ of $T_0 \cap M'_0$ having an endpoint on $\boundary_- M'_0$. In $A'$, there is one square having one side on a portion of $h$, one side on a spoke $\sigma$ adjacent to the vertex $\boundary h - \boundary_- M$, and one side on $\boundary_+ G'_0$. The fourth side of the square does not lie on either $T$ or $\boundary G$, but does lie in $E$. For such a square $S$, choose a path $\beta(S)$ joining the point $\boundary \sigma \cap \boundary_+ G'_0$ to $p_+$, which contains the edge $S \cap \boundary_+ G'_0$ and which is otherwise disjoint from $A'$. Let $B(h) = \beta(S) \times I \subset G'_0$ and let $B$ be the union of the $B(h)$ for all $h \subset T_0$. 

Figure \ref{Fig: G3} shows an example of $G'_0 \cap (A \cup B)$ for a component $G'_0$ which is a $\text{square} \times I$. In this example, $T_0$ has an edge with a single endpoint on $\boundary_- G'_0$.

\begin{center}
\begin{figure}[ht]
\scalebox{0.35}{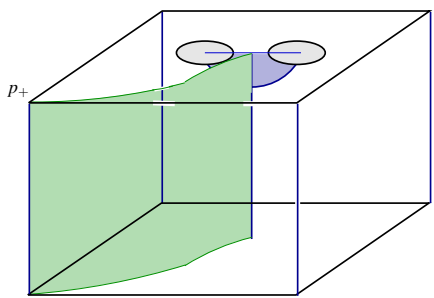}
\caption{An example with $g=1$. The blue surface is a portion of $A$, the green surface is a component of $B$.}
\label{Fig: G3}
\end{figure}
\end{center}

Figure \ref{Fig: Surface B1} shows an example of the surface $B$ in a compressionbody $M$ with $\boundary_- M$ a torus and $\boundary_+ M$ a genus 2 surface. In the picture, $M$ has been cut open along $P$. In this example, $T_0$ is a single edge with both endpoints on $\boundary_- M$.

\begin{center}
\begin{figure}[ht]
\scalebox{0.35}{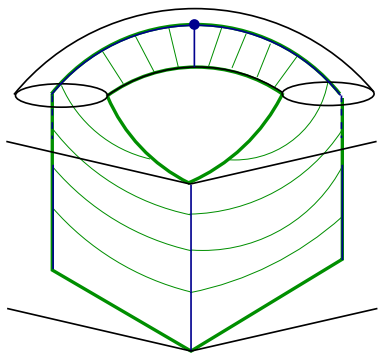}
\caption{An example of the surface $B$ when $T_0$ has an edge with both endpoints on $\boundary_- M$.}
\label{Fig: Surface B1}
\end{figure}
\end{center}

Isotope any spokes adjacent to $h$ so that they lie in $B$. Let $A = \cls(A' - B)$ so that we may think of $A$ as the union of squares each with three edges on $T$ and one edge on $\boundary_+ M$. These squares may contain edges of $T_s$ in their interior. Let $R = A \cup B \cup D \cup P$.

\begin{lemma}
The complex $R$ has the following properties:
\begin{enumerate}
\item $T \subset R$.
\item Each component of $R - T$ is a disc with boundary lying on $\boundary M \cup T$.
\item $G = \cls(M - R)$ is a collection of 3--balls containing copies of portions of $R$ in their boundary.
\item $R$ is $T$--incompressible and $T$--$\boundary$--incompressible.
\item No component of $R$ is disjoint from $T$.
\end{enumerate}
\end{lemma}
\begin{proof}
To see that the first claim is true, notice that each component of $T_v$ is contained in some component of $P$, $T_0$ is contained in $A \cup B$, and the spokes in $T_s$ were all isotoped to lie in $A \cup B \cup D$ depending on which edges or vertices of $T_0$ they were adjacent to.

The second observation is equally easy to see. Notice that each component of $R - T$ is a component of $A - T$, $D - T$, $B - T$, or $P - T$. We have already observed that the first statement is true for components of $A - T$. The statement is obviously true for components of $P - T$ since $P$ is the union of annuli containing components of $T_v$. Each component of $D$ and $B$ is a disc and so the statement is also true for components of $D - T$ and $B - T$. See Figure \ref{Fig: Opened R1} for a picture of $R \subset \boundary G$ when $M$ and $T$ are the 3--manifold and graph from Figure \ref{Fig: Surface B1}. See Figure \ref{Fig: Opened R2} for a picture of $R \subset \boundary G$ when $M$ and $T$ are the 3--manifold and graph from Figure \ref{Fig: G3}. 

\begin{center}
\begin{figure}[ht]
\scalebox{0.5}{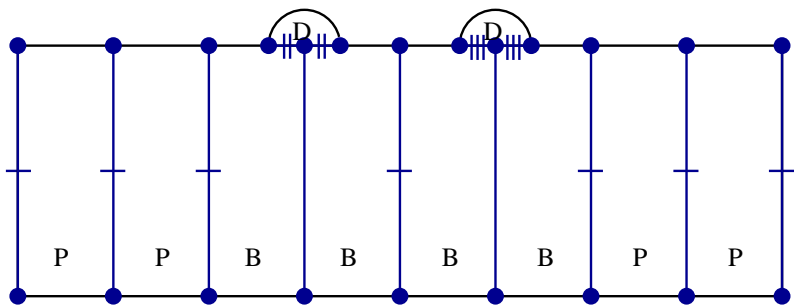}
\caption{The surface $R$ in $\boundary M$ when $M$ is a compressionbody with $\boundary_- M$ a torus and $\boundary_+ M$ a genus 2 surface. The graph $T_0$ is a single edge with both vertices on $\boundary_- M$. The left and right edges should be glued together so that $P \cup B$ is an annulus in $\boundary G$. Edges with the same positive number of hash marks are identified in $M$. The edge with a single hash mark is the sole edge of $T_v$. The edges with two and three hash marks belong to $T_s$. Non-horizontal edges without hash marks are not glued to any other edges.}
\label{Fig: Opened R1}
\end{figure}
\end{center}

\begin{center}
\begin{figure}[ht]
\scalebox{0.5}{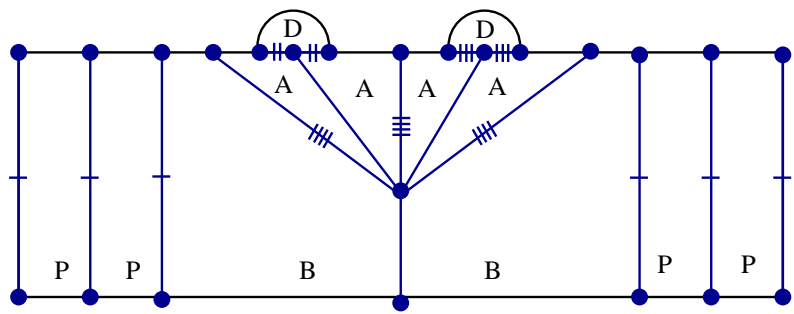}
\caption{The surface $R$ in $\boundary M$ when $M$ is a compressionbody with $\boundary_- M$ a torus and $\boundary_+ M$ a genus 2 surface. The graph $T_0$ is a graph with two edges and two vertices, one of which is on $\boundary_- M$. The left and right edges should be glued together so that $P \cup B$ is an annulus in $\boundary G$. Edges with the same positive number of hash marks are identified in $M$. The edge with a single hash mark is the sole edge of $T_v$. The edges with two, three, and four hash marks belong to $T_s$. Non-horizontal edges without hash marks are not glued to any other edges. }
\label{Fig: Opened R2}
\end{figure}
\end{center}

To see the third statement recall that a component of $M - (D \cup P)$ adjacent to $\boundary_- M$ is a ball of the form $G'_0 = (4g-\text{gon}) \times I$. Each component of $(A \cap G'_0)$ is a disc with boundary on $\boundary_+ G'_0 \cup T$. Each component of $A \cap G'_0$ is joined to a single $(\text{vertex} \times I)$ by a square in $B$. Thus, $G'_0 - (A \cup B)$ is simply connected and so is a 3--ball containing portions of $R$ in its boundary. 

A component of $M - (D \cup P)$ not adjacent to $\boundary_- M$ is also a component of $M - D$. Since every edge of $T_0$ is adjacent to a spoke in its interior and since every such spoke is contained in a disc of $D$, a component of $M - D$ not adjacent to $\boundary_- M$ is a 3--ball. Considerations similar to those already discussed show that the lemma is true in this case as well.

$R$ is $T$--incompressible and $T$--$\boundary$--incompressible, because, by construction, each component of $R - T$ is a disc. Similarly, from the construction it is obvious that no component of $R$ is disjoint from $T$.
\end{proof}

\section{Special case of Theorem \ref{Main Theorem}}

In this section we prove a key result which we will make use of in the proof of Theorem \ref{Main Theorem}. As a result of Lemmas \ref{Lem: Resolving Pod Handles 1} and \ref{Lem: Resolving Pod Handles 2}, we make the following assumption:

\textbf{(NPH)} $T_1$ and $T_2$ contain no pod handles adjacent to $\boundary_- M$.

A \defn{vertical cut disc} for $M$ is a boundary reducing disc for $\boundary_+ M$ which intersects $T$ transversally in a single point and intersects $H$ in a single simple closed curve. Notice that if $M = B^3$ or if $M = \boundary_- M \times I$, then there is no vertical cut disc. This is not the same definition of ``cut disc'' as that used in \cite{TT}, although it is related.

Let $\mb{I}$ denote the union of components of $T_s$ such that each component of $T_s$ in $\mb{I}$ consists of a single edge such that an endpoint of the edge is either at a vertex of $T_0$ or in the interior of an edge of $T_0$ which has zero or two endpoints on $\boundary_- M$. If $M = B^3$, we require that $\mb{I} = \nil$. Let $\mb{T} = T - \mb{I}$. Notice that $H$ is a Heegaard surface for $(M,\mb{T})$. 

\begin{proposition}\label{Prop: Strengthened Perturbed}
Suppose $M$ does not have a vertical cut disk and suppose that $H$ as a splitting of $(M,\mb{T})$ does not satisfy conclusions (1), (2), (5), or (6) of Theorem \ref{Main Theorem}. Suppose also that no edge of $T$ is perturbed and $T$ satisfies conditions (A)-(C) in Section \ref{sec:R} and condition (NPH). Then one of the following occurs:
\begin{itemize}
\item $H$ is perturbed as a splitting of both $(M,T)$ and of $(M,\mb{T})$; or
\item $T$ has a removable cycle.
\end{itemize}
\end{proposition}

\begin{proof}
Consider the complex $R$ constructed in Section \ref{sec:R}. Because $R$ is $T$--incompressible and $T$--$\boundary$--incompressible, a $T$-compression or $T$--boundary compression of $R\cap C_1$ may be accomplished by an isotopy of $R$. We may, therefore, isotope $R$ (fixing $T$ and $H$) so that $R \cap C_1$ consists of surfaces of type (2) - (9) in the statement of Proposition \ref{Prop: surface classification}. Out of all such isotopies of $R$, assume that $R$ has been isotoped so as to minimize $|R \cap H|$. Notice that $R \cap C_1$ is $T_1$--incompressible and $T_1$--$\boundary$-incompressible. $R \cap C_2$ is $T_2$--incompressible, but may be $T_2$--$\boundary$-compressible.

The following sequence of lemmas will complete the proof.

\begin{lemma}\label{Lem: vert cut disc}
If $R \cap C_1$ has a disc of type (2), then there exists a vertical cut disc for $M$. Furthermore, $R \cap C_2$ does not have a disc of type (2).
\end{lemma}

\begin{proof}
Suppose that $R \cap C_1$ has a disc of type (2).

\textbf{Case A:} $R \cap C_2$ is $T_2$--$\boundary$--compressible. 

Because $R - T$ is the union of discs, the proofs of \cite[Proposition 3.3]{HS1} and \cite[Proposition 3.3]{HS2} show that $R \cap C_1$ does not have a disc of type (2). \qed(Case A)

\textbf{Case B:} $R \cap C_2$ is $T_2$--$\boundary$--incompressible.

Let $E_1$ be the component of $R \cap C_1$ which is a disc of type (2). Let $E_2$ be the component of $(R \cap C_2) - T_2$ which contains $\boundary E_1$. Since $E_2$ is $T_2$--incompressible and $T_2$--$\boundary$--incompressible and since $\boundary E_1$ is disjoint from $T$, $E_2$ is a surface of type (2), (7), or (8) from the statement of Proposition \ref{Prop: surface classification}. Let $E = E_1 \cup E_2$. $E$ is one of the following:
\begin{itemize}
\item[(a)] a sphere disjoint from $T$,
\item[(b)] a disc disjoint from $T$ which intersects $H$ in a single simple closed curve, or 
\item[(c)] a disc intersecting $H$ in a single simple closed curve and which contains some number of pod handles in $C_2$.
\end{itemize}

Since $E$ is a component of $R - T$, it cannot be of type (a). Similarly, since $T$ contains a spine for $M$ and no component of $R$ is disjoint from $T$, $E$ cannot be of type (b). Therefore, $E$ is of type (c). The only components of $R - T$ whose closures have all of their boundary contained in $\boundary_+ M$ are discs in $D$. Thus, $E \subset D$. Each disc in $D$ contains a single spoke, and therefore $D \cap T_0$ consists of a single point. A slight isotopy of $D$ produces a vertical cut disc. \qed(Case B)

Suppose that $R \cap C_2$ has a disc of type (2). Since $R \cap C_1$ is $T_1$--$\boundary$--incompressible, an argument symmetric to that of Case B above produces a compressing disc for $\boundary_- M$ in $M$. Since $M$ is a compressionbody, this is impossible.
\end{proof}

As $M$ has no vertical cut-disks, we conclude that $R \cap C_1$ has no disc of type (2).

Let $Q$ be an elementary spine of $(C_1, T_1)$. In Case 2 below, we will be deforming $Q$ into a spine which may not be elementary. After such a deformation, we consider $Q = Q^1 \cup Q^2$ with $Q^1$ a graph and $Q^2$ the union of discs.

Think of $C_1$ as being a very small regular neighborhood of $Q \cup (\boundary_- M \cap C_1)$. By hypothesis (NPH), $Q$ does not contain any edges of $T$ (i.e. pod handles). Since $R \cap C_1$ does not have a disc of type (2), $Q \cap R = \boundary_1 Q \subset T$.

The next lemma will be useful at several points in upcoming arguments.

\begin{lemma}\label{Lem: Bdy-incomp discs}
Let $D'$ be a disc in $D$ and suppose that $D' \cap C_2$ is $T_2$--$\boundary$--incompressible. Let $x = D' \cap T_0$. The point $x$ may be a vertex of $T_s$ of valence 2 or more. Then $Q \cap D' \subset \boundary_1 Q$ is a single point and one of the following occurs:
\begin{itemize}
\item the component of $T_s$ contained in $D'$ consists of a single edge and $\boundary_1 Q \cap D' = x$.
\item $T_s \cap D'$ has a single vertex not on $\boundary D'$ and that vertex is $x$ which has valence at least 3. Furthermore, $\boundary_1 Q \cap D' = x$.
\item $T_s \cap D'$ has two vertices not on $\boundary D'$. One of those vertices is $x$ and $x$ has valence 2. The other vertex is the sole point of $\boundary_1 Q \cap D'$.
\end{itemize}
\end{lemma}

\begin{proof}
Let $\tild{D}$ be the closure of a component of $D' - T$. The surface $\tild{D} \cap C_2$ has a portion of its boundary on $\boundary_+ M \subset \boundary_- C_2$ and also intersects $T$. It must, therefore, be a surface of type (4), (6), (7) or (8). 

If $\tild{D} \cap C_2$ is of type (4), then $D' = \tild{D}$. Furthermore, $T_s \cap D'$ is a single edge which is disjoint from $\boundary_1 Q$ except at the endpoint of $T_s \cap D'$ in the interior of $D'$. Hence, the first conclusion holds.

We may, therefore, assume that $\tild{D} \cap C_2$ is a surface of type (6), (7), or (8). Suppose that $\tild{D} \cap C_2$ is an annulus of type (7) or (8). In this case $\tild{D} = D'$. Since $\tild{D} \cap C_2$, in this case, has one circular boundary component on $H$ and the other on $\boundary_+ M$, $|\boundary_1 Q \cap D'| = 1$. Furthermore, because there is exactly one component of $T_2 \cap (\tild{D} \cap C_2)$ with an endpoint on $H$, and that component spans $\tild{D} \cap C_2$, the point $x$ is equal to $\boundary_1 Q \cap D'$. Since $\tild{D} = D'$, we conclude that $T_s \cap D'$ consists of a single edge and $\boundary_1 Q \cap D' = x$.

Henceforth, we assume that for each disc $\tild{D} \subset \cls(D' - T)$, $\tild{D} \cap C_2$ is a disc of type (6) or (8). Since $D' \cap T_0$ is a single point, $x$ is the only valence one vertex of $D' \cap T_s$. Furthermore, since $T_s \cap D'$ is a tree, it is connected. Thus, $\tild{D} \cap C_2$ is a disc of type (6) and not of type (8). This implies that no point of $T_s \cap D'$ is a vertex of valence one on the interior of $D'$. Also, if $\tild{D}$ is the closure of a component of $D' - T$, $\tild{D}$ is adjacent to exactly one point of $\boundary_1 Q$. Since no edge of $T$ with neither endpoint on $\boundary M$ lies entirely in $C_2$, each edge of $T_s \cap D'$ with neither endpoint on $\boundary D'$ contains at least one point of $\boundary_1 Q$. Similarly, if $e_1$ and $e_2$ are edges in $T_s \cap D'$ sharing one endpoint and having their other endpoints on $\boundary D'$, then there must be a point of $\boundary_1 Q$ in $e_1 \cup e_2$. If $x$ is a vertex of valence 2, for the moment cease considering it as a vertex of $T_s \cap D'$. Then, the hypotheses of Lemma \ref{Lem: Tree Lemma} (in the appendix) are satisfied with $T' = T_s \cap D'$ and $\mc{P} = \boundary_1 Q$. Hence, $|\boundary_1 Q \cap D'| = 1$ and $T_s \cap D'$ has at most a single vertex $y$ not on $\boundary D'$.  Furthermore, if $y$ exists, it is the sole point of $\boundary_1 Q \cap D'$. 

If $x$ has valence 3 or more, then $y = x$, since $y$ is the sole vertex of valence 3 or more on the interior of $D'$.
\end{proof}

\begin{corollary}\label{Cor: Incomp Disc implies Vert. Cut Disc}
Suppose that $M \neq B^3$ and that $D'$ is a disc in $D$ such that $D' \cap C_2$ is $T_2$--$\boundary$--incompressible. Then there exists a vertical cut disc, contrary to the hypotheses.\end{corollary}
\begin{proof}
By Lemma \ref{Lem: Bdy-incomp discs}, $D'$ intersects $Q$ exactly once. Since $M \neq B^3$, the disc $D'$ is an essential disc in $M$. By construction, $\boundary D'$ lies on $\boundary_+ M$ and $D'$ intersects $T_0$ exactly once in the interior of an edge $e$ of $T_0$. A slight isotopy of $D'$ off $T_s$ produces a disc which is disjoint from $T_s$, intersects $T_0$ in a single point and intersects $H = \boundary \eta(Q \cup \boundary_- C_1)$ in a single simple closed curve.
\end{proof}

We now consider two cases. The first is when $R \cap C_2$ is $T_2$--$\boundary$--incompressible and the second is when $R \cap C_2$ is $T_2$--$\boundary$--compressible.

\textbf{Case 1:} $R \cap C_2$ is $T_2$--$\boundary$--incompressible.

In this case the closure of each component of $(R \cap C_2) - T_2$ is a surface of type (2) - (8).  By Lemma \ref{Lem: vert cut disc} each component of $(R \cap C_2) - T_2$ is, in fact, a surface of type (3) - (8). 

\begin{lemma}\label{Lem: No D}
Either $M = B^3$ or $M = \boundary_- M \times I$.
\end{lemma}
\begin{proof}
If $D \neq \nil$, then for any disc $D'$ in $D$, $D' \cap C_2$ is $T_2$--$\boundary$--incompressible. If $T_0$ contains an edge, then (by construction of $T_0$ and $R$), there exists a disc $D'$ in $D$ with $\boundary D'$ essential in $\boundary_+ M$. Then by Corollary \ref{Cor: Incomp Disc implies Vert. Cut Disc}, there exists a vertical cut disc, contrary to the hypotheses of the Proposition.  Thus, if $D \neq \nil$, $T_0$ is a single point. In which case, $M = B^3$, $T_s = T$, and $R = D$ is a single disc containing $T_s$. If $D = \nil$, then $M = \boundary_- M \times I$.
\end{proof}

\begin{lemma}\label{Lem: Case 1a.i for B^3}
If $M = B^3$, then either $H$ is stabilized or $H- T$ is properly isotopic in $M - T$ to $\boundary M - T$.
\end{lemma}
\begin{proof}
If $M = B^3$ then $T_0$ is a single point $x$ and $T = T_s$. $R = D$ is a single disc which contains $T$. By Lemma \ref{Lem: Bdy-incomp discs}, $Q \cap R$ is a single point. $G$, which is the closure of $M - R$, consists of two 3--balls, each with a copy of $T$ in its boundary. If $Q$ has an edge, then there must be a cycle in $Q \cap G$. By Frohman's trick \cite[Lemma 4.1]{HS1}, $H$ is stabilized. Thus we may assume that $Q$ is a single point. If $T$ contains more than one edge, $Q$ is the sole interior vertex of $T$. Whether or not $T$ has more than one edge, $H - T$ is parallel in $M - T$ to $\boundary M - T$.
\end{proof}

We now turn our attention to the case when $M = \boundary_- M \times I$. In this case, $T = T_v$.

\begin{lemma}\label{Lem: P disjt from Q}
Suppose that $\tild{P}$ is the closure of a component of $P - T$ adjacent to $\boundary_- M$. Then $\tild{P} \cap Q = \nil$.
\end{lemma}
\begin{proof}
We may think of $\tild{P}$ as a square with two opposite sides on $T$, one side on $\boundary_+ M$, and one side on $\boundary_- M$. Suppose that $Q \cap \tild{P} \neq \nil$.

If $\boundary_- M \subset C_2$, then $\tild{P} \cap C_2$ is a disc having at least three disjoint sub-arcs of its boundary on $T$ and two subarcs on $\boundary_- C_2$. However, no surface of type (2)-(8) has this property. (The surfaces of type (8) contain pod handles in their interior, not on their boundary.) Thus, $\boundary_- M \subset C_1$.

If $\boundary_- M \subset C_1$, then $\tild{P} \cap C_2$ is a disc having at least two subarcs of its boundary on $H$, one subarc on $\boundary_+ M \subset \boundary_- C_2$, and at least three subarcs on $T$. No such surface appears in the list from Lemma \ref{Prop: surface classification}.
\end{proof}

\begin{corollary}\label{Cor: H separates}
$H$ separates $\boundary_- M$ from $\boundary_+ M$.
\end{corollary}
\begin{proof}
Suppose to the contrary that $\boundary_- M \subset C_2$. Let $\tild{P}$ be the closure of a component of $P - T$ adjacent to $\boundary_- M$. By Lemma \ref{Lem: P disjt from Q}, $\tild{P} \cap Q = \nil$. A path from $\boundary_+ M$ to $\boundary_- M$ lying in $\boundary \tild{P} \cap T$ is a path in $T_2$ joining $\boundary_- C_2$ to itself. This is an impossibility since $T_2$ is trivially embedded in $C_2$. 
\end{proof}

\begin{corollary}\label{Cor: T disjoint from Q}
$T_v = T$ is disjoint from $Q$ and each component of $T_v$ is an edge.
\end{corollary}
\begin{proof}
Let $\phi$ be a path in $T_v$ joining $\boundary_- M$ to $\boundary_+ M$ which contains a vertex of $T$. If such a path exists, it is possible to choose $\phi$ so that it is contained in the boundary of $\tild{P}$, the closure of a component of $P - T$ adjacent to $\boundary_- M$. By Lemma \ref{Lem: P disjt from Q}, the edge of $\phi$ adjacent to $\boundary_- M$ is disjoint from $\boundary_1 Q$ and is, therefore, a pod handle in $C_2$. This is impossible by assumption (NPH). Thus, $T_v$ contains no vertices. 

Since component of $T_v$ is an edge, each component is contained in the closure of a component of $P-T$ adjacent to $\boundary_- M$; by Lemma \ref{Lem: P disjt from Q}, each edge of $T_v$ is disjoint from $Q$.
\end{proof}

We now put the previous results together to obtain:

\begin{corollary}
$H - T$ is properly isotopic in $M - T$ to $\boundary_+ M - T$.
\end{corollary}
\begin{proof}
Recall that, by hypothesis, $H$ is not stabilized. Since $Q \cap P = \nil$, an application of Frohman's trick shows that $Q_G$ is either empty or is a tree with at most one endpoint on $\boundary_- M$. By Corollary \ref{Cor: T disjoint from Q}, $T$ is disjoint from $Q$ and so $Q = \nil$. Also, Corollary \ref{Cor: T disjoint from Q} shows that each component of $T_v = T$ is a vertical edge. Thus, $C_1 = \eta(\boundary_- M)$. Since $M = \boundary_- M \times I$, this implies that $H - T$ is properly isotopic to $\boundary_+ M - T$ in $M - T$.
\end{proof}

We have shown, therefore, that the hypotheses of Case 1 imply that one of Conclusions (1), (5), or (6) of Theorem \ref{Main Theorem} occurs, contrary to the hypotheses of Proposition \ref{Prop: Strengthened Perturbed}. This concludes Case 1. 

\textbf{Case 2:} $R \cap C_2$ is $T_2$--$\boundary$--compressible.

This case requires the all the technology of the proofs of the main results in \cite{HS1} and \cite{HS2}.

Let $D_1, D_2, \hdots, D_n$ be the sequence of $T_2$--$\boundary$ compressions which convert $R \cap C_2$ into a $T_2$--$\boundary$--incompressible surface. These $T_2$--$\boundary$ compressions may be accomplished by isotopies of $R$. Let $R_i$ be the result of boundary-compressing $R_{i-1}$ using the disc $D_i$ beginning with $R_0 = R$. The boundary of each disc $D_i$ is the endpoint union of two arcs $\alpha_i$ and $\delta_i$ with $\alpha_i \subset R_{i-1}$ and $\delta_i \subset H$. We may isotope each disc $D_i$ so that each $\alpha_i \subset R$ and we may extend each $D_i$ so that $\delta_i$ is a, possibly non-embedded, path in $\boundary_- C_1 \cup Q$. Let $\Gamma_i$ be the graph with vertices $\boundary_1 Q$ and edges $\alpha_1, \hdots, \alpha_i$. 

Recall that $G$ is the closure of $M - R$. $G$ consists of 3--balls containing copies of $R$ in their boundary. Let $\tild{\Gamma}_i$ be the intersection of $\Gamma_i$ with $\boundary G$, so that $\tild{\Gamma}_i$ contains two copies of each edge of $\Gamma_i$. 

Suppose that $G'$ is a component of $G$ which is adjacent to $\boundary_- M$. A \defn{horizontal disc} in $G'$ is a properly embedded disc $E$ in $G'$ with the following properties:
\begin{itemize}
\item if $\tild{P}$ is a copy in $\boundary G'$ of the closure of a component of $(P \cup B) - T$ adjacent to $\boundary_- G'$ then $E \cap \tild{P}$ is a single edge joining distinct components of $\tild{P} \cap T$.
\item $D \cup A$ does not contain any edges of $\boundary E$.
\item If $\tild{B}$ is a copy in $\boundary G'$ of a component of $B - T$, then $\boundary E \cap \tild{B}$ is connected and is the union of at most two edges. If $\boundary E \cap \tild{B}$ is a single edge, it joins $T_v \cap \tild{B}$ to $T_0 \cap \tild{B}$. If it is the union of two edges, one edge joins $T_v \cap \tild{B}$ to $T_s \cap \tild{B}$ and the second edge joins $T_s \cap \tild{B}$ to $T_0 \cap \tild{B}$.
\end{itemize}

Figure \ref{Fig: Horiz Disc} depicts a horizontal disc in the case when $M = T^2 \times I$ and $T = T_v$ is a single vertical edge. If $E$ is a horizontal disc such that there exists a component $\tild{B}$ of $B - T \subset \boundary G'$ such that $E \cap \tild{B}$ is the union of two edges, then we call $E$ a \defn{crooked} horizontal disc.  Figure \ref{Fig: Crooked Horiz Disc} depicts the two possible arcs in the boundary of a crooked horizontal disc.

\begin{center}
\begin{figure}[ht]
\scalebox{0.5}{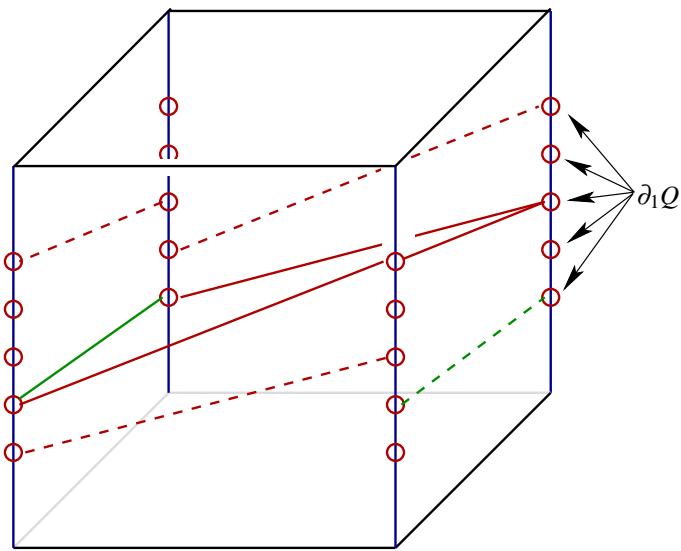}
\caption{The solid edges outline a horizontal disc in $G$ which is obtained from $M = T^2 \times I$. The graph $T = T_v$ is a single vertical edge. Since every edge of $\Gamma_i$ appears twice in $\tild{\Gamma}_i$, if the boundary of the horizontal disc lies in $\Gamma_i$, both solid and dashed edges will appear in $\tild{\Gamma}_i$. If $j$ is the smallest number such that $\tild{\Gamma}_j$ contains the cycle which is the boundary of the horizontal disc,  the green edge is $\alpha_j$ and is known as the last edge of the cycle.  }
\label{Fig: Horiz Disc}
\end{figure}
\end{center}

\begin{center}
\begin{figure}[ht]
\scalebox{0.5}{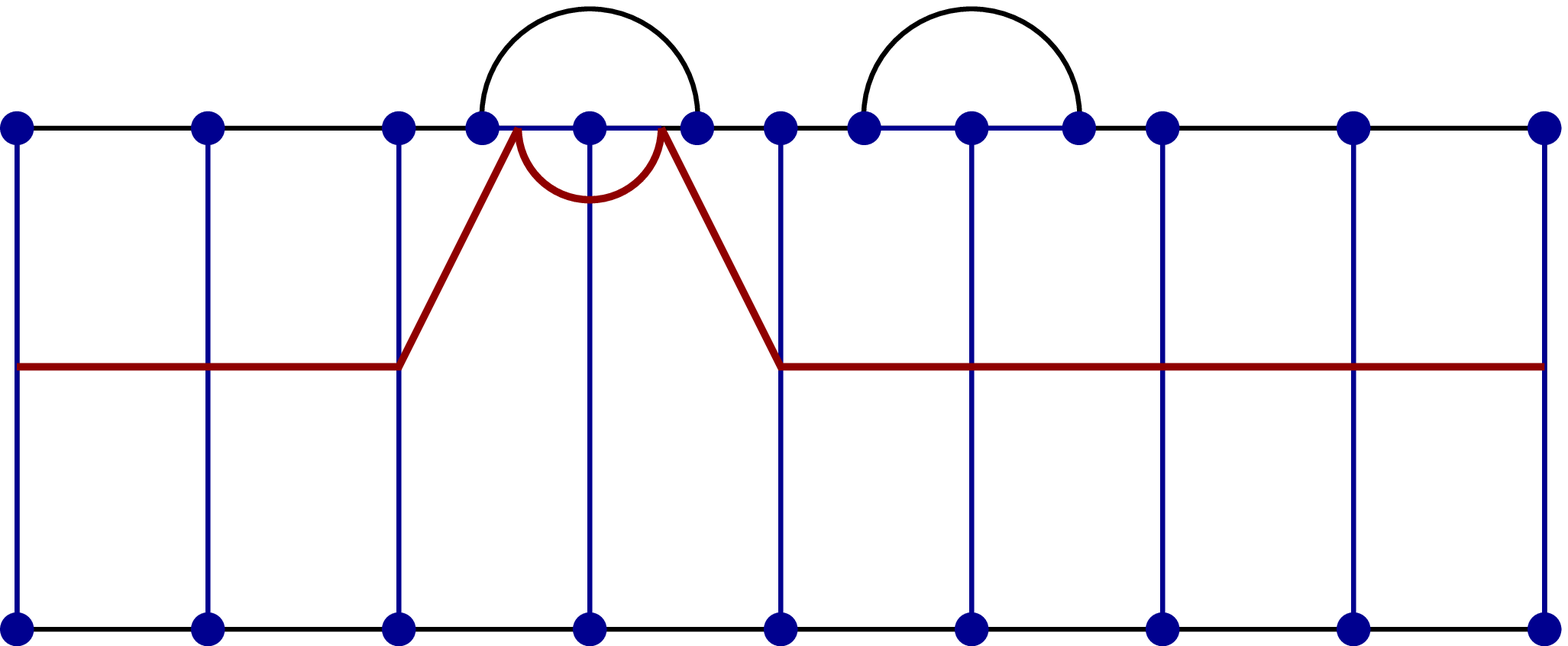}
\scalebox{0.5}{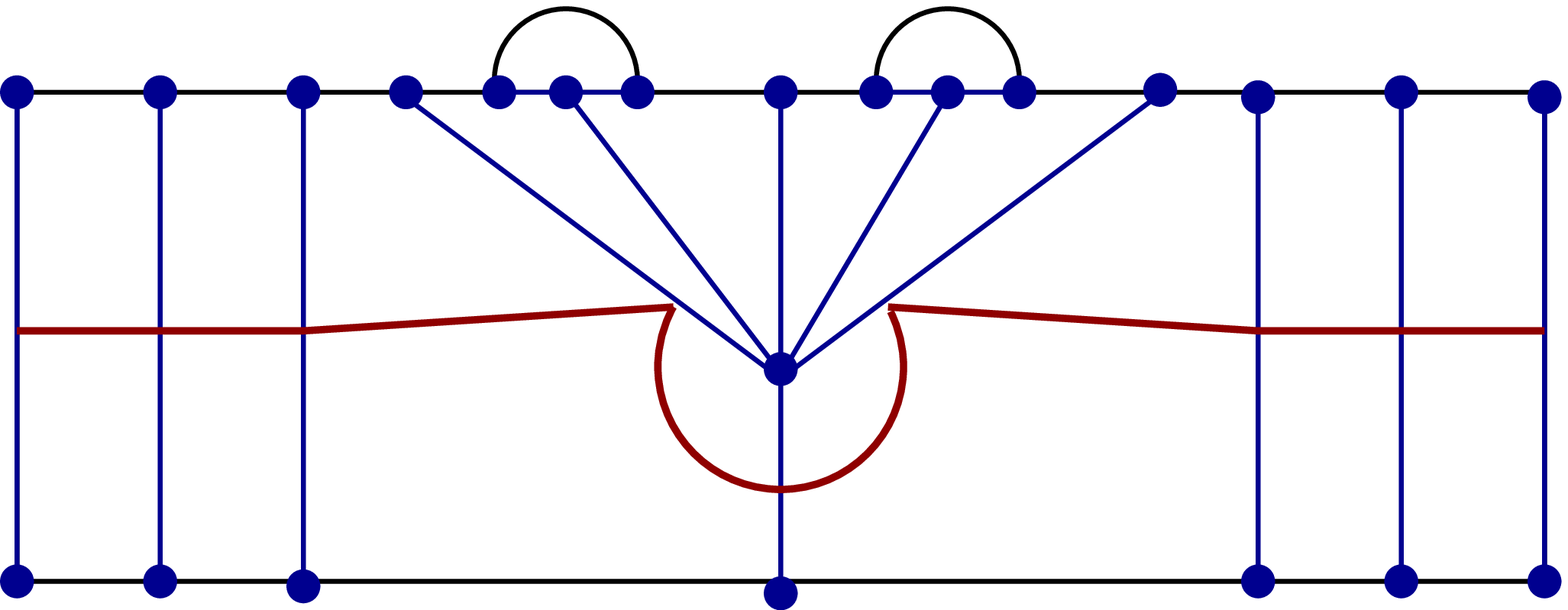}
\caption{ }
\label{Fig: Crooked Horiz Disc}
\end{figure}
\end{center}

We will soon be describing a sequence of deformations of the spine $Q$. Let $Q_i$ denote the spine just after the $i$th deformation. The deformations will be designed so that, for small enough $i$, $Q_i \cap R = \Gamma_i$. Before any deformations, $Q$ is an elementary spine and so contains no horizontal discs. The spine $Q_i$ for $i \geq 1$, may not be an elementary spine. It will, however, be the union of horizontal discs $Q^2_i$ and a graph $Q^1_i$. Let $Q_{Gi} = Q_i \cap G$, $Q^1_{Gi} = Q^1_i \cap G$, and $Q^2_{Gi} = Q^2_i \cap G$. The graph $Q^1_{Gi}$ contains two copies of each edge of $\Gamma_i$.

Define a boundary compressing disc $D_i$ to be a \defn{splendisc} (short for ``splendid disc'') if:
\begin{itemize}
\item either $i = 1$ or $D_{i-1}$ is a splendisc;
\item we can slide the edges of $Q^1_{Gi} - \Gamma_{i-1}$ and isotope $D_i$ so that $\delta_i$ does not meet an edge of $Q_{Gi}$ more than once; and
\item if after performing the previous operations, $\delta_i$ is disjoint from $Q^1_{Gi} - \Gamma_{i-1}$ then $D_i$ is a horizontal disc.
\end{itemize}

If $D_i$ is a splendisc, we always assume that the sliding and isotoping in the definition of ``splendisc'' have been performed. Let $k$ be the largest number such that $D_k$ is a splendisc.

We perform the $i$th deformation of $Q$ along $D_i$  for $i \leq k$ as follows: 
\begin{itemize}

\item When $\delta_i$ meets an edge $e$ of $Q - \tild{\Gamma}_{i-1}$, we slide the edge $e$ along the arcs $\delta_i - e$ and isotope along the disc $D_i$ to push $e$ in $R$. Then $e$ is in place of $\alpha_i$. We call this a \defn{type (i)} deformation. 

\item If $D_i$ is a horizontal disc, we extend $Q$ along the disc $D_i$. We call $\alpha_i$ the \defn{last edge} of $D_i$. This is a \defn{type (ii)} deformation. See Figure \ref{Fig: Horiz Disc}.
\end{itemize}

\begin{remark}
We now embark on a study of the graphs $\tild{\Gamma}_i$ for $i \leq k$. We will eventually conclude that all of the discs $D_i$ are splendiscs (that is, that $k = n$.) A little more work will then show that Proposition \ref{Prop: Strengthened Perturbed} holds.
\end{remark}

We begin with two definitions. 

Suppose that $\tau$ is a component of $T_s \cap D$ embedded in a disc $D'$ of $D$. Recall that $\tau$ is a tree and that $\boundary D' \cap \tau \subset \boundary \tau$ and at most one point of $\boundary \tau$ is in the interior of $D'$. Let $\alpha$ be an arc in $D'$ joining distinct points of $\tau$ with interior disjoint from $\tau$. We say that $\alpha$ joins \defn{opposite sides} of $\tau$ if the interior of the disc in $D'$ bounded by $\alpha$ and a minimal path in $\tau$ contains the valence 1 vertex of $\tau$ in the interior of $D'$. If $\alpha$ does not join opposite sides of $\tau$, we say that $\alpha$ joins \defn{the same side} of $\tau$. See Figure \ref{Fig: Arcs Joining T}.

\begin{center}
\begin{figure}[ht]
\scalebox{0.35}{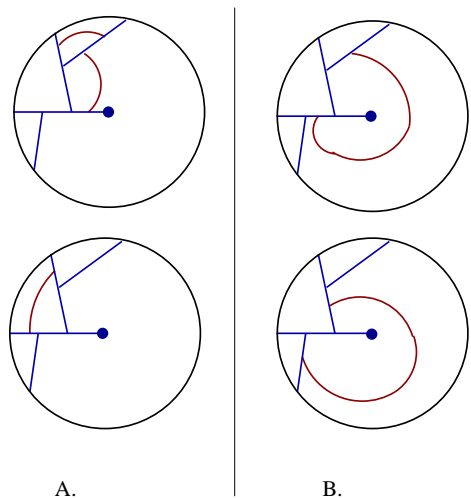}
\caption{Examples of arcs in a disc of $D$. Figure A depicts arcs joining the same side of $T$ and Figure B depicts arcs joining opposite sides of $T$.}
\label{Fig: Arcs Joining T}
\end{figure}
\end{center}

Suppose that in $R \subset \boundary G$ there is an embedded disc $E$ such that 
\begin{itemize}
\item $\boundary E \cap T$ consists of a single arc $t$ and, possibly, isolated points
\item the arc $\boundary E - t$ in $R \subset M$ does not join opposite sides of a component of $T_s \cap D$
\item $\boundary E - t \subset (\Gamma_i \cup F)$ for some component $F$ of $\boundary_- C_1$
\item the interior of $E$ is disjoint from $T \cup \Gamma_i$.
\end{itemize}
We call such a disc $E$, a \defn{container disc} in $\Gamma_i$. Figure \ref{Fig: Container Discs} shows some container discs which we will make use of later.

\begin{center}
\begin{figure}[ht]
\scalebox{0.5}{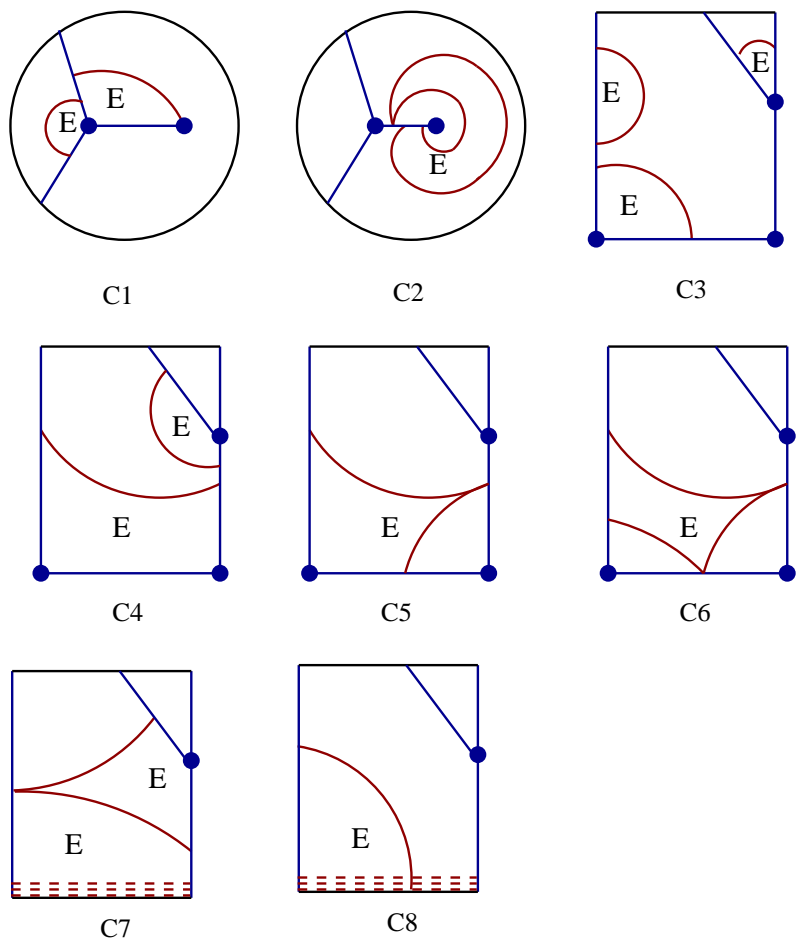}
\caption{Each example of a container disc is labelled $E$. Diagrams C1 and C2 are in $D$. Diagrams C3 - C6 are in $A$. Diagrams C7 and C8 are in $B$ or $P$. In diagrams C7 and C8, the dashed red lines mean that the component of $\boundary_- M$ is in $C_1$. This list of container discs is by no means exhaustive.}
\label{Fig: Container Discs}
\end{figure}
\end{center}

\begin{lemma}\label{Lem: Container Disc}
Suppose that, for some $i \leq k$, there exists a container disc $E$ in $\Gamma_i$. Then $E_2 = E \cap C_2$ is a bridge disc for $T_2$ and there exists a disc $E_v \subset \eta(v)$ where $v \in \boundary_1 Q$ lies on $T_0$ such that one of the following occurs:
\begin{itemize}
\item $\{E_2,E_v\}$ is a perturbing pair for $(M,T)$. The interior of the arc $\boundary E \cap T$ contains a vertex of $\mb{I} \cap T_0$.
\item $\{E_2,E_v\}$ is a cancelling pair for $(M,T)$. $T_0$ consists of a single loop, $T_s$ is connected, and $M$ is a solid torus. The disc $E$ is a subset of $A$.
\end{itemize}
\end{lemma}

\begin{proof}
Let $v$ and $w$ be the endpoints of $t$.

\textbf{Case $\alpha$:} The arc $t$ does not contain any vertices of $\boundary_1 Q$ on its interior.

In this case, the disc $E_2$ is a bridge disc for $t \cap C_2$. There are discs $E_v \subset \eta(v) \subset C_1$ and $E_w \subset \eta(w) \subset C_1$ so that the pairs $\{E,E_v\}$ and $\{E, E_w\}$ are cancelling pairs for $H$ as a splitting of $(M,T)$. See Figure \ref{Fig: Container Bridge Disc}.

\begin{center}
\begin{figure}[ht]
\scalebox{0.5}{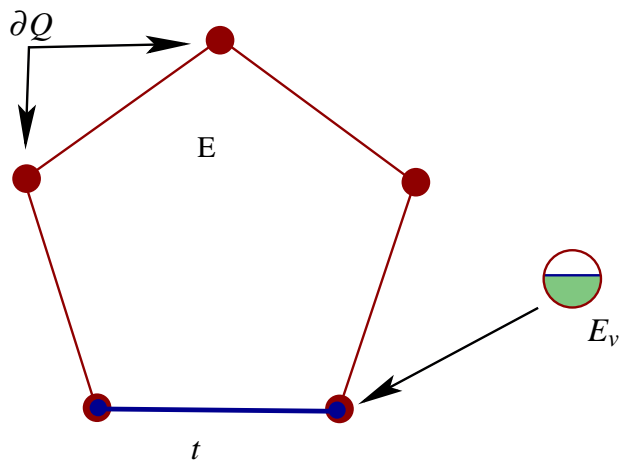}
\caption{The disc $E$ and the disc $E_v$ are a cancelling pair of discs.}
\label{Fig: Container Bridge Disc}
\end{figure}
\end{center}

Suppose first that $v$ and $w$ are identified in $M$, so that $E_v = E_w$ and $\{E_2,E_v\}$ is not a perturbing pair for $H$ as a splitting of $(M,T)$. This implies that $\boundary E - t$ is not in $D$ since then $\boundary E - t$ would join opposite sides of a component of $T_s \cap D$. Since $v$ and $w$ are in the same component of the closure of $R - T \subset \boundary G$, $\boundary E - t$ lies in $A$, and $v$ and $w$ lie on a component of $T_s \cap A$ which appears twice in the same component of $A \subset \boundary G$. By hypothesis (A) and (B) this implies that $T_0$ consists of a single loop and $T_s$ is connected. This implies that $M$ is a solid torus. 

Assume that this does not occur. Hence, $v$ is not identified with $w$ in $M$ and so both $\{E_2, E_v\}$ and $\{E_2, E_w\}$ are perturbing pairs for $H$ as a splitting of $(M,T)$.

Suppose that neither $v$ nor $w$ lie on $T_0$. Then either $t \subset T_s$, $t \subset T_v$, or an entire edge of $T_0 - T_s$ (and hence an entire edge of $T$) lies in $t$. This latter possibility cannot occur since an edge of $T$ with neither endpoint on $\boundary M$ must intersect $\boundary_1 Q$, and in Case $\alpha$, we assume that $t$ does not have any vertices of $\boundary_1 Q$ in its interior. If $t \subset T_v$, then both $v$ and $w$ lie on $T_v$ and $\{E,E_v\}$ and $\{E,E_w\}$ are perturbing pairs for $(M,\mb{T})$. Likewise, if $t \subset T_s$, then either an edge of $\mb{I}$ is perturbed, or $H$ is perturbed as a splitting of $(M,\mb{T})$. Thus, either $v$ or $w$ lies on $T_0$ and the interior of $t$ intersects the interior of some edge $e$ of $T_0$. Without loss of generality, suppose $v \in e$. If $v$ is an endpoint of $e$, choose $E_v$ so that it is a bridge disc for $T_0 \cap C_1$.

If $w$ lies on $e$, then $\{E_2,E_v\}$ is a perturbing pair for $(M,\mb{T})$, contrary to the hypothesis. Thus, $w$ lies on a component $\tau$ of $T_s$. If $\tau \not \in \mb{I}$, then $\{E_2,E_w\}$ is a perturbing pair for $(M,\mb{T})$, contrary to the hypothesis. Thus, $\tau \in \mb{I}$. Since $w \not \in e$, $t$ contains the vertex $\tau \cap T_0$ in its interior, as desired.

\textbf{Case $\beta$:} The arc $t$ contains vertices of $\boundary_1 Q$ in its interior.

This case should be compared to \cite[Lemma 4.4]{HS1} and \cite[Lemma 4.6]{HS2}.

Let $v = v_0, v_1, \hdots, v_m = w$ be the vertices of $\boundary_1 Q$ lying in order on $t$. Let $t_j$ be the arc between $v_{j-1}$ and $v_j$, for $1 \leq j \leq m$. Notice that each $t_j \cap T_2$ is either a bridge edge or a bridge arc  lying in a (possibly vertical) pod of $T_2$. For the former, choose a bridge disc $E'_j$. For the latter, choose a pod disc $E'_j$ containing $t_j$. Choose the set $\{E'_j\}$ so that the discs are pairwise disjoint and so that $\mc{E} = \bigcup_j E'_j$ intersects $E$ in properly embedded arcs joining vertices of $\boundary_1 Q$. 

Isotope $\mc{E}$ along edges of $\Gamma_i \cap \boundary E$ so that no edge of $\mc{E} \cap E$ has an endpoint at a vertex of $\boundary_1 Q \cap \boundary E$ which is not on $t$. (See Figures 4.6 and 4.7 of \cite[Lemma 4.4]{HS1}.)

Out of all such collections of discs, choose $\mc{E}$ so that the number of edges $|\mc{E} \cap E|$ is minimal.  

Suppose that there is an edge of $\mc{E} \cap E$ with both endpoints at the same vertex $v_j$. Let $\rho$ be an outermost such edge. Let $E'$ be the disc cut off from $E$ by $\rho$ which is disjoint from all vertices $v_{j'} \neq v_j$. Perform a surgery on $\mc{E}$ along the disc $E'$ and discard the component disjoint from $T_2$. This reduces the number $|\mc{E} \cap E|$, contradicting the hypothesis that the number of edges of $\mc{E} \cap E$ was minimal. Thus, no edge of $\mc{E} \cap E$ has both endpoints at the same vertex.

Define an edge of $\mc{E} \cap E$ to be \defn{good} if it joins adjacent vertices. Define a vertex $v_j$ to be \defn{good} if all edges adjacent to it are good. 

\textbf{Claim 1:} No edge of $\mc{E} \cap E$ is a good edge.

Suppose that such an edge exists and let $\mu$ be an outermost good edge joining vertices $v_{j-1}$ and $v_j$. Let $E'_l$ be the disc in $\mc{E}$ containing $\mu$. Let $\epsilon_1$ be the outermost disc in $E$ which is bounded by $\mu$. The arc $\mu$ divides $E'_l$ into two discs, one of which, $\epsilon_2$, contains $t_l$. Then a slight isotopy of $\epsilon_1 \cup \epsilon_2$ produces a bridge disc for $t_l$ which intersects $E$ in fewer arcs than does $E'_l$. Notice that the disc $E'_l$ is disjoint from all other discs in $\mc{E}$. Thus, $(\mc{E} \cup E'_l) - E'_{j'}$ is a collection of bridge discs which contradicts the choice of $\mc{E}$. \qed(Claim 1)

Claim 1 implies that a vertex is good if and only if there are no adjacent edges of $\mc{E} \cap E$.

\textbf{Claim 2:} There exists a good vertex of $\boundary_1 Q$ on $t$ which is not an endpoint of $t$. Furthermore, we can arrange that one of $v$ or $w$, whichever we wish, is also a good vertex.

If $\mc{E} \cap E$ has no edges, every vertex of $\boundary_1 Q$ on $t$ is a good vertex. If such is the case, then since $v$ and $w$ are not adjacent vertices of $\boundary_1 Q$, there exists a good vertex interior to $t$. Also, both $v$ and $w$ are good vertices. Suppose, therefore, that $\mc{E} \cap E$ contains edges.

Choose an outermost edge $\mu$ of $\mc{E} \cap E$ on $E$ and let $t'$ be the subarc of $t$ which cobounds with $\mu$ an outermost disc $E'$. Think of $E'$ as being a polygon with edges consisting of $\mu$ and those $t_j$ which lie on $t'$. Since no edge of $\mc{E} \cap E$ joins adjacent vertices, $E'$ has at least one vertex $v_j$ which is not an endpoint of $\mu$. Such a vertex is a good vertex and cannot be an endpoint of $t$. 

Suppose, now, that we also wish for $v$, say, to be a good vertex. If $e$ is an edge of $\mc{E} \cap E$ adjacent to $v$, we can push the end of $e$ on $v$ across the arc $\boundary E - t$, to be an endpoint at $w$. After moving the endpoints at $v$ of all edges of $\mc{E} \cap E$ over to $w$, $v$ is a good vertex. This maneuvre does not change the fact that no edges are adjacent to $v_j$ and so $v_j$ is still a good vertex. \qed (Claim 2).

Let $v_j$ be a good vertex which is not an endpoint of $t$. Let $D'_j$ be the disc $E \cap \eta(v_j)$. The pairs $\{E'_{j},D'_j\}$ and $\{E'_{j+1}, D'_j\}$ are both perturbing pairs for $H$ as a splitting of $(M,T)$.

Suppose that $t$ lies in a component $\tau$ of $T_s$. If $\tau \not \subset \mb{I}$, then $\{E'_j,D'_j\}$ and $\{E'_{j+1},D'_j\}$ are both perturbing pairs for $(M,\mb{T})$. Suppose, therefore, that $\tau \subset \mb{I}$. If the point $\tau \cap T_0$ lies on $t$ it must be an endpoint of $t$. Thus, $v_j$ is not the point $\tau \cap T_0$. At least one of $E'_{j}$ or $E'_{j+1}$ does not contain the point $\tau \cap T_0$ in its boundary. Suppose, without loss of generality, that it is $E'_j$. Then $\{E'_j,D'_j\}$ is a perturbing pair for the edge $\tau$. This also is a contradiction. Hence $t$ does not lie in $T_s$.

A similar argument shows that $t$ does not lie in $T_v$. Hence, the interior of $t$ is not disjoint from $T_0$. Let $e$ be the edge of $T_0 - T_s$ which has interior not disjoint from $t$. If $v_j$ does not lie on $e$, then one of $t_{j-1}$ or $t_j$ is entirely disjoint from $e$. Consequently, one of $\{E'_{j},D'_j\}$ or $\{E'_{j+1},D'_j\}$ is a perturbing pair for either an edge of $\mb{I}$ or $H$ as a splitting of $(M,\mb{T})$. (Which it is depends on whether or not the component of $T_s$ on which $t_{j}$ or $t_{j+1}$ lies is or is not in $\mb{I}$.) Both possibilities contradict our hypotheses. Thus, $v_j \in e$.

If $v_j$ is an endpoint of $e$, then one of $t_1$ or $t_m$ lies on $T_s$. Without loss of generality, suppose it is $t_1$. Then $v$ is a point on $T_s$ which is not also on $T_0$. By Claim 2, we may arrange for $v$ to be a good vertex. Let $D'_v = E \cap \eta(v)$. Then $\{E'_1,D'_v\}$ is a perturbing pair for $H$ as a splitting of $(M,T)$. Since $t_1$ is completely contained in $T_s$ and since $v$ does not lie on $T_0$, $\{E'_1,D'_v\}$ is a perturbing pair for an edge of $\mb{I}$, if the component of $T_s$ containing $t_1$ is contained in $\mb{I}$; otherwise, $\{E'_1,D'_v\}$ is a perturbing pair for $H$ as a splitting of $(M,\mb{T})$. Thus, $v_j$ is not an endpoint of $e$.

Since $t$ does not lie in $T_v \cup T_s$, $E$ must lie in either $A$ or $B$. Suppose that $v_{j-1}$ or $v_{j+1}$ lies on $T_0$. If $v_{j-1}$ lies on $T_0$, then $t_{j} \subset T_0$; and if $v_{j+1}$ lies on $T_0$, then $t_{j+1} \subset T_0$. In the former case, $\{E'_{j}, D'_j\}$ is a perturbing pair for $H$ as a splitting of $(M,\mb{T})$. In the latter case, $\{E'_{j+1}, D'_j\}$ is a perturbing pair for $H$ as a splitting of $(M,\mb{T})$. Thus, neither $t_{j-1}$ nor $t_j$ lies on $T_0$, but $v_j$ does lie on $T_0$. This implies that $E$ must lie in $A$, as every edge of $T_0 \cap B$ has an endpoint on $\boundary_- M$.

The arc $\boundary E'_j \cap T$ contains a vertex $v'$ of $T_s \cap T_0$, since $v_j$ is interior to $e$ and $v_{j+1}$ lies on $T_s$. The disc $E'_j$ is a pod disc with at least two of the edges of $T_2 \cap E'_j$ incident to $v'$ belonging to $T_0$. Let $E''_j$ be the closure of the component of $E'_j - T_2$ containing those two edges. Then $E''_j$ is a bridge disc for a component of $T_0 \cap C_2$. Thus, $\{E''_j, D'_j\}$ is a perturbing pair for $H$ as a splitting of $(M,\mb{T})$. This contradicts our hypotheses. Thus, Case $\beta$ cannot occur.
\end{proof}

We next present a sequence of lemmas, applying Lemma \ref{Lem: Container Disc} to the graphs $\tild{\Gamma}_i$ in $D$, $A$, $B$, and $P$.

\begin{lemma}\label{Lem: Configs in Disc}
For $i \leq k$, any component of the intersection between $\tild{\Gamma}_i$ and $D$ is either a single point or a single edge joining distinct points on opposite sides of $T_s \cap D$.
\end{lemma}
\begin{proof}
Suppose that the lemma is not true and let $i$ be the smallest number for which the lemma fails. If $\tild{\Gamma}_i$ contains an edge joining a point of $\boundary_1 Q$ to itself, then since $\tild{\Gamma}_i$ can be achieved by deformations of $Q$ of types (i) and (ii), by Frohman's trick, $H$ would be stabilized. Similarly, $\tild{\Gamma}_i \cap D$ does not contain a closed loop.

Consider the following two possible configurations in $\Gamma_i$. (See Figure \ref{Fig: Container Discs}.)
\begin{enumerate}
\item[(C1)] An edge of $\tild{\Gamma}_i$ in a disc $D'$ of $D$ joins two vertices on the same side of $T_s \cap D'$.
\item[(C2)] There are two edges $e_1$ and $e_2$ of $\tild{\Gamma}_i$ sharing an endpoint and joining distinct points on opposite sides of $T_s \cap D'$.
\end{enumerate}

One of these two configurations can be achieved by deformations of type (i) and (ii). In each configuration, there is an obvious container disc (labeled $E$ in the figure). By Lemma \ref{Lem: Container Disc}, the arc component of $\boundary E \cap T$ contains a vertex of $\mb{I} \cap T_0$. But in neither configuration is this the case.
\end{proof}

Similarly,

\begin{lemma}\label{Lem: Bridge disc in A 1}
Suppose that $M$ is not a solid torus with $T$ a core loop and single spoke. For $i \leq k_2$, any component of the intersection between $\tild{\Gamma}_i$ and $A$ containing an edge of $\tild{\Gamma}_i$ lies in the closure $\tild{A}$ of a component of $A - T$ which is adjacent to the interior of an edge of $T_0$ and is one of the following:
\begin{enumerate}
\item A single edge joining the edge $T_0 \cap \tild{A}$ to a component of $\mb{I} \cap \tild{A}$.
\item An edge joining distinct components of $T_s \cap \tild{A}$. At least one of the components of $T_s \cap \tild{A}$ is in $\mb{I}$.
\item The union of two edges of type (1). The edges have their endpoints on $T_0$ in common, and this point is on the interior of the edge $T_0 \cap \tild{A}$. The edges have their other endpoints on distinct components of $\mb{I} \cap \tild{A}$.
\item The union of two edges. One of the edges is of type (1) and the other is of type (2). Both components of $T_s \cap \tild{A}$ belong to $\mb{I}$.
\end{enumerate}
\end{lemma}
\begin{proof}
Suppose that the lemma is not true and let $i$ be the smallest number for which it is not true. Let $e$ be an edge of $\tild{\Gamma}_i$ which lies in the closure $\tild{A}$ of a component of $A - T$. Let $v$ and $w$ be its endpoints.

Suppose, first, that $\tild{A}$ is not adjacent to the interior of an edge of $T_0$. This could only happen if edges from $T_s$ lie in the interior of a square of $A$. In fact, the endpoints of $e$ must be contained on a component $\tau$ of $T_s$ not in $\mb{I}$. The disc $\tild{A}$ has one arc of its boundary on $\boundary_+ M$ and the complementary arc lies in $\tau$. In the disc $\tild{A}$, $e$ cuts off a subdisc $E$ which is disjoint $\boundary_+ M$. By our choice of $i$, $E$ is a container disc for $\Gamma_i$. By Lemma \ref{Lem: Container Disc}, $E \cap C_2$ is a bridge disc for a component of $T_2$ and the interior of $\boundary E \cap \tau$ contains a vertex of $\mb{I} \cap T_0$. $\tild{A}$ contains at most one vertex $x$ of $T_s \cap T_0$. Since $x \in \mb{I}$, the square of $A$ containing $\tild{A}$ cannot have an edge of $T_s$ interior to the square with endpoint $x$. Consequently, $\tild{A}$ must be adjacent to the interior of an edge of $T_0$.

We think of the disc $\tild{A}$ as being a square with two edges lying in components of $T_s$, one edge on $\boundary_+ M$, and one on $T_0$. Let $\tau_1$ and $\tau_2$ be the components of $T_s \cap \tild{A}$. The only way in which $\tau_1 = \tau_2$ is if $M$ is a solid torus, and $T$ is a core loop with a single spoke. Thus, we assume that $\tau_1 \neq \tau_2$. 

The edge $e$ cuts off a disc $E \subset \tild{A}$ with boundary disjoint from $\boundary_+ M$. If the disc $E$ is a container disc for $\Gamma_i$, by Lemma \ref{Lem: Container Disc}, $e$ is an edge of type (1). If $E$ is not a container disc, it contains other edges. If $e'$ is an edge of $\tild{\Gamma}_i \cap E$ cutting off from $E$ a disc $E'$ disjoint from $e$, the disc $E'$ is a container disc and so $e'$ is an edge of type (1). One endpoint of $e$ must be on the same component of $T_s \cap \tild{A}$ as an endpoint of $e'$. This implies that one endpoint of $e$ lies on $\mb{I}$ and that $e$ is of type (1) or (2). See (C3) and (C4) of Figure \ref{Fig: Container Discs}.

Suppose that $\tild{\Gamma}_i \cap \tild{A}$ has two edges $e_1$ and $e_2$ which share at least one endpoint. If these edges share both endpoints, then $H$ is stabilized, a contradiction. Suppose that $e_1$ and $e_2$ share one endpoint $v$ and that the other endpoints, $w_1$ and $w_2$, respectively, are distinct. By the previous paragraph, both $e_1$ and $e_2$ are of type (1) or (2).

Suppose that they are both of type (1), but that $e_1 \cup e_2$ is not of type (3). Then $v$ lies on a component of $T_s$. The edges $e_1 \cup e_2$ cut off a disc $E$ from $\tild{A}$ with $\boundary E \cap T \subset T_0$. Since $i$ is the smallest number for which the lemma fails, $E$ is a container disc. This contradicts Lemma \ref{Lem: Container Disc}. Thus, if $e_1$ and $e_2$ are both of type (1), $e_1 \cup e_2$ is of type (3). See (C5) of Figure \ref{Fig: Container Discs}.

Suppose that $e_1$, say, is of type (1) and $e_2$ is of type (2). The disc cut off from $E$ by $e_1 \cup e_2$ which has boundary containing $e_1 \cup e_2$ is either a container disc or contains an edge $e_3$ of $\tild{\Gamma}_i \cap \tild{A}$ of type (1) with an endpoint on the component of $T_s \cap \tild{A}$ containing $w_2$. Thus, both components of $T_s \cap \tild{A}$ belong to $\mb{I}$. This implies that $e_1 \cup e_2$ is of type (4).

Suppose that both $e_1$ and $e_2$ are of type (2). Then, by our choice of $i$, they cut off from $\tild{A}$ a container disc $E$ having $\boundary E \subset T_s$. This contradicts Lemma \ref{Lem: Container Disc}.

Thus, the union of two edges sharing an endpoint is either of type (3) or (4).

Finally, suppose that a component contains three edges, $e_1$, $e_2$, and $e_3$. Let $E \subset \tild{A}$ be a disc cut off by $e_1 \cup e_2 \cup e_3$ which is disjoint from $\boundary_+ M$ and which contains at least two of $\{e_1, e_2, e_3\}$. If $e_1$, $e_2$, and $e_3$ form a closed loop in $\tild{A}$, then since these are all distinct edges in $\Gamma_i$, a meridian for any of these edges intersects $\boundary E$ in a single point, showing that $H$ is stabilized. Thus, $E$ is a container disc for $\Gamma_i$. However, the arc component of $\boundary E \cap T$ lies solely in $T_0$ or $T_s$. This contradicts Lemma \ref{Lem: Container Disc}. See (C6) of Figure \ref{Fig: Container Discs}.
\end{proof}

\begin{lemma}\label{Lem: Bridge disc in A 2}
Suppose that $\tild{A} \subset A$ is in a component of $G$ adjacent to $\boundary_- M$. Suppose that $\tild{\Gamma}_i \cap \tild{A}$ has a component which is either an edge of type (1) with neither endpoint in $D$, or of type (4), with neither endpoint in $D$. Then there exist bridge discs $E_1 \subset \tild{A}$ and $E_2 \subset \eta(v)$ which form a perturbing pair for $H$ as a splitting of $(M,T)$. The vertex $v$ of $\boundary_1 Q$ lies on $T_0$. The interior of the arc $\boundary E_1 \cap T$ contains a vertex of $\mb{I} \cap T_0$.
\end{lemma}
\begin{proof}
This follows directly from Lemma \ref{Lem: Container Disc} and Lemma \ref{Lem: Bridge disc in A 1}.
\end{proof}

The proof of the next lemma is similar to the proof of Lemma \ref{Lem: Bridge disc in A 1} and so we omit it. The configurations (C7) and (C8) of Figure \ref{Fig: Container Discs} are relevant for its proof.

\begin{lemma}\label{Lem: Bridge disc in B 1}
For $i \leq k_2$, any component of the intersection between $\tild{\Gamma}_i$ and $B$ lies in the closure $\tild{B}$ of a component of $B - T$ adjacent to $\boundary_- M$ and is one of the following: 
\begin{enumerate}
\item an edge joining distinct components of $T \cap \tild{B}$
\item an edge joining the edge $\mb{I} \cap \tild{B}$ to the adjacent edge of $T_0$.
\item the union of two edges, one of each of the above types.
\end{enumerate}
\end{lemma}

Finally we turn to $P$. Again the proof is similar to the proof of Lemma \ref{Lem: Bridge disc in A 1} and so we omit it. As in the previous lemma, configurations (C7) and (C8) of Figure \ref{Fig: Container Discs} are relevant for its proof.

\begin{lemma}\label{Lem: Bridge disc in P}
For $i \leq k_2$, any component of the intersection between $\tild{\Gamma}_i$ and $P$ lies in the closure $\tild{P}$ of a component of $P - T$ adjacent to $\boundary_- M$ and is an edge joining distinct components of $T_v \cap \tild{P}$.
\end{lemma}

We assemble the previous results to produce:

\begin{lemma}\label{Lem: Cycles 1}
For $i \leq k$, if $\sigma$ is a cycle in $\tild{\Gamma}_i$ then either:
\begin{itemize}
\item $\sigma$ bounds a horizontal disc in $G$, or
\item $M$ is a solid torus, $T$ consists of a core loop $T_0$ and a single spoke $T_s$, and $\sigma$ consists of a single edge in $A$ joining the vertex of $T$ to itself.
\end{itemize}
\end{lemma}
\begin{proof}
Suppose, first, that $M$ is a solid torus, and $T$ consists of a core loop and a single spoke. Let $j$ be the smallest such number such that $\tild{\Gamma}_j$ contains a container disc. The existence of an edge in $A$ is enough to guarantee that there exists a container disc. By Lemma \ref{Lem: Configs in Disc}, the container disc $E$ is contained in $A$. By Lemma \ref{Lem: Container Disc}, $\boundary E - T$ is an edge joining the vertex of $T$ to itself. This edge appears twice in $\boundary G$, and so $j = i$ and $\sigma$ is the union of both copies of this edge. Henceforth, we assume that this does not occur.

Let $G'$ be the component of $G$ containing $\sigma$. Let $E$ be a disc in $G'$ with boundary $\sigma$. By sliding and isotoping $Q$, we may assume that $Q$ is disjoint from the interior of $E$. (See the proof of the Claim in \cite[Lemma 4.3]{HS2}.) Suppose that in $\Gamma_i$, $\sigma$ runs across an edge $e$ exactly once. Then $E$ and a meridian of the arc $e$ form a stabilizing pair for $H$. Since we are assuming that $H$ is not stabilized, $\sigma$ cannot run across any edge of $\Gamma_i$ exactly once.

Suppose that at least one edge of $\sigma$ lies in $D$. By Lemma \ref{Lem: Configs in Disc}, each such edge joins distinct points on opposite sides of $T_s \cap D$ and no two edges have endpoints in common. Choose $e_0$ to be an edge of $\sigma$ lying in a component $D'$ of $D$, chosen so that among all such edges $e_0$ is outermost. Let $v$ and $w$ be the endpoints of $e_0$. Since no component of $\tild{\Gamma}_i \cap D$ has more than one edge, there exist edges $e_{-1}$ and $e_{+1}$ in $A \cup B$ so that $v \in \boundary e_{-1}$, $w \in \boundary e_{+1}$, and $e_{-1} \cup e_0 \cup e_{+1}$ lies in $\sigma$. Since $v \neq w$ in $\Gamma_i$, the edges $e_{-1}$ and $e_{+1}$ are distinct in $\Gamma_i$. Furthermore, the components $\tild{A}_1$ and $\tild{A}_2$ of $(A \cup B) - T$ containing $e_{-1}$ and $e_{+1}$ are identified in $M$. Each component of $(A \cup B) - T$ appears exactly twice in $\boundary G$. Let $e_j \in \{e_{-1},e_{+1}\}$ be the edge which is closest to the edge $\tild{A}_1 \cap \tild{A}_2$. The cycle $\sigma$ runs exactly once across $e_j$ in $M$. Thus, $E$ and a meridian of $e_j$ form a stabilizing pair of discs for $H$, a contradiction. Thus, $D$ contains no edge of $\sigma$. 

Suppose that $A$ contains an edge of $\sigma$. Since $D$ contains no edge of $\sigma$, $\sigma$ must be contained in a component of $G'$ adjacent to $\boundary_- M$. The cycle $\sigma$ has a subpath $\rho$ contained in two adjacent components $\tild{A}$ and $\tild{A}'$  of $A - T \subset \boundary G$ which are not identified in $M$. Thus, $\rho \cap \tild{A}$ and $\rho \cap \tild{A}'$ are components of $\tild{\Gamma}_i \cap A$ satisfying the hypothesis of Lemma \ref{Lem: Bridge disc in A 2}. Let $\{E_1, E_2\}$ and $\{E'_1, E'_2\}$ be the bridge discs in $\tild{A}$ and $\tild{A}'$ respectively provided by Lemma \ref{Lem: Bridge disc in A 2}. See Figure \ref{Fig: Matching Bridge Discs}. Notice that since $E'_1 \subset \tild{A}'$, $\boundary E'_1$ is disjoint from $E_2$. Let $E = E_1 \cup E'_1$ and push $E$ slightly off $\mb{I}$, so that $E$ is a bridge disc for $T_0 \cap C_2$. Then $\{E, E_2\}$ is a perturbing pair of discs for $(M,\mb{T})$, contrary to our hypothesis.

\begin{center}
\begin{figure}[ht]
\scalebox{0.5}{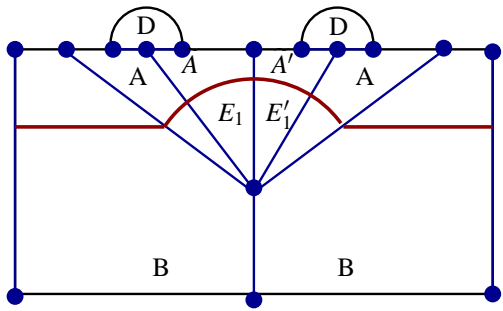}
\caption{The bridge discs $E_1$ and $E'_1$. The arc $\rho$ is in red.}
\label{Fig: Matching Bridge Discs}
\end{figure}
\end{center}

Thus, no edge of $\sigma$ is contained in $A \cup D$. Suppose that $B \subset \boundary G$ contains an edge of $\sigma$. Let $\rho \subset \sigma$ be a path in some component $\tild{B}$ of $B - T$. By Lemma \ref{Lem: Bridge disc in B 1}, $\rho$ appears as in Figure \ref{Fig: Horiz Disc} or Figure \ref{Fig: Crooked Horiz Disc}. By Lemma \ref{Lem: Bridge disc in P}, $\sigma$ bounds a horizontal disc in $G$.
\end{proof}

\begin{lemma}[{\cite[Theorems 1.1 and 3.1]{HS1}}]
If $M$ is a solid torus and $T$ is a core loop with a single spoke, then either $H$ is perturbed as a splitting of $(M,\mb{T})$ or $\mb{T}$ has a removable cycle.
\end{lemma}
\begin{proof}
Suppose not. By Lemma \ref{Lem: Container Disc}, since for some $i$ an edge of $\tild{\Gamma}_i$ is contained in $A$, there exists $i$ so that $\tild{\Gamma}_i$ contains a cycle. By Lemma \ref{Lem: Cycles 1}, that cycle is a single edge joining the vertex of $T$ to itself. Since $Q_{Gi}$ is a tree, we can shrink edges of $Q$ so that $Q$ is a single loop containing the vertex of $T$. By Lemma \ref{Lem: Container Disc}, the edge $Q_{Gi} \cap A$ cuts off a bridge disc $E$ for a component of $T_0 - H$. There is a disc $E_v$ such that $\{E,E_v\}$ is a cancelling pair of discs for $T_0$. Furthermore, a meridian of the edge $Q_{Gi} \cap A$ is disjoint from $E_v$ and intersects $\boundary E$ in a single point. Thus, $T$ has a removable cycle. Since every edge of $\mb{I}$ has an endpoint on $\boundary M$, $\mb{T}$ also has a removable cycle.
\end{proof}

Henceforth we assume that if $M$ is a solid torus, $T$ is not a core loop with a single spoke.

\begin{lemma}\label{Lem: Cycles 2}
For $i \leq k$, if $\sigma$ is a cycle in $\tild{\Gamma}_i$ then $\sigma$ bounds a horizontal disc of $Q_{Gi}$.
\end{lemma}

This proof of this lemma should be compared to \cite[Lemma 4.3]{HS2}. (Note the typographical error in the statement of that lemma: ``vertical'' should be ``horizontal''.) In particular, we refer the reader to that lemma for some of the details of the argument.

\begin{proof}
By the proof of Lemma \ref{Lem: Cycles 1}, $\sigma$ bounds a horizontal disc $E$ in $G'$ which has interior disjoint from $Q^1_G$. Since $H$ is not stabilized, $\sigma$ does not run across any edge of $\Gamma_i$ exactly once. Let $F$ be the component of $\boundary_- M$ adjacent to the component of $G$ containing $E$. If $F \neq S^2$, in $M$, $E$ glues up to be a surface $F^*$ parallel to $F$. If $F = S^2$, then in another component of $G$, there is a horizontal disc $E'$ such that $E \cup E'$ is a 2--sphere $F^*$ in $M$ parallel to $F$. Let $M'$ be the closure of the component of $M - F^*$ containing $F$. Notice that $M'$ is homeomorphic to $F \times I$. If $E$ is not a crooked horizontal disc, then $T \cap M'$ consists of vertical edges which are subsets of $T_v$. If $E$ is a crooked horizontal disc, then $T \cap M'$ is the union of vertical edges and bridge edges. Each bridge edge is a subset of $\mb{I} \subset T$.

Perform surgery on $H$, using $E$. This cuts $H$ into two surfaces. One of these surfaces $H'$ lies in $M'$ and is a Heegaard splitting for $(M',T \cap M')$. If $E$ is not a crooked horizontal disc, then by \cite{HS2}, $H'$ is either stabilized or trivial. If $E$ is a crooked horizontal disc, then by \cite{HS2}, $H'$ is either stabilized or trivial as a splitting of $(M',\mb{T} \cap M')$.

If $H'$ is stabilized as a splitting of $(M',T \cap M')$ or $(M',\mb{T} \cap M')$ then $H$ is stabilized as a splitting of $(M,T)$ or $(M,\mb{T})$. If $H'$ is trivial of type II as a splitting of $(M',T \cap M')$ or $(M',\mb{T} \cap M')$, then $H$ is stabilized as a splitting of $(M,T)$ or $(M,\mb{T})$. If $H'$ is trivial of type I as a splitting of $(M',T \cap M')$ or $(M',\mb{T} \cap M')$ then $H$ is boundary stabilized as a splitting of $(M,T)$ or $(M,\mb{T})$. Each of these possibilities contradicts our hypotheses.
\end{proof}

\begin{corollary}
$k = n$. That is, every disc $D_i$ is a splendisc.
\end{corollary}
\begin{proof}
Suppose that $k < n$ so that $D_{k + 1}$ is not a splendisc. Then $\delta_{k + 1}$ contains an arc $\sigma$ such that $\sigma$ is a cycle in $Q$. By an isotopy of $\delta_{k + 1}$, we can assume that $\sigma$ does not bound a horizontal disc or crooked horizontal disc in $Q^2_G$. Thus, by Lemma \ref{Lem: Cycles 2}, the cycle $\sigma$ must contain an edge $e$ of $Q^1_G$. Then, by the proof of \cite[Lemma 4.2]{HS1}, $H$ is stabilized.
\end{proof}

For the remainder, we assume that the $n$th deformation of $Q$ has been performed, so that $R \cap C_2$ is $T_2$--$\boundary$--incompressible. 

\begin{lemma}\label{Lem: No edges in D}
$D$ contains no edges of $\Gamma_n$.
\end{lemma}

\begin{proof}
Let $D'$ be a disc in $D$ and let $x = D' \cap T_0$. Let $\tau = T_s \cap D'$. By Lemma \ref{Lem: Configs in Disc}, each edge of $\Gamma_n$ joins distinct points on opposite sides of $\tau$. If $x$ is not of valence 1 in $\tau$, then no edge in $D$ joins opposite sides of $\tau$. Thus, we may assume that $x$ is a vertex of valence 1 in $\tau$.

\textbf{Case $\alpha$:} Suppose that $x \not \in \boundary_1 Q$.  

If $\tau$ is not disjoint from $\boundary_1 Q$, then the component of $\tau - \Gamma_n$ which contains $x$ must be a pod leg in $C_2$. This implies that the component of $\tau - \Gamma_n$ which contains $x$ is a single edge of $T_2$. Hence, the component of $D' \cap C_2$ containing $x$ does not contain an adjacent pod handle. Thus, by Proposition \ref{Prop: surface classification}, the component is a surface of type (3). This implies that there is an edge of $\Gamma_n$ in $D'$ joining one point to itself. This, however, contradicts Lemma \ref{Lem: Configs in Disc}. Thus, $\tau$ and $D'$ are disjoint from $\boundary_1 Q$. But $D'$ is an essential disc in $M$ and $Q$ is a spine for $M$ and so this is also an impossibility.

\textbf{Case $\beta$:} Suppose that $x \in \boundary_1 Q$.

If $D'$ contains an edge $e$ of $\Gamma_n$, such an edge joins distinct points of $\tau$. Choose $e$ so that no point of $\boundary_1 Q$ on $\tau$ which is an endpoint of an edge of $\Gamma_n$ in $D'$ is closer to $x$ than an endpoint of $e$. Let $q \neq x$ be a point of $\boundary_1 Q$ which can be joined by a minimal path $\rho \subset \tau$ to $x$ with interior disjoint from $\boundary_1 Q$. Let $D''$ be the closure of the component of $D' - (T \cup \Gamma_n)$ which contains $\rho$. The surface $D'' \cap C_2$ has at least one boundary component ($\boundary \eta(x)$) lying completely on $H$. The other boundary component has one arc on $T_2$ and one arc on $H$. Such a surface does not appear in the conclusion to Proposition \ref{Prop: surface classification}, and so $D'$ does not contain an edge of $\Gamma_n$.
\end{proof}

We are now in a position to conclude the proof of Proposition \ref{Prop: Strengthened Perturbed}.

By Lemma \ref{Lem: No edges in D}, $D$ contains no edges of $\Gamma_n$. Hence, before any deformations of $Q$, $D \cap C_2$ is $T_2$--$\boundary$--incompressible. If $T_0$ has an edge, $D$ contains a disc $D'$ with boundary essential on $\boundary_+ M$. By Corollary \ref{Cor: Incomp Disc implies Vert. Cut Disc}, $D'$ is a vertical cut disc. This contradicts the hypotheses of the Proposition. Thus, $T_0$ has no edges and so either $M = B^3$ or $M = \boundary_- M \times I$. If $M = B^3$, then $R = D$ and so we are not in Case 2, but rather Case 1 which has already been completed. Hence, $M = \boundary_- M \times I$, $T = T_v$ and $R = P$. 

Let $\tild{P}$ be the closure of a component of $P - T_v$. By Lemma \ref{Lem: Bridge disc in P}, each component of $\tild{P} \cap \Gamma_n$ containing an edge is a single edge joining distinct components of $T_v \cap \tild{P}$. Suppose that $e$ is such an edge, chosen so that out of all such edges it is the closest to $\boundary_- M$. Let $\tild{P}_0$ be the component of $\tild{P} - \Gamma_n$ between $e$ and $\boundary_- M$. Consider, $\tild{P}_0 \cap C_2$. If $\boundary_- M \subset C_1$, $\tild{P}_0 \cap C_2$ is a surface which we can think of as a square having two edges on $T_v$ and two edges on $H$. This surface is supposed to be $T_2$--$\boundary$--incompressible, but it does not appear in the conclusion of Proposition \ref{Prop: surface classification}. Thus, if $\boundary_- M \subset C_1$, $P$ contains no edges of $\Gamma_n$ and so before any deformations of $Q$, $P \cap C_2 = R \cap C_2$ is $T_2$--$\boundary$-incompressible. This was Case 1, which has already been completed. Consequently, we may assume that $\boundary_- M \subset C_2$.

Suppose that $\tild{P}$ has a portion of its boundary on $\boundary_+ M$. Since $\boundary_+ M \cup \boundary_- M \subset C_2$, by Proposition \ref{Prop: surface classification}, $\tild{P}$ must contain an edge of $\Gamma_n$. In fact, the argument of the previous paragraph applied to a component $\tild{P}_0$ between two adjacent edges in $\tild{P}$ shows that $\tild{P}$ contains exactly one edge of $\Gamma_n$. Examining the surfaces in the conclusion of Proposition \ref{Prop: surface classification}, shows that, in fact, $\Gamma_n \cap \tild{P}$ is that unique edge. (In other words, apart from the endpoints of the edge, there are no other vertices of $\boundary_1 Q$ on $\tild{P}$.)

We will now show that $\tild{\Gamma}_n$ is a cycle in $P$ if $\boundary_- M \neq S^2$ and two cycles in $P$ (one in each component of $P \subset \boundary G$) if $\boundary_- M = S^2$. 

Suppose that a component $\tau$ of $T_v$ has a vertex. Let $v$ be a vertex of $\tau$ adjacent to the closure $\tild{P}$ of a component of $P - T$ with boundary on $\boundary_+ M$. Since $\tild{P}$ contains a single edge of $\Gamma_n$, and since every edge of $T_2$ has at least one vertex on $\boundary C_2$, $v$ is the unique vertex of $\tau \cap \tild{P}$. Let $h$ be the edge of $T_v$ with one endpoint on $\boundary_- M$ and the other endpoint at $v$. Let $e$ be the edge $\Gamma_n \cap \tild{P}$.

By assumption (NPH), $h$ is not a pod handle of $T_2$. Thus, $h$ must contain a vertex of $\boundary_1 Q$. Since $h \subset \boundary \tild{P}$, $h$ contains an endpoint of $e$. This implies that if $P' \neq \tild{P}$ is the closure of a component of $P - T$ containing $h$, then the edge $\Gamma_n \cap P'$ and the edge $e \subset \tild{P}$ share an endpoint. Notice also that if two components of $P - T$ are adjacent along a vertical edge of $T_v$, then the edges of $\Gamma_n$ in those two components must also share an endpoint. Hence, $\Gamma_n$ consists of a cycle and, possibly, vertices of $\boundary_1 Q$ which lie in the boundaries of components of $P - T$ not adjacent to $\boundary_- M$.

We return to the consideration of the vertex $v \in \tild{P}$. Let $P'$ be the closure of a component of $P - T$ containing $v$ which is not adjacent to $\boundary_- M$. $P'$ cannot contain an edge of $\Gamma_n$ and $P' \cap C_2$ must appear in the conclusion of Proposition \ref{Prop: surface classification}. Thus, $\boundary P'$ contains a single point of $\boundary_1 Q$. Choosing $P'$ so that it is adjacent to $\tild{P}$ shows that that point must be $v$. 

Let $D' \subset P$ be a disc such that $\boundary D'$ is the union of an arc in $\boundary_+ M$ and an arc in $P$ which intersects $h$ once and is otherwise disjoint from $T$. Construct $D'$ so $T \cap D' \subset \tau$. Then $T' = \tau \cap D'$ satisfies the hypotheses of Lemma \ref{Lem: Tree Lemma} (in the appendix), with $\mc{P} = \boundary_1 Q$. Thus, by Lemma \ref{Lem: Tree Lemma}, $\boundary_1 Q \cap \tau = v$. Thus, $\tild{\Gamma}_n \cap P$ is a single cycle if $\boundary_- M \neq S^2$ and two cycles if $\boundary_- M = S^2$. 

By Lemma \ref{Lem: Cycles 2}, each cycle of $\tild{\Gamma}_n$ bounds a non-crooked horizontal disc in $G$. Since $\boundary_- M \subset C_2$, $C_1$ is the regular neighborhood of a surface parallel to $\boundary_- M$. This implies that $H$ is not connected. This is not possible since $H$ is a Heegaard surface. Thus, $M \neq \boundary_- M \times I$.
\end{proof}

\section{The proof of Theorem \ref{Main Theorem}}

\textbf{Proof of Theorem \ref{Main Theorem}} We will continue to assume that condition (NPH) is satisfied. The proof is by induction on $-\chi(\boundary_+ M)$.

\textbf{Case 1:} $M$ does not have a vertical cut disc.

If $T$ satisfies (A)-(C) in Section \ref{sec:R}, then the result follows from Proposition \ref{Prop: Strengthened Perturbed} letting $\mb{I} = \nil$ so suppose that $T$ does not satisfy at least one of (A)-(C).

Let $\mc{F}$ be the union of the components of $\boundary_- M$ which are disjoint from $T_v$. Let $\mc{V}$ be the union of vertices of $T_0$ which are disjoint from $T_s$. Let $\mc{E}$ be the union of edges of $T_0$ with zero or two endpoints on $\boundary_- M$ which are disjoint from $T_s$. 

\begin{lemma}\label{Lem: Adding Edges}
Either it is possible to form a new graph $\wihat{T}$ by adding disjoint edges to $T$ so that $(M,\wihat{T})$ satisfies the hypotheses of Theorem \ref{Main Theorem} and also assumptions (NPH) and (A) - (C), or $T$ has a monotonic interior edge.
\end{lemma}
\begin{proof}
For each component $F$ of $\mc{F}$, let $e_F$ be a vertical arc joining $F$ to $\boundary_+ M$. For each edge $\epsilon$ of $\mc{E}$ let $e_\epsilon$ be an edge joining $\epsilon$ to $\boundary_+ M$ which is vertical in the product structure on $M - \inter{\eta}(T)$. For each vertex $v$ of $\mc{V}$, let $e_v$ be an edge joining $v$ to $\boundary_+ M$ which is vertical in the product structure on $M - \inter{\eta}(T)$. Let $T'$ be the union of all the $e_F, e_\epsilon$, and $e_v$ and let $\wihat{T} = T \cup T'$. We need to show that $T'$ can be chosen so that $T'$ can be isotoped in the complement of $T$ so that $(M,\wihat{T})$ is in bridge position with respect to $H$. The isotopy is not allowed to move $T$.

The ``Proof of Theorem 1.1 given Theorem 3.1'' from \cite{HS1}, shows that this is possible for the edges $e_F$ and $e_v$ (with $F \subset \mc{F}$ and $v \subset \mc{V}$). Suppose, therefore, that $\epsilon$ is an edge in $\mc{E}$. If $\epsilon \cap C_1$ or $\epsilon \cap C_2$ contains a bridge arc, then attach $e_\epsilon$ to a bridge arc. The arc $e_\epsilon$ can then be isotoped to be in bridge position. If $\epsilon$ does not contain a bridge arc, then it must be a monotonic interior edge. All the edges of $\mc{E}$ belong to $T_0$, so the edge $\epsilon$ is an edge of $T_0$ with distinct endpoints and with neither endpoint on $\boundary M$. 

If $\wihat{T}$ is defined, it is obvious that the hypotheses of Theorem \ref{Main Theorem} and assumptions (A) - (C) hold. Assumption (NPH) also holds, since no edge of $\wihat{T} - T$ has an endpoint on $\boundary_- M$.
\end{proof}

We also need to know that adding edges does not create any new vertical cut discs.

\begin{lemma}
Let $\wihat{T}$ be the graph formed from $T$ by Lemma \ref{Lem: Adding Edges}. If $(M,\wihat{T})$ has a vertical cut disc, then so does $(M,T)$.
\end{lemma}
\begin{proof}
Suppose that $D$ is a vertical cut disc for $(M,\wihat{T})$. Let $\mb{I}$ be the edges of $\wihat{T}$ which are not subsets of $T$. If $D$ intersects an edge of $T$, then $D$ is a vertical cut disc for $(M,T)$. If $D$ intersects an edge of $\mb{I}$, then $D$ is a compressing disc for $\boundary_+ M$, intersecting $H$ in a single simple closed curve, which is disjoint from $T$. However, $T$ contains a spine $T_0$ for $M$, and so $T$ is not disjoint from any boundary compressing disc for $\boundary_+ M$. 
\end{proof}

Let $(M',T')$ be obtained from $(M,T)$ by puncturing all interior monotonic edges. It is obvious that $M'$ does not have any vertical cut discs, since $M$ had no vertical cut discs. It is equally easy to see that since $(M,T)$ satisfies (NPH), $(M',T')$ does as well. Since $T'$ has no interior monotonic edges, the graph $\wihat{T}$ from Lemma \ref{Lem: Adding Edges} is defined. Assume that $(M,\wihat{T})$ is in bridge position with respect to $H$. Assume also that no edge of $\wihat{T} - T'$ is perturbed and that $(M,\wihat{T})$ satisfies the hypotheses (A)-(C) in Section \ref{sec:R}. Applying Proposition \ref{Prop: Strengthened Perturbed} with $\mb{I} = \wihat{T} - T'$, $\wihat{T}$ in place of $T$, and $T'$ in place of $\mb{T}$, we conclude that Theorem \ref{Main Theorem} holds for $H$ as a splitting of $(M',T')$. Hence, by Lemma \ref{Lem: puncturing mie}, Theorem \ref{Main Theorem} holds for $H$ as a splitting of $(M,T)$. 

This concludes Case 1. Next we perform the inductive step.

\textbf{Case 2:} There exists a vertical cut disc for $\boundary_+ M$.

Let $\Delta$ be a vertical cut disc for $\boundary_+ M$. Let $M'$ be obtained from $M$ by boundary-reducing using $\Delta$. Let $T' = T \cap M'$ and notice that $|T' \cap \boundary_+ M'| = |T \cap \boundary_+ M| + 2$ . Let $H'$ be obtained from $H$ by compressing along $\Delta \cap C_1$ and a slight isotopy to make it properly embedded. 

\begin{lemma}
There exists a vertical cut disc $\Delta$ such that $H'$ is a Heegaard surface for $(M',T')$ and assumption (NPH) is satisfied. Furthermore the disc component of $\Delta - H$ intersects $T$.
\end{lemma}

\begin{proof}
Let $\Delta_i = \Delta \cap C_i$ for $i = 1,2$ so that $\Delta_2$ is an annulus and $\Delta_1$ is a disc. Let $C'_i = C_i \cap M'$ for $i = 1,2$. It is well-known that $C'_i$ is a (possibly disconnected) compressionbody for $i = 1,2$. 

\textbf{Claim: } We may assume that $\Delta_1 \cap T \neq \nil$.

Suppose that $\Delta_2 \cap T \neq \nil$. Let $\tau_e$ be the edge of $T_2$ which intersects $\Delta_2$. Let $\tau$ be the component of $T_2$ containing $\tau_e$.

\textbf{Case (a):} $\tau$ is a bridge edge, bridge pod, or vertical pod and $\tau_e$ is not a pod handle.

In this case, let $E$ be a bridge disc containing $\tau_e$ in its boundary. Choose $E$ so that out of all such bridge discs, $|E \cap \Delta|$ is minimal. An easy innermost disc/outermost arc argument shows that $\Delta_2 \cap E$ consists of a single arc $\gamma$ joining the puncture $\tau_e \cap \Delta_2$ to $\Delta \cap H$. Let $E'$ be a subdisc of $E$ cut off by $\gamma$. If $\tau$ is a pod, choose $E'$ so that it does not contain the vertex of the pod handle in its boundary. An isotopy of $\Delta$ across $E'$ carries $\Delta$ to a disc $\Delta'$ intersecting $H$ in a single simple closed curve. Furthermore, $\Delta'$ intersects $T$ exactly once, in a point contained in $C_1$.

\textbf{Case (b):} Either $\tau$ is a vertical edge or $\tau$ is a vertical pod, $\tau_e$ is the handle, and $\tau_e$ is adjacent to $\boundary_+ M$.

In this case, $\Delta$ is a compressing disc for $\boundary_+ M$ which is disjoint from the spine $T_0$ of $M$, an impossibility.

Thus, if $\Delta_2 \cap T \neq \nil$ then $\Delta_2$ intersects a pod handle of $T_2$ adjacent to $\boundary_- M$. This contradicts assumption (NPH). \qed(Claim)

By the claim, $\Delta_1 \cap T \neq \nil$. Let $C'_i$ be the compressionbodies into which $H'$ separates $M'$. Since $T'_2$ is disjoint from $\Delta_2$ it is not hard to see that $T'_2$ is trivially embedded in $C'_2$. We show, therefore, that $T'_1$ is trivially embedded in $C'_1$. 

Let $\mc{D}$ be a complete collection of bridge discs and pod discs in $C_1$, chosen so as to minimize $|\mc{D} \cap \Delta_1|$. It is not hard to see that if $D$ is a bridge disc in $\mc{D}$ then either $D \cap \Delta_1 = \nil$, or $D$ runs along the edge $\tau_e$ of $T_1$ intersecting $\Delta_1$. In this latter case, $D \cap \Delta_1$ consists of a single arc and $D - \Delta_1$ consists of two discs, each a bridge disc for an arc in $T'_1$. Thus, each arc in $T'_1$ with both endpoints on $H'$ has a bridge disc. A similar argument using vertical discs in $C_1$ in place of $\mc{D}$ shows that if $\tau$ is the component in $T_1$ intersecting $\Delta_1$ then $\Delta_1$ cuts $\tau$ into two components, one of which is a vertical arc and one of which is a bridge edge or bridge pod. Any vertical edge or vertical pod disjoint from $\Delta_1$ remains a vertical edge or vertical pod in $C'_1$. Hence, $T'_1$ is trivially embedded in $C'_1$.

Finally, we need to show that $(M',T')$ satisfies (NPH). Consider an edge $e$ of $T_1$ or $T_2$ adjacent to $\boundary_- M$. Such an edge is not a pod handle of $T$ since $(M,T)$ satisfies (NPH). Every non-boundary vertex of $T'$ is a vertex of $T$, and so if a component of $e_-$ of $T \cap M'$ is a pod handle, $e_-$ contains a vertex of $T$, implying that $e_-$ is a pod handle of $T'$, a contradiction. Suppose, therefore, that $e_-$ is an edge of $T_1$ or $T_2$ with an endpoint on $H$. If $e$ is disjoint from $\Delta$, then $e$ does not become a pod handle in $T'$. If $e$ intersects $\Delta$, each component of $e - H'$ has an endpoint on $H'$ and so no new pod handles are introduced.
\end{proof}

Notice that if $M'_0$ is a component of $M'$, then $(M'_0,T'_0)$ satisfies the hypotheses of Theorem \ref{Main Theorem} and $-\chi(\boundary_+ M'_0) < -\chi(\boundary_+ M)$. By the inductive hypothesis, since $-\chi(\boundary_+ M'_0) < -\chi(\boundary_+ M)$, we may assume that $(M'_0,T')$ satisfies one of Conclusions (1) - (6) of Theorem \ref{Main Theorem}. Without loss of much generality, from now on we assume that $M'$ is connected and consider $H'$ to be a splitting of $(M',T')$. $H$ can be reconstructed from $H'$ by attaching a tube to the ends of two vertical arcs $\psi_1$ and $\psi_2$ in $C'_2$.

\begin{lemma}
If $H'$ is stabilized or boundary-stabilized, then $H$ is stabilized or boundary-stabilized.
\end{lemma}
\begin{proof}
This is obvious from the construction.
\end{proof}

\begin{lemma} Either $T$ has a removable path disjoint from $\boundary_+ M$ or, if $H$ is not perturbed, it is possible to rechoose $\Delta$ so that $H'$ is not perturbed.
\end{lemma}
\begin{proof}
Let $\{D_1,D_2\}$ be perturbing discs with $D_1 \subset C'_1$ and $D_2 \subset C'_2$. Let $\psi_1$ and $\psi_2$ be the arcs in $T \cap C'_2$ which are glued together in $M$. Let $\tau_j$ be the component of $T \cap C'_1$ adjacent to $\psi_j$ for $j = 1,2$.

If $D_1$ is disjoint from $\psi_1 \cup \psi_2$ then $H$ is perturbed. Assume, therefore that $\boundary D_1 \cap H$ has one endpoint at $\psi_1 \cap H$. Notice that at least one of $\tau_1$ or $\tau_2$ must not contain a vertex of $T$ as otherwise, an interior of edge of $T$ would be disjoint from $H$, contradicting bridge position.

\textbf{Claim (a):} Suppose that $\tau_2$ is a bridge edge or pod. Furthermore, suppose that there is a bridge disc $D$ containing the edge of $\tau_2$ adjacent to $\psi_2$ such that $\boundary D \cap H$ has disjoint interior from $\boundary D_2 \cap H$. Then either $H$ is perturbed or $T$ has a removable cycle.

In this case, the union of $D$ with $D_1$ in $M$ is a bridge disc for $T_2$ which intersects $D_2$ only at points of $T \cap H$. Thus, $H$ is cancellable. In fact, if $H$ is not perturbed, then the loop which is the closure of the union of $(\boundary D \cup \boundary D_1 \cup \boundary D_2) - H$ is a removable cycle with the disc $\Delta$ playing the role of $E$ in (RP3).\qed(Claim (a)).

\textbf{Claim (b):} If $\tau_2$ is a bridge edge or pod and if $\tau_1$ is a bridge edge, then the hypotheses of Claim (a) hold.

By the proof of \cite[Lemma 3.1]{STo} (cf. \cite[Lemma 8.6]{TT}) there exists a pod disc for $\tau_2$ which is disjoint from $D_2 - T$. \qed(Claim (b))

\textbf{Claim (c):} If $\tau_2$ is a bridge edge or pod then either $H$ is perturbed, or $T$ has a removable cycle, or there exists a vertical cut disc $\Delta$ such that $H'$ is not perturbed.

Suppose that $H$ is not perturbed and that no cycle of $T$ is removable. Suppose that $E$ is a bridge disc containing $e$ in $C_2$. Isotope $E$ so that it intersects $\Delta$ minimally. Then $E \cap \Delta$ consists of a single arc. Since $H$ is not perturbed and since no cycle of $T$ is removable, the interiors of the arcs $\boundary E \cap H$ and $\boundary D_2 \cap H$ intersect. Assume that $E$ has been chosen so as to minimize the number of intersection points. The process of converting $M$ to $M'$ cuts $E$ into two bridge discs, one of which is $D_1$. Call the other one $D$. Assume that $\{D_1, D_2\}$ are a perturbing pair for $H'$. 

Isotope $\Delta$ along the edge $D \cap H$ so that $\boundary \Delta$ is moved past one intersection point of the interior of $D \cap H$ with the interior of $D_2 \cap H$. See Figure \ref{Fig: ChoppingUnperturbed}. Then after using this new $\Delta'$ to create $M'$, the pair $\{D_1,D_2\}$ is no longer a perturbing pair for $H'$. By the choice of $E$ to minimize $|\boundary E \cap \boundary D_1|$, we are done unless $\{D,D_2\}$ is now a perturbing pair. This can only happen if the closure of $(\boundary D_2 \cup \boundary E) - H$ is a cycle. Since by Claims (a) and (b), $\tau_1$ contains a vertex of $T$, $\tau_2$ does not contain a vertex. Then applying Claims (a) and (b) with $\tau_1$ and $\tau_2$ reversed completes the proof of Claim (c). \qed(Claim (c))

\begin{center}
\begin{figure}[ht]
\scalebox{0.4}{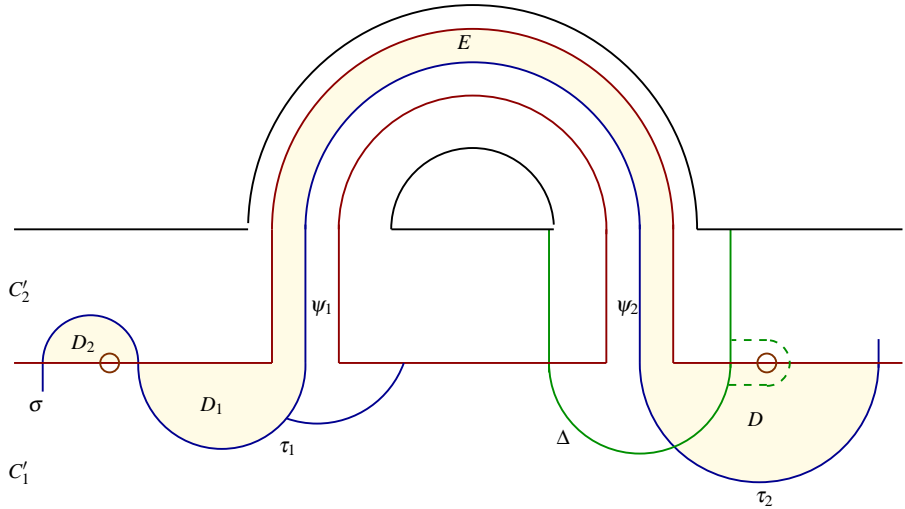}
\caption{Isotope the disc $\Delta$ so that it envelopes one of the intersections between $\boundary D \cap H$ and $\boundary D_2 \cap H$. The intersection point is the hollow circle and the disc $\Delta'$ is formed from $\Delta$ by the dashed green isotopy.}
\label{Fig: ChoppingUnperturbed}
\end{figure}
\end{center}

Thus, we may assume that $\tau_2$ is not a bridge edge or pod. That is, $\tau_2$ is a vertical edge.

\textbf{Claim (d):} Either $H$ is perturbed or the component $\sigma$ of $T_2$ adjacent to the point $\boundary D_2 \cap H \cap T - \boundary D_1$ is a vertical edge.

First, notice that $\tau_1$ cannot contain a vertex of $T$, since then $e$ would be a pod handle for $T_2$ adjacent to $\boundary_- M$, contradicting assumption (NPH). Thus, $\tau_1$ is a bridge edge. Suppose that $\sigma$ is a pod or bridge edge. Then by the proof of Claim (b), there is a bridge disc $E$ for $\sigma$ such that $\{D_2,E\}$ is a perturbing pair for $H'$. Since $E$ is disjoint from $\psi_2$, $\{D_2,E\}$ is also a perturbing pair for $H$. Thus, $\sigma$ must be a vertical edge. \qed(Claim (d))

We conclude by noticing that if $H$ is not perturbed then $\zeta = e \cup \sigma \cup (\boundary D_2 - H)$ is a removable edge. Since $\psi_1$ and $\psi_2$ lie in $C'_2$, $\sigma$ and $\tau_2$ lie in $C'_1$. In particular, the endpoints of $\zeta$ lie on $\boundary_- M$ and so $\zeta$ is disjoint from $\boundary_+ M$, as desired.
\end{proof}

\begin{lemma} If $T'$ contains a removable path, then that path is also removable in $T$.
\end{lemma}
\begin{proof}
By the definition of removable path, the path is disjoint from $\psi_1 \cup \psi_2$. Therefore, it continues to be a removable path in $T$.
\end{proof}

\begin{lemma} $(M',T')$ and $H'$ do not satisfy conclusions (5) or (6) of Theorem \ref{Main Theorem}.
\end{lemma}

\begin{proof}
\textbf{Case a: } $M'$ is a 3--ball, $T'$ is a connected graph with at most one vertex, and $H'$ is the boundary of a regular neighborhood of that vertex.

In this case, a cycle of $T$ would be disjoint from $H$, contradicting the definition of bridge position. \qed(Case a)

\textbf{Case b: } $M' = \boundary_- M' \times I$, and $H - \inter{\eta}(T)$ is isotopic in $M - \inter{\eta}(T)$ to $\boundary_+ M - \inter{\eta}(T)$.

In this case, there would be a non-backtracking path in $T$ starting and ending at $\boundary_- M$ and remaining entirely inside $C_1$. This contradicts the definition of bridge position. \qed(Case b)
\end{proof}

Combining the previous corollaries we immediately obtain:
\begin{corollary}
If $(M,T)$ contains a vertical cut disc, then $H$ and $(M,T)$ satisfy one of conclusions (1) - (4) of Theorem \ref{Main Theorem}.
\end{corollary}

\appendix\section{A Combinatorial Lemma}
\begin{lemma}\label{Lem: Tree Lemma}
Suppose that $T'$ is a finite tree containing at least one edge that is properly embedded in a disc $D'$. Suppose that $\mc{P}$ is a finite nonempty collection of points on $T' - \boundary T'$ such that the following conditions hold:
\begin{itemize}
\item If $\tild{D}$ is the closure of a component of $D' - T'$ then $\tild{D}$ contains exactly one point of $\mc{P}$ in its boundary.
\item Each edge of $T'$ without an endpoint on $\boundary D'$ contains at least one point of $\mc{P}$, possibly at one or both of its endpoints.
\item If $e_1$ and $e_2$ are edges in $T'$, each with a single endpoint on $\boundary D$ and which share the other endpoint, then $e_1 \cup e_2$ contains at least one point of $\mc{P}$.
\end{itemize}
Then $|\mc{P}| = 1$ and $T'$ is either a single edge or has a single vertex not in $\boundary T'$. Furthermore, in the latter case, the point in $\mc{P}$ is the single vertex of $T' - \boundary T'$.
\end{lemma}
\begin{proof}
We induct on $|\boundary T'|$. If $|\boundary T'|=2$, then $T'$ is a single edge and the lemma is obvious. We suppose, therefore, that $|\boundary T'| \geq 3$. We begin with a

\textbf{Key Observation:} Suppose that $p_1$ and $p_2$ are distinct points in $P$ and that $\beta \subset T'$ is a path with interior disjoint from $\mc{P}$ joining them. Then, since every edge of $T'$ with neither endpoint on $\boundary D'$ contains a point of $\mc{P}$, the interior of $\beta$ can contain at most one vertex of $\mc{P}$.

Since $|\boundary T'| \geq 3$, there are at least 3 distinct edges of $T$ with an endpoint on $\boundary D'$. Since the closure of each component of $D' - T$ contains exactly one point of $\mc{P}$, there exists an edge $\alpha$ which does not contain a point of $\mc{P}$ in its interior. Let $\boundary_0 \alpha$ be the endpoint of $\alpha$ not on $\boundary D'$.

\textbf{Claim:} The point $\boundary_0 \alpha$ is in $\mc{P}$.

Suppose not. Let $\tild{D}_1$ and $\tild{D}_2$ be the closures of the components of $D' - T'$ adjacent to $\alpha$. Since $T'$ is properly embedded in $D'$, $\tild{D}_1 \neq \tild{D}_2$. Let $p_1$ and $p_2$ be the points of $\mc{P}$ contained in $\tild{D}_1$ and $\tild{D}_2$ respectively. These are distinct points. Let $\beta$ be the path in $T'$ joining them and notice that the interior of $\beta$ is disjoint from $\mc{P}$. By the key observation, $\boundary_0 \alpha$ is the sole vertex of $T'$ contained in the interior of $\beta$. Push the interior of $\beta$ off $T$ away from $\alpha$ to create a path $\beta'$. If the interior of $\beta'$ is disjoint from $T'$ then we have contradicted the hypothesis that the closure of each component of $D' - T'$ is adjacent to exactly one point of $\mc{P}$. Thus the interior of $\beta'$ must intersect $T'$. As we travel from $p_1$ to $p_2$ along $\beta'$, let $e$ be the first edge of $T$ we encounter on the interior of $\beta'$. Notice that $e$ has an endpoint at $\boundary_0 \alpha$ but that $e \neq \alpha$. See Figure \ref{Fig: CombinatorialLemma}.

\begin{center}
\begin{figure}[ht]
\scalebox{0.5}{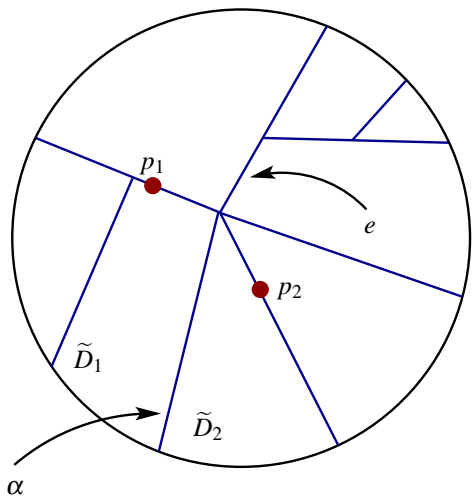}
\caption{The edge $e$.}
\label{Fig: CombinatorialLemma}
\end{figure}
\end{center}

The closure of one of the components of $D' - T'$ contains both $p_1$ and $e$. Let $\tild{D}_3$ be the closure of that component. Since $p_1$ is not in $e$ and since $\tild{D}_3$ contains exactly one point of $\mc{P}$, $e$ does not contain a point of $\mc{P}$. Thus, $e$ has one endpoint on $\boundary D'$ and the other at $\boundary_0 \alpha$. Since $\alpha$ is disjoint from $\mc{P}$, $\alpha \cup e$ contradicts our final hypothesis on $\mc{P}$. This contradiction shows that $\boundary_0 \alpha \in \mc{P}$. \qed (Claim)

Let $T''$ be the result of removing $\alpha$ from $T'$. If $\boundary_0 \alpha$ is a trivalent vertex of $T'$, it is no longer a vertex of $T''$. $T''$ is a tree with $|\boundary T''| < |\boundary T'|$ since one endpoint of $\alpha$ is on $\boundary D'$. By the claim, $T''$ and $\mc{P}$ satisfy the hypotheses of the lemma. Hence, by induction, $|\mc{P}| = 1$ and $T''$ is either a single edge or $T''$ has a single interior vertex and that vertex is the sole point in $\mc{P}$. Since $\boundary_0 \alpha \in P$, reattaching $\alpha$ to produce $T'$ shows that $T'$ also satisfies the conclusions of the lemma. 
\end{proof}

\end{document}

%% file: TrivialGraph.eps_t
\begin{picture}(0,0)%
\includegraphics{TrivialGraph.eps}%
\end{picture}%
\setlength{\unitlength}{3947sp}%
\begingroup\makeatletter\ifx\SetFigFont\undefined%
\gdef\SetFigFont#1#2#3#4#5{%
  \reset@font\fontsize{#1}{#2pt}%
  \fontfamily{#3}\fontseries{#4}\fontshape{#5}%
  \selectfont}%
\fi\endgroup%
\begin{picture}(9838,2646)(601,-3919)
\put(8551,-1561){\makebox(0,0)[lb]{\smash{{\SetFigFont{20}{24.0}{\familydefault}{\mddefault}{\updefault}{\color[rgb]{0,0,0}$\boundary_+(C)$}%
}}}}
\put(7351,-3811){\makebox(0,0)[lb]{\smash{{\SetFigFont{20}{24.0}{\familydefault}{\mddefault}{\updefault}{\color[rgb]{0,0,0}$\boundary_- C$}%
}}}}
\put(3076,-3586){\makebox(0,0)[lb]{\smash{{\SetFigFont{20}{24.0}{\familydefault}{\mddefault}{\updefault}{\color[rgb]{0,0,0}bridge edge}%
}}}}
\put(601,-3586){\makebox(0,0)[lb]{\smash{{\SetFigFont{20}{24.0}{\familydefault}{\mddefault}{\updefault}{\color[rgb]{0,0,0}vertical edge}%
}}}}
\put(8701,-3136){\makebox(0,0)[lb]{\smash{{\SetFigFont{20}{24.0}{\familydefault}{\mddefault}{\updefault}{\color[rgb]{0,0,0}vertical pod}%
}}}}
\put(8701,-2386){\makebox(0,0)[lb]{\smash{{\SetFigFont{20}{24.0}{\familydefault}{\mddefault}{\updefault}{\color[rgb]{0,0,0}bridge pod}%
}}}}
\end{picture}%

%% file: Spine.eps_t
\begin{picture}(0,0)%
\includegraphics{Spine.eps}%
\end{picture}%
\setlength{\unitlength}{3947sp}%
\begingroup\makeatletter\ifx\SetFigFont\undefined%
\gdef\SetFigFont#1#2#3#4#5{%
  \reset@font\fontsize{#1}{#2pt}%
  \fontfamily{#3}\fontseries{#4}\fontshape{#5}%
  \selectfont}%
\fi\endgroup%
\begin{picture}(7186,3046)(720,-3684)
\end{picture}%

%% file: Amalgamation.eps_t
\begin{picture}(0,0)%
\includegraphics{Amalgamation.eps}%
\end{picture}%
\setlength{\unitlength}{3947sp}%
\begingroup\makeatletter\ifx\SetFigFont\undefined%
\gdef\SetFigFont#1#2#3#4#5{%
  \reset@font\fontsize{#1}{#2pt}%
  \fontfamily{#3}\fontseries{#4}\fontshape{#5}%
  \selectfont}%
\fi\endgroup%
\begin{picture}(15996,6651)(1153,-6394)
\put(7861,-3856){\makebox(0,0)[lb]{\smash{{\SetFigFont{20}{24.0}{\familydefault}{\mddefault}{\updefault}{\color[rgb]{0,0,0}Amalgamation}%
}}}}
\end{picture}%

%% file: GraphT.eps_t
\begin{picture}(0,0)%
\includegraphics{GraphT.eps}%
\end{picture}%
\setlength{\unitlength}{3947sp}%
\begingroup\makeatletter\ifx\SetFigFont\undefined%
\gdef\SetFigFont#1#2#3#4#5{%
  \reset@font\fontsize{#1}{#2pt}%
  \fontfamily{#3}\fontseries{#4}\fontshape{#5}%
  \selectfont}%
\fi\endgroup%
\begin{picture}(6366,3216)(2218,-3994)
\end{picture}%

%% file: CbdyMWithGraph.eps_t
\begin{picture}(0,0)%
\includegraphics{CbdyMWithGraph.eps}%
\end{picture}%
\setlength{\unitlength}{3947sp}%
\begingroup\makeatletter\ifx\SetFigFont\undefined%
\gdef\SetFigFont#1#2#3#4#5{%
  \reset@font\fontsize{#1}{#2pt}%
  \fontfamily{#3}\fontseries{#4}\fontshape{#5}%
  \selectfont}%
\fi\endgroup%
\begin{picture}(8217,6372)(1855,-5209)
\end{picture}%

%% file: CBdyWithGraph2.eps_t
\begin{picture}(0,0)%
\includegraphics{CBdyWithGraph2.eps}%
\end{picture}%
\setlength{\unitlength}{3947sp}%
\begingroup\makeatletter\ifx\SetFigFont\undefined%
\gdef\SetFigFont#1#2#3#4#5{%
  \reset@font\fontsize{#1}{#2pt}%
  \fontfamily{#3}\fontseries{#4}\fontshape{#5}%
  \selectfont}%
\fi\endgroup%
\begin{picture}(8217,6372)(1855,-5209)
\end{picture}%

%% file: ResolvingPodHandles.eps_t
\begin{picture}(0,0)%
\includegraphics{ResolvingPodHandles.eps}%
\end{picture}%
\setlength{\unitlength}{3947sp}%
\begingroup\makeatletter\ifx\SetFigFont\undefined%
\gdef\SetFigFont#1#2#3#4#5{%
  \reset@font\fontsize{#1}{#2pt}%
  \fontfamily{#3}\fontseries{#4}\fontshape{#5}%
  \selectfont}%
\fi\endgroup%
\begin{picture}(8916,1716)(1168,-3244)
\end{picture}%

%% file: PodHandlesBdyStab.eps_t
\begin{picture}(0,0)%
\includegraphics{PodHandlesBdyStab.eps}%
\end{picture}%
\setlength{\unitlength}{3947sp}%
\begingroup\makeatletter\ifx\SetFigFont\undefined%
\gdef\SetFigFont#1#2#3#4#5{%
  \reset@font\fontsize{#1}{#2pt}%
  \fontfamily{#3}\fontseries{#4}\fontshape{#5}%
  \selectfont}%
\fi\endgroup%
\begin{picture}(10146,3456)(763,-3289)
\put(3338,-2513){\makebox(0,0)[lb]{\smash{{\SetFigFont{20}{24.0}{\familydefault}{\mddefault}{\updefault}{\color[rgb]{0,0,0}$\psi$}%
}}}}
\put(2274,-2671){\makebox(0,0)[lb]{\smash{{\SetFigFont{20}{24.0}{\familydefault}{\mddefault}{\updefault}{\color[rgb]{0,0,0}$V$}%
}}}}
\end{picture}%

%% file: PodHandlesRemovablePath.eps_t
\begin{picture}(0,0)%
\includegraphics{PodHandlesRemovablePath.eps}%
\end{picture}%
\setlength{\unitlength}{3947sp}%
\begingroup\makeatletter\ifx\SetFigFont\undefined%
\gdef\SetFigFont#1#2#3#4#5{%
  \reset@font\fontsize{#1}{#2pt}%
  \fontfamily{#3}\fontseries{#4}\fontshape{#5}%
  \selectfont}%
\fi\endgroup%
\begin{picture}(6421,4593)(1188,-4252)
\put(3316,-76){\makebox(0,0)[lb]{\smash{{\SetFigFont{20}{24.0}{\familydefault}{\mddefault}{\updefault}{\color[rgb]{0,0,0}$e$}%
}}}}
\put(4951,-811){\makebox(0,0)[lb]{\smash{{\SetFigFont{20}{24.0}{\familydefault}{\mddefault}{\updefault}{\color[rgb]{0,0,0}$Es$}%
}}}}
\put(4036,-3526){\makebox(0,0)[lb]{\smash{{\SetFigFont{20}{24.0}{\familydefault}{\mddefault}{\updefault}{\color[rgb]{0,0,0}$h_\lambda$}%
}}}}
\put(4659, 36){\makebox(0,0)[lb]{\smash{{\SetFigFont{20}{24.0}{\familydefault}{\mddefault}{\updefault}{\color[rgb]{0,0,0}$D_1$}%
}}}}
\put(5476,-3586){\makebox(0,0)[lb]{\smash{{\SetFigFont{20}{24.0}{\familydefault}{\mddefault}{\updefault}{\color[rgb]{0,0,0}$h_\lambda'$}%
}}}}
\end{picture}%

%% file: EssentialSurfCBdy.eps_t
\begin{picture}(0,0)%
\includegraphics{EssentialSurfCBdy.eps}%
\end{picture}%
\setlength{\unitlength}{3947sp}%
\begingroup\makeatletter\ifx\SetFigFont\undefined%
\gdef\SetFigFont#1#2#3#4#5{%
  \reset@font\fontsize{#1}{#2pt}%
  \fontfamily{#3}\fontseries{#4}\fontshape{#5}%
  \selectfont}%
\fi\endgroup%
\begin{picture}(8466,3098)(868,-3994)
\end{picture}%

%% file: DiscsD.eps_t
\begin{picture}(0,0)%
\includegraphics{DiscsD.eps}%
\end{picture}%
\setlength{\unitlength}{3947sp}%
\begingroup\makeatletter\ifx\SetFigFont\undefined%
\gdef\SetFigFont#1#2#3#4#5{%
  \reset@font\fontsize{#1}{#2pt}%
  \fontfamily{#3}\fontseries{#4}\fontshape{#5}%
  \selectfont}%
\fi\endgroup%
\begin{picture}(5901,1438)(1588,-1447)
\end{picture}%

%% file: DiscE.eps_t
\begin{picture}(0,0)%
\includegraphics{DiscE.eps}%
\end{picture}%
\setlength{\unitlength}{3947sp}%
\begingroup\makeatletter\ifx\SetFigFont\undefined%
\gdef\SetFigFont#1#2#3#4#5{%
  \reset@font\fontsize{#1}{#2pt}%
  \fontfamily{#3}\fontseries{#4}\fontshape{#5}%
  \selectfont}%
\fi\endgroup%
\begin{picture}(7251,5877)(778,-3559)
\put(7036,359){\makebox(0,0)[lb]{\smash{{\SetFigFont{20}{24.0}{\familydefault}{\mddefault}{\updefault}{\color[rgb]{0,0,0}$E(\tau)$}%
}}}}
\put(4426,-1831){\makebox(0,0)[lb]{\smash{{\SetFigFont{20}{24.0}{\familydefault}{\mddefault}{\updefault}{\color[rgb]{0,0,0}$\tau$}%
}}}}
\put(5026,-2671){\makebox(0,0)[lb]{\smash{{\SetFigFont{20}{24.0}{\familydefault}{\mddefault}{\updefault}{\color[rgb]{0,0,0}$h$}%
}}}}
\put(1171,-1846){\makebox(0,0)[lb]{\smash{{\SetFigFont{20}{24.0}{\familydefault}{\mddefault}{\updefault}{\color[rgb]{0,0,0}$T_s$}%
}}}}
\put(3901,1979){\makebox(0,0)[lb]{\smash{{\SetFigFont{20}{24.0}{\familydefault}{\mddefault}{\updefault}{\color[rgb]{0,0,0}$T_s$}%
}}}}
\end{picture}%

%% file: SurfaceR.eps_t
\begin{picture}(0,0)%
\includegraphics{SurfaceR.eps}%
\end{picture}%
\setlength{\unitlength}{3947sp}%
\begingroup\makeatletter\ifx\SetFigFont\undefined%
\gdef\SetFigFont#1#2#3#4#5{%
  \reset@font\fontsize{#1}{#2pt}%
  \fontfamily{#3}\fontseries{#4}\fontshape{#5}%
  \selectfont}%
\fi\endgroup%
\begin{picture}(6366,3216)(2218,-3994)
\end{picture}%

%% file: G1.eps_t
\begin{picture}(0,0)%
\includegraphics{G1.eps}%
\end{picture}%
\setlength{\unitlength}{3947sp}%
\begingroup\makeatletter\ifx\SetFigFont\undefined%
\gdef\SetFigFont#1#2#3#4#5{%
  \reset@font\fontsize{#1}{#2pt}%
  \fontfamily{#3}\fontseries{#4}\fontshape{#5}%
  \selectfont}%
\fi\endgroup%
\begin{picture}(8383,5610)(1861,-5059)
\put(9691,-2596){\makebox(0,0)[lb]{\smash{{\SetFigFont{20}{24.0}{\familydefault}{\mddefault}{\updefault}{\color[rgb]{0,0,0}$p$}%
}}}}
\put(6346,299){\makebox(0,0)[lb]{\smash{{\SetFigFont{20}{24.0}{\familydefault}{\mddefault}{\updefault}{\color[rgb]{0,0,0}$D$}%
}}}}
\put(4861,-1306){\makebox(0,0)[lb]{\smash{{\SetFigFont{20}{24.0}{\familydefault}{\mddefault}{\updefault}{\color[rgb]{0,0,0}$A'$}%
}}}}
\put(5971,-2596){\makebox(0,0)[lb]{\smash{{\SetFigFont{20}{24.0}{\familydefault}{\mddefault}{\updefault}{\color[rgb]{0,0,0}$T$}%
}}}}
\put(7366,-3916){\makebox(0,0)[lb]{\smash{{\SetFigFont{20}{24.0}{\familydefault}{\mddefault}{\updefault}{\color[rgb]{0,0,0}$p$}%
}}}}
\put(4231,-2611){\makebox(0,0)[lb]{\smash{{\SetFigFont{20}{24.0}{\familydefault}{\mddefault}{\updefault}{\color[rgb]{0,0,0}$p$}%
}}}}
\put(1861,-3976){\makebox(0,0)[lb]{\smash{{\SetFigFont{20}{24.0}{\familydefault}{\mddefault}{\updefault}{\color[rgb]{0,0,0}$p$}%
}}}}
\end{picture}%

%% file: G2.eps_t
\begin{picture}(0,0)%
\includegraphics{G2.eps}%
\end{picture}%
\setlength{\unitlength}{3947sp}%
\begingroup\makeatletter\ifx\SetFigFont\undefined%
\gdef\SetFigFont#1#2#3#4#5{%
  \reset@font\fontsize{#1}{#2pt}%
  \fontfamily{#3}\fontseries{#4}\fontshape{#5}%
  \selectfont}%
\fi\endgroup%
\begin{picture}(7296,5181)(2353,-5059)
\end{picture}%

%% file: G3.eps_t
\begin{picture}(0,0)%
\includegraphics{G3.eps}%
\end{picture}%
\setlength{\unitlength}{3947sp}%
\begingroup\makeatletter\ifx\SetFigFont\undefined%
\gdef\SetFigFont#1#2#3#4#5{%
  \reset@font\fontsize{#1}{#2pt}%
  \fontfamily{#3}\fontseries{#4}\fontshape{#5}%
  \selectfont}%
\fi\endgroup%
\begin{picture}(7623,5181)(2026,-5059)
\put(2026,-1381){\makebox(0,0)[lb]{\smash{{\SetFigFont{20}{24.0}{\familydefault}{\mddefault}{\updefault}{\color[rgb]{0,0,0}$p_+$}%
}}}}
\end{picture}%

%% file: SurfaceB1.eps_t
\begin{picture}(0,0)%
\includegraphics{SurfaceB1.eps}%
\end{picture}%
\setlength{\unitlength}{3947sp}%
\begingroup\makeatletter\ifx\SetFigFont\undefined%
\gdef\SetFigFont#1#2#3#4#5{%
  \reset@font\fontsize{#1}{#2pt}%
  \fontfamily{#3}\fontseries{#4}\fontshape{#5}%
  \selectfont}%
\fi\endgroup%
\begin{picture}(6726,6228)(1753,-4300)
\end{picture}%

%% file: OpenedR1.eps_t
\begin{picture}(0,0)%
\includegraphics{OpenedR1.eps}%
\end{picture}%
\setlength{\unitlength}{3947sp}%
\begingroup\makeatletter\ifx\SetFigFont\undefined%
\gdef\SetFigFont#1#2#3#4#5{%
  \reset@font\fontsize{#1}{#2pt}%
  \fontfamily{#3}\fontseries{#4}\fontshape{#5}%
  \selectfont}%
\fi\endgroup%
\begin{picture}(9966,3730)(1018,-4066)
\put(1651,-3586){\makebox(0,0)[lb]{\smash{{\SetFigFont{20}{24.0}{\familydefault}{\mddefault}{\updefault}{\color[rgb]{0,0,0}P}%
}}}}
\put(2926,-3586){\makebox(0,0)[lb]{\smash{{\SetFigFont{20}{24.0}{\familydefault}{\mddefault}{\updefault}{\color[rgb]{0,0,0}P}%
}}}}
\put(8926,-3586){\makebox(0,0)[lb]{\smash{{\SetFigFont{20}{24.0}{\familydefault}{\mddefault}{\updefault}{\color[rgb]{0,0,0}P}%
}}}}
\put(10126,-3586){\makebox(0,0)[lb]{\smash{{\SetFigFont{20}{24.0}{\familydefault}{\mddefault}{\updefault}{\color[rgb]{0,0,0}P}%
}}}}
\put(4051,-3586){\makebox(0,0)[lb]{\smash{{\SetFigFont{20}{24.0}{\familydefault}{\mddefault}{\updefault}{\color[rgb]{0,0,0}B}%
}}}}
\put(5251,-3586){\makebox(0,0)[lb]{\smash{{\SetFigFont{20}{24.0}{\familydefault}{\mddefault}{\updefault}{\color[rgb]{0,0,0}B}%
}}}}
\put(6526,-3586){\makebox(0,0)[lb]{\smash{{\SetFigFont{20}{24.0}{\familydefault}{\mddefault}{\updefault}{\color[rgb]{0,0,0}B}%
}}}}
\put(7726,-3586){\makebox(0,0)[lb]{\smash{{\SetFigFont{20}{24.0}{\familydefault}{\mddefault}{\updefault}{\color[rgb]{0,0,0}B}%
}}}}
\put(4651,-736){\makebox(0,0)[lb]{\smash{{\SetFigFont{20}{24.0}{\familydefault}{\mddefault}{\updefault}{\color[rgb]{0,0,0}D}%
}}}}
\put(7051,-736){\makebox(0,0)[lb]{\smash{{\SetFigFont{20}{24.0}{\familydefault}{\mddefault}{\updefault}{\color[rgb]{0,0,0}D}%
}}}}
\end{picture}%

%% file: OpenedR2.eps_t
\begin{picture}(0,0)%
\includegraphics{OpenedR2.eps}%
\end{picture}%
\setlength{\unitlength}{3947sp}%
\begingroup\makeatletter\ifx\SetFigFont\undefined%
\gdef\SetFigFont#1#2#3#4#5{%
  \reset@font\fontsize{#1}{#2pt}%
  \fontfamily{#3}\fontseries{#4}\fontshape{#5}%
  \selectfont}%
\fi\endgroup%
\begin{picture}(9891,3804)(1048,-4066)
\put(1651,-3661){\makebox(0,0)[lb]{\smash{{\SetFigFont{20}{24.0}{\familydefault}{\mddefault}{\updefault}{\color[rgb]{0,0,0}P}%
}}}}
\put(7051,-3661){\makebox(0,0)[lb]{\smash{{\SetFigFont{20}{24.0}{\familydefault}{\mddefault}{\updefault}{\color[rgb]{0,0,0}B}%
}}}}
\put(4726,-3661){\makebox(0,0)[lb]{\smash{{\SetFigFont{20}{24.0}{\familydefault}{\mddefault}{\updefault}{\color[rgb]{0,0,0}B}%
}}}}
\put(2551,-3661){\makebox(0,0)[lb]{\smash{{\SetFigFont{20}{24.0}{\familydefault}{\mddefault}{\updefault}{\color[rgb]{0,0,0}P}%
}}}}
\put(9226,-3586){\makebox(0,0)[lb]{\smash{{\SetFigFont{20}{24.0}{\familydefault}{\mddefault}{\updefault}{\color[rgb]{0,0,0}P}%
}}}}
\put(10276,-3586){\makebox(0,0)[lb]{\smash{{\SetFigFont{20}{24.0}{\familydefault}{\mddefault}{\updefault}{\color[rgb]{0,0,0}P}%
}}}}
\put(4411,-1321){\makebox(0,0)[lb]{\smash{{\SetFigFont{20}{24.0}{\familydefault}{\mddefault}{\updefault}{\color[rgb]{0,0,0}A}%
}}}}
\put(4531,-616){\makebox(0,0)[lb]{\smash{{\SetFigFont{20}{24.0}{\familydefault}{\mddefault}{\updefault}{\color[rgb]{0,0,0}D}%
}}}}
\put(6946,-631){\makebox(0,0)[lb]{\smash{{\SetFigFont{20}{24.0}{\familydefault}{\mddefault}{\updefault}{\color[rgb]{0,0,0}D}%
}}}}
\put(5431,-1366){\makebox(0,0)[lb]{\smash{{\SetFigFont{20}{24.0}{\familydefault}{\mddefault}{\updefault}{\color[rgb]{0,0,0}A}%
}}}}
\put(6211,-1366){\makebox(0,0)[lb]{\smash{{\SetFigFont{20}{24.0}{\familydefault}{\mddefault}{\updefault}{\color[rgb]{0,0,0}A}%
}}}}
\put(7096,-1351){\makebox(0,0)[lb]{\smash{{\SetFigFont{20}{24.0}{\familydefault}{\mddefault}{\updefault}{\color[rgb]{0,0,0}A}%
}}}}
\end{picture}%

%% file: HorizDisc.eps_t
\begin{picture}(0,0)%
\includegraphics{HorizDisc.eps}%
\end{picture}%
\setlength{\unitlength}{3947sp}%
\begingroup\makeatletter\ifx\SetFigFont\undefined%
\gdef\SetFigFont#1#2#3#4#5{%
  \reset@font\fontsize{#1}{#2pt}%
  \fontfamily{#3}\fontseries{#4}\fontshape{#5}%
  \selectfont}%
\fi\endgroup%
\begin{picture}(7959,6846)(1072,-6394)
\put(9031,-2041){\makebox(0,0)[lb]{\smash{{\SetFigFont{20}{24.0}{\familydefault}{\mddefault}{\updefault}{\color[rgb]{0,0,0}$\boundary_1 Q$}%
}}}}
\end{picture}%

%% file: CrookedHorizDisc1.eps_t
\begin{picture}(0,0)%
\includegraphics{CrookedHorizDisc1.eps}%
\end{picture}%
\setlength{\unitlength}{3947sp}%
\begingroup\makeatletter\ifx\SetFigFont\undefined%
\gdef\SetFigFont#1#2#3#4#5{%
  \reset@font\fontsize{#1}{#2pt}%
  \fontfamily{#3}\fontseries{#4}\fontshape{#5}%
  \selectfont}%
\fi\endgroup%
\begin{picture}(9812,4029)(1095,-4066)
\put(7051,-586){\makebox(0,0)[lb]{\smash{{\SetFigFont{20}{24.0}{\familydefault}{\mddefault}{\updefault}{\color[rgb]{0,0,0}D}%
}}}}
\put(6526,-3586){\makebox(0,0)[lb]{\smash{{\SetFigFont{20}{24.0}{\familydefault}{\mddefault}{\updefault}{\color[rgb]{0,0,0}B}%
}}}}
\put(7726,-3586){\makebox(0,0)[lb]{\smash{{\SetFigFont{20}{24.0}{\familydefault}{\mddefault}{\updefault}{\color[rgb]{0,0,0}B}%
}}}}
\put(4651,-586){\makebox(0,0)[lb]{\smash{{\SetFigFont{20}{24.0}{\familydefault}{\mddefault}{\updefault}{\color[rgb]{0,0,0}D}%
}}}}
\put(1651,-3586){\makebox(0,0)[lb]{\smash{{\SetFigFont{20}{24.0}{\familydefault}{\mddefault}{\updefault}{\color[rgb]{0,0,0}P}%
}}}}
\put(2926,-3586){\makebox(0,0)[lb]{\smash{{\SetFigFont{20}{24.0}{\familydefault}{\mddefault}{\updefault}{\color[rgb]{0,0,0}P}%
}}}}
\put(8926,-3586){\makebox(0,0)[lb]{\smash{{\SetFigFont{20}{24.0}{\familydefault}{\mddefault}{\updefault}{\color[rgb]{0,0,0}P}%
}}}}
\put(10126,-3586){\makebox(0,0)[lb]{\smash{{\SetFigFont{20}{24.0}{\familydefault}{\mddefault}{\updefault}{\color[rgb]{0,0,0}P}%
}}}}
\put(4051,-3586){\makebox(0,0)[lb]{\smash{{\SetFigFont{20}{24.0}{\familydefault}{\mddefault}{\updefault}{\color[rgb]{0,0,0}B}%
}}}}
\put(5251,-3586){\makebox(0,0)[lb]{\smash{{\SetFigFont{20}{24.0}{\familydefault}{\mddefault}{\updefault}{\color[rgb]{0,0,0}B}%
}}}}
\end{picture}%

%% file: CrookedHorizDisc2.eps_t
\begin{picture}(0,0)%
\includegraphics{CrookedHorizDisc2.eps}%
\end{picture}%
\setlength{\unitlength}{3947sp}%
\begingroup\makeatletter\ifx\SetFigFont\undefined%
\gdef\SetFigFont#1#2#3#4#5{%
  \reset@font\fontsize{#1}{#2pt}%
  \fontfamily{#3}\fontseries{#4}\fontshape{#5}%
  \selectfont}%
\fi\endgroup%
\begin{picture}(9812,3804)(1095,-4066)
\put(7096,-1351){\makebox(0,0)[lb]{\smash{{\SetFigFont{20}{24.0}{\familydefault}{\mddefault}{\updefault}{\color[rgb]{0,0,0}A}%
}}}}
\put(6946,-631){\makebox(0,0)[lb]{\smash{{\SetFigFont{20}{24.0}{\familydefault}{\mddefault}{\updefault}{\color[rgb]{0,0,0}D}%
}}}}
\put(5431,-1366){\makebox(0,0)[lb]{\smash{{\SetFigFont{20}{24.0}{\familydefault}{\mddefault}{\updefault}{\color[rgb]{0,0,0}A}%
}}}}
\put(6211,-1366){\makebox(0,0)[lb]{\smash{{\SetFigFont{20}{24.0}{\familydefault}{\mddefault}{\updefault}{\color[rgb]{0,0,0}A}%
}}}}
\put(1651,-3661){\makebox(0,0)[lb]{\smash{{\SetFigFont{20}{24.0}{\familydefault}{\mddefault}{\updefault}{\color[rgb]{0,0,0}P}%
}}}}
\put(7051,-3661){\makebox(0,0)[lb]{\smash{{\SetFigFont{20}{24.0}{\familydefault}{\mddefault}{\updefault}{\color[rgb]{0,0,0}B}%
}}}}
\put(4726,-3661){\makebox(0,0)[lb]{\smash{{\SetFigFont{20}{24.0}{\familydefault}{\mddefault}{\updefault}{\color[rgb]{0,0,0}B}%
}}}}
\put(2551,-3661){\makebox(0,0)[lb]{\smash{{\SetFigFont{20}{24.0}{\familydefault}{\mddefault}{\updefault}{\color[rgb]{0,0,0}P}%
}}}}
\put(9226,-3586){\makebox(0,0)[lb]{\smash{{\SetFigFont{20}{24.0}{\familydefault}{\mddefault}{\updefault}{\color[rgb]{0,0,0}P}%
}}}}
\put(10276,-3586){\makebox(0,0)[lb]{\smash{{\SetFigFont{20}{24.0}{\familydefault}{\mddefault}{\updefault}{\color[rgb]{0,0,0}P}%
}}}}
\put(4411,-1321){\makebox(0,0)[lb]{\smash{{\SetFigFont{20}{24.0}{\familydefault}{\mddefault}{\updefault}{\color[rgb]{0,0,0}A}%
}}}}
\put(4531,-616){\makebox(0,0)[lb]{\smash{{\SetFigFont{20}{24.0}{\familydefault}{\mddefault}{\updefault}{\color[rgb]{0,0,0}D}%
}}}}
\end{picture}%

%% file: ArcsJoiningT.eps_t
\begin{picture}(0,0)%
\includegraphics{ArcsJoiningT.eps}%
\end{picture}%
\setlength{\unitlength}{3947sp}%
\begingroup\makeatletter\ifx\SetFigFont\undefined%
\gdef\SetFigFont#1#2#3#4#5{%
  \reset@font\fontsize{#1}{#2pt}%
  \fontfamily{#3}\fontseries{#4}\fontshape{#5}%
  \selectfont}%
\fi\endgroup%
\begin{picture}(8291,8802)(1018,-8026)
\end{picture}%

%% file: ContainerDiscs.eps_t
\begin{picture}(0,0)%
\includegraphics{ContainerDiscs.eps}%
\end{picture}%
\setlength{\unitlength}{3947sp}%
\begingroup\makeatletter\ifx\SetFigFont\undefined%
\gdef\SetFigFont#1#2#3#4#5{%
  \reset@font\fontsize{#1}{#2pt}%
  \fontfamily{#3}\fontseries{#4}\fontshape{#5}%
  \selectfont}%
\fi\endgroup%
\begin{picture}(9889,11658)(343,-10861)
\put(1531,-2896){\makebox(0,0)[lb]{\smash{{\SetFigFont{20}{24.0}{\familydefault}{\mddefault}{\updefault}{\color[rgb]{0,0,0}C1}%
}}}}
\put(5026,-6886){\makebox(0,0)[lb]{\smash{{\SetFigFont{20}{24.0}{\familydefault}{\mddefault}{\updefault}{\color[rgb]{0,0,0}C5}%
}}}}
\put(1501,-10861){\makebox(0,0)[lb]{\smash{{\SetFigFont{20}{24.0}{\familydefault}{\mddefault}{\updefault}{\color[rgb]{0,0,0}C7}%
}}}}
\put(1651,-6886){\makebox(0,0)[lb]{\smash{{\SetFigFont{20}{24.0}{\familydefault}{\mddefault}{\updefault}{\color[rgb]{0,0,0}C4}%
}}}}
\put(8776,-6886){\makebox(0,0)[lb]{\smash{{\SetFigFont{20}{24.0}{\familydefault}{\mddefault}{\updefault}{\color[rgb]{0,0,0}C6}%
}}}}
\put(5176,-2866){\makebox(0,0)[lb]{\smash{{\SetFigFont{20}{24.0}{\familydefault}{\mddefault}{\updefault}{\color[rgb]{0,0,0}C2}%
}}}}
\put(8731,-2821){\makebox(0,0)[lb]{\smash{{\SetFigFont{20}{24.0}{\familydefault}{\mddefault}{\updefault}{\color[rgb]{0,0,0}C3}%
}}}}
\put(5101,-10786){\makebox(0,0)[lb]{\smash{{\SetFigFont{20}{24.0}{\familydefault}{\mddefault}{\updefault}{\color[rgb]{0,0,0}C8}%
}}}}
\end{picture}%

%% file: ContainerBridgeDisc.eps_t
\begin{picture}(0,0)%
\includegraphics{ContainerBridgeDisc.eps}%
\end{picture}%
\setlength{\unitlength}{3947sp}%
\begingroup\makeatletter\ifx\SetFigFont\undefined%
\gdef\SetFigFont#1#2#3#4#5{%
  \reset@font\fontsize{#1}{#2pt}%
  \fontfamily{#3}\fontseries{#4}\fontshape{#5}%
  \selectfont}%
\fi\endgroup%
\begin{picture}(8244,5707)(2011,-8780)
\put(9286,-7201){\makebox(0,0)[lb]{\smash{{\SetFigFont{25}{30.0}{\rmdefault}{\mddefault}{\updefault}{\color[rgb]{0,0,0}$E_v$}%
}}}}
\put(2011,-3361){\makebox(0,0)[lb]{\smash{{\SetFigFont{25}{30.0}{\rmdefault}{\mddefault}{\updefault}{\color[rgb]{0,0,0}$\boundary Q$}%
}}}}
\put(4306,-8686){\makebox(0,0)[lb]{\smash{{\SetFigFont{25}{30.0}{\rmdefault}{\mddefault}{\updefault}{\color[rgb]{0,0,0}$t$}%
}}}}
\end{picture}%

%% file: MatchingBridgeDiscs.eps_t
\begin{picture}(0,0)%
\includegraphics{MatchingBridgeDiscs.eps}%
\end{picture}%
\setlength{\unitlength}{3947sp}%
\begingroup\makeatletter\ifx\SetFigFont\undefined%
\gdef\SetFigFont#1#2#3#4#5{%
  \reset@font\fontsize{#1}{#2pt}%
  \fontfamily{#3}\fontseries{#4}\fontshape{#5}%
  \selectfont}%
\fi\endgroup%
\begin{picture}(6212,3804)(2895,-4066)
\put(6151,-1711){\makebox(0,0)[lb]{\smash{{\SetFigFont{20}{24.0}{\familydefault}{\mddefault}{\updefault}{\color[rgb]{0,0,0}$E_1'$}%
}}}}
\put(7276,-1261){\makebox(0,0)[lb]{\smash{{\SetFigFont{20}{24.0}{\familydefault}{\mddefault}{\updefault}{\color[rgb]{0,0,0}A}%
}}}}
\put(5101,-1111){\makebox(0,0)[lb]{\smash{{\SetFigFont{20}{24.0}{\familydefault}{\mddefault}{\updefault}{\color[rgb]{0,0,0}$\tild{A}$}%
}}}}
\put(5551,-1711){\makebox(0,0)[lb]{\smash{{\SetFigFont{20}{24.0}{\familydefault}{\mddefault}{\updefault}{\color[rgb]{0,0,0}$E_1$}%
}}}}
\put(7051,-3661){\makebox(0,0)[lb]{\smash{{\SetFigFont{20}{24.0}{\familydefault}{\mddefault}{\updefault}{\color[rgb]{0,0,0}B}%
}}}}
\put(4726,-3661){\makebox(0,0)[lb]{\smash{{\SetFigFont{20}{24.0}{\familydefault}{\mddefault}{\updefault}{\color[rgb]{0,0,0}B}%
}}}}
\put(4531,-616){\makebox(0,0)[lb]{\smash{{\SetFigFont{20}{24.0}{\familydefault}{\mddefault}{\updefault}{\color[rgb]{0,0,0}D}%
}}}}
\put(6946,-631){\makebox(0,0)[lb]{\smash{{\SetFigFont{20}{24.0}{\familydefault}{\mddefault}{\updefault}{\color[rgb]{0,0,0}D}%
}}}}
\put(6226,-1111){\makebox(0,0)[lb]{\smash{{\SetFigFont{20}{24.0}{\familydefault}{\mddefault}{\updefault}{\color[rgb]{0,0,0}$\tild{A}'$}%
}}}}
\put(4426,-1261){\makebox(0,0)[lb]{\smash{{\SetFigFont{20}{24.0}{\familydefault}{\mddefault}{\updefault}{\color[rgb]{0,0,0}A}%
}}}}
\end{picture}%

%% file: ChoppingUnperturbed.eps_t
\begin{picture}(0,0)%
\includegraphics{ChoppingUnperturbed.eps}%
\end{picture}%
\setlength{\unitlength}{3947sp}%
\begingroup\makeatletter\ifx\SetFigFont\undefined%
\gdef\SetFigFont#1#2#3#4#5{%
  \reset@font\fontsize{#1}{#2pt}%
  \fontfamily{#3}\fontseries{#4}\fontshape{#5}%
  \selectfont}%
\fi\endgroup%
\begin{picture}(14058,7946)(-1274,-5959)
\put(2956,-5056){\makebox(0,0)[lb]{\smash{{\SetFigFont{20}{24.0}{\familydefault}{\mddefault}{\updefault}{\color[rgb]{0,0,0}$\tau_1$}%
}}}}
\put(10441,-5851){\makebox(0,0)[lb]{\smash{{\SetFigFont{20}{24.0}{\familydefault}{\mddefault}{\updefault}{\color[rgb]{0,0,0}$\tau_2$}%
}}}}
\put(7321,-4936){\makebox(0,0)[lb]{\smash{{\SetFigFont{20}{24.0}{\familydefault}{\mddefault}{\updefault}{\color[rgb]{0,0,0}$\Delta$}%
}}}}
\put(5761,1274){\makebox(0,0)[lb]{\smash{{\SetFigFont{20}{24.0}{\familydefault}{\mddefault}{\updefault}{\color[rgb]{0,0,0}$E$}%
}}}}
\put(10321,-4651){\makebox(0,0)[lb]{\smash{{\SetFigFont{20}{24.0}{\familydefault}{\mddefault}{\updefault}{\color[rgb]{0,0,0}$D$}%
}}}}
\put(-299,-3436){\makebox(0,0)[lb]{\smash{{\SetFigFont{20}{24.0}{\familydefault}{\mddefault}{\updefault}{\color[rgb]{0,0,0}$D_2$}%
}}}}
\put(1726,-4411){\makebox(0,0)[lb]{\smash{{\SetFigFont{20}{24.0}{\familydefault}{\mddefault}{\updefault}{\color[rgb]{0,0,0}$D_1$}%
}}}}
\put(-1274,-2611){\makebox(0,0)[lb]{\smash{{\SetFigFont{20}{24.0}{\familydefault}{\mddefault}{\updefault}{\color[rgb]{0,0,0}$C'_2$}%
}}}}
\put(-1274,-5461){\makebox(0,0)[lb]{\smash{{\SetFigFont{20}{24.0}{\familydefault}{\mddefault}{\updefault}{\color[rgb]{0,0,0}$C'_1$}%
}}}}
\put(3451,-2836){\makebox(0,0)[lb]{\smash{{\SetFigFont{20}{24.0}{\familydefault}{\mddefault}{\updefault}{\color[rgb]{0,0,0}$\psi_1$}%
}}}}
\put(8251,-2836){\makebox(0,0)[lb]{\smash{{\SetFigFont{20}{24.0}{\familydefault}{\mddefault}{\updefault}{\color[rgb]{0,0,0}$\psi_2$}%
}}}}
\put(-974,-4411){\makebox(0,0)[lb]{\smash{{\SetFigFont{20}{24.0}{\familydefault}{\mddefault}{\updefault}{\color[rgb]{0,0,0}$\sigma$}%
}}}}
\end{picture}%

%% file: CombinatorialLemma.eps_t
\begin{picture}(0,0)%
\includegraphics{CombinatorialLemma.eps}%
\end{picture}%
\setlength{\unitlength}{3947sp}%
\begingroup\makeatletter\ifx\SetFigFont\undefined%
\gdef\SetFigFont#1#2#3#4#5{%
  \reset@font\fontsize{#1}{#2pt}%
  \fontfamily{#3}\fontseries{#4}\fontshape{#5}%
  \selectfont}%
\fi\endgroup%
\begin{picture}(5853,6216)(2461,-5644)
\put(2461,-5536){\makebox(0,0)[lb]{\smash{{\SetFigFont{20}{24.0}{\familydefault}{\mddefault}{\updefault}{\color[rgb]{0,0,0}$\alpha$}%
}}}}
\put(4816,-4816){\makebox(0,0)[lb]{\smash{{\SetFigFont{20}{24.0}{\familydefault}{\mddefault}{\updefault}{\color[rgb]{0,0,0}$\tild{D}_2$}%
}}}}
\put(3316,-3961){\makebox(0,0)[lb]{\smash{{\SetFigFont{20}{24.0}{\familydefault}{\mddefault}{\updefault}{\color[rgb]{0,0,0}$\tild{D}_1$}%
}}}}
\put(6961,-2266){\makebox(0,0)[lb]{\smash{{\SetFigFont{20}{24.0}{\familydefault}{\mddefault}{\updefault}{\color[rgb]{0,0,0}$e$}%
}}}}
\put(5866,-3031){\makebox(0,0)[lb]{\smash{{\SetFigFont{20}{24.0}{\familydefault}{\mddefault}{\updefault}{\color[rgb]{0,0,0}$p_2$}%
}}}}
\put(4141,-1456){\makebox(0,0)[lb]{\smash{{\SetFigFont{20}{24.0}{\familydefault}{\mddefault}{\updefault}{\color[rgb]{0,0,0}$p_1$}%
}}}}
\end{picture}%

%% file: FinalClassification.bbl
\begin{thebibliography}{99}

\bibitem[CG]{CG} {Casson and Gordon} `Reducing Heegaard splittings'.
{\it Topology Appl. } 27  (1987),  no. 3, 275--283.

\bibitem[F]{F} {Frohman} `The topological uniqueness of triply periodic minimal surfaces in
 $R\sp 3$'. {\it J. Differential Geom.}  31  (1990),  no. 1, 277--283.

\bibitem[G]{G} {Gabai} `Foliations and the topology of $3$-manifolds. III'.
{\it J. Differential Geom.}  26  (1987),  no. 3, 479--536.

\bibitem[HS1]{HS1} {Hayashi and Shimokawa} `Heegaard splittings of the pair of the solid torus and the core loop'. {\it Rev. Mat. Complut.}  14  (2001),  no. 2, 479--501.

\bibitem[HS2]{HS2} {Hayashi and Shimokawa} `Heegaard splittings of trivial arcs in compressionbodies'. {\it J. Knot Theory Ramifications}  10  (2001),  no. 1, 71--87.

\bibitem[HS3]{HS3} {Hayashi and Shimokawa} `Thin position of a pair (3-manifold, 1-submanifold)'. {\it Pacific J. Math.}  197  (2001),  no. 2, 301--324.

\bibitem[GST]{GST}{Goda, Scharlemann, and Thompson} `Levelling an unknotting tunnel'.{\it Geom. Topol.}  4  (2000), 243--275 (electronic).

\bibitem[L]{L}{Li} `Thin position and planar surfaces for graphs in the 3-sphere'. {\tt arxiv: 0807.2865}

\bibitem[S]{S}{Scharlemann} `Thin position in the theory of classical knots'.
{\it Handbook of knot theory}, 429--459, Elsevier B. V., Amsterdam,  2005. 

\bibitem[ST1]{ST}{Scharlemann and Thompson} `Heegaard splittings of $({\rm surface})\times I$ are standard'. {\it Math. Ann.}  295  (1993),  no. 3, 549--564.

\bibitem[ST2]{ST2}{Scharlemann and Thompson} `Thin position and Heegaard splittings of the $3$-sphere'. {\it J. Differential Geom.}  39  (1994),  no. 2, 343--357.

\bibitem[ST3]{ST3}{Scharlemann and Thompson} `Thin position for $3$-manifolds'.
{\it Geometric topology (Haifa, 1992)}, 231--238, Contemp. Math., 164, Amer. Math. Soc., Providence, RI,  1994.

\bibitem[STo]{STo}{Scharlemann and Tomova} `Uniqueness of bridge surfaces for 2-bridge knots'. {\it Math. Proc. Cambridge Philos. Soc.}  144  (2008),  no. 3, 639--650.

\bibitem[Sc]{Sc}{Schultens} `The classification of Heegaard splittings for (compact orientable surface)$\,\times\, S\sp 1$'. {\it Proc. London Math. Soc.} (3)  67  (1993),  no. 2, 425--448.

\bibitem[TT]{TT}{Taylor and Tomova} `Essential surfaces in (3-manifold, graph) pairs and leveling edges of Heegaard spines'. {\tt arxiv: 0910.3251}

\bibitem[T]{T}{Tomova} `Thin position for knots in a 3-manifold'.
 {\it J. Lond. Math. Soc.} (2)  80  (2009),  no. 1, 85--98.


\end{thebibliography}
